\documentclass[reqno,11pt]{amsart}
\usepackage{amsthm,amsfonts,amssymb,euscript,mathrsfs,graphics,color,amsmath,latexsym,marginnote}

\usepackage[dvips]{graphicx}

\usepackage{hyperref}

\usepackage{graphicx}

\numberwithin{equation}{section}

\def\C{{\mathbb{C}}}

\def\oe{\overline{\varep}}

\def\ze{\zeta}

\def\varep{\varepsilon}

\def\al{\alpha}

\def\div{{\,\mathrm{div}\,}}
\def\curl{{\,\mbox{curl}\,}}

\def\wt{\widetilde}
\def\bar{\overline}
\def\R{{\mathbb R}}
\def\e{{\varepsilon}}

\def\KK{{\mathcal K}}

\def\KK{{\mathcal K}}

\def\R{{\bf R}}

\def\Z{{\mathbb Z}}

\def\X{{\mathbb X}}

\def\bar{\overline}
\def\what{\widehat}
\def\R{\mathbb{R}}

\def\T{{\mathbb T}}

\newtheorem{theorem}{Theorem}[section]
\newtheorem{lemma}[theorem]{Lemma}
\newtheorem{proposition}[theorem]{Proposition}

\newtheorem{definition}[theorem]{Definition}
\newtheorem{remark}[theorem]{Remark}

\begin{document}

\title[Periodic water waves in 3D]{Long-time existence for multi-dimensional \\ periodic water waves}
\markboth{A. Ionescu and F. Pusateri}{Periodic water waves}


\author{A. D. Ionescu}
\address{Princeton University}
\email{aionescu@math.princeton.edu}

\author{F. Pusateri}
\address{University of Toronto}
\email{fabiop@math.toronto.edu}

\thanks{A. D. Ionescu was supported in part by NSF grant DMS-1600028 and by NSF-FRG grant DMS-1463753.
F. Pusateri was supported in part by Start-up grants from Princeton University and the University of Toronto, and 
NSERC grant RGPIN-0648}

\begin{abstract}
We prove an extended lifespan result for the full gravity-capillary water waves system with a $2$ dimensional periodic interface:
for initial data of sufficiently small size $\e$, smooth solutions exist up to times of the order of $\e^{-5/3+}$,
for almost all values of the gravity and surface tension parameters.

Besides the quasilinear nature of the equations, the main difficulty 
is to handle the weak small divisors bounds for quadratic and cubic interactions, growing with the size of the largest frequency.
To overcome this difficulty we use (1) the (Hamiltonian)
structure of the equations which gives additional smoothing close to the resonant hypersurfaces, 
(2) another structural property, connected to time-reversibility, that allows us to handle ``trivial'' cubic resonances, 
(3) sharp small divisors lower bounds on three and four-waves modulation functions based on counting arguments,
and (4) partial normal form transformations and symmetrization arguments in the Fourier space.
Our theorem appears to be the first extended lifespan result for quasilinear 
equations with non-trivial resonances on a multi-dimensional torus.

\end{abstract}

\maketitle

\setcounter{tocdepth}{1}

\tableofcontents

\section{Introduction}

\subsection{Free boundary Euler equations and water waves} The evolution of an inviscid perfect fluid that occupies a domain $\Omega_t \subset \R^n$, for $n \geq 2$, at time $t$,
is described by the free boundary incompressible Euler equations.
If $v$ and $p$ denote the velocity and the pressure of the fluid (with constant density equal to $1$)
at time $t$ and position $x \in \Omega_t$, these equations are
\begin{equation}\label{E}
(\partial_t + v \cdot \nabla) v = - \nabla p - g e_n,\qquad \nabla \cdot v = 0,\qquad x \in \Omega_t,
\end{equation}
where $g$ is the gravitational constant. The first equation in \eqref{E} is the conservation of momentum equation and  the second equation is the incompressibility condition.
The free surface $S_t := \partial \Omega_t$ moves with the normal component of the velocity according to the kinematic boun\-dary condition
 \begin{equation}\label{BC1}
\partial_t + v \cdot \nabla  \,\, \mbox{is tangent to} \,\, {\bigcup}_t S_t \subset \R^{n+1}_{x,t}.
\end{equation}
The pressure on the interface is given by
\begin{equation}
\label{BC2}
p (x,t) = \sigma \kappa(x,t),  \qquad x \in S_t,
\end{equation}
where $\kappa$ is the mean-curvature of $S_t$ and $\sigma \geq 0$ is the surface tension coefficient.
At liquid-air interfaces, the surface tension force results from the greater attraction
of water molecules to each other than to the molecules in the air.

One can consider the free boundary Euler equations \eqref{E}-\eqref{BC2} in various types of domains $\Omega_t$ (bounded, periodic, unbounded)
and study flows with different characteristics (with or without vorticity, with gravity and/or surface tension),
or even more complicated scenarios where the moving interface separates two fluids.
In this paper we consider the case of irrotational flows, i.e. $\rm{\curl} v = 0$, with infinite bottom, in the periodic setting. 

In this presentation we consider the equations in both the Euclidean and periodic settings. 
In the case of irrotational flows one can reduce \eqref{E}-\eqref{BC2} to a system on the boundary.
Indeed, assume that $\X=\T$ (periodic setting) or $\X=\R$ (Euclidean setting) and $\Omega_t \subset \X^{n-1}\times \R$ is the region below the graph of a function $h : \X^{n-1}_x \times I_t \rightarrow \R$,
\begin{align*}
\Omega_t = \{ (x,y) \in \X^{n-1} \times \R \, : y \leq h(x,t) \} \quad \mbox{and} \quad S_t = \{ (x,y) : y = h(x,t) \}.
\end{align*}
Let $\Phi$ denote the velocity potential, $\nabla_{x,y} \Phi(x,y,t) = v (x,y,t)$ for $(x,y) \in \Omega_t$, which vanishes as $y\to -\infty$.
If $\phi(x,t) := \Phi (x, h(x,t),t)$ is the restriction of $\Phi$ to the boundary $S_t$,
the equations of motion reduce to the following system for the unknowns $h, \phi : \X^{n-1}\times I_t \rightarrow \R$:
\begin{equation}
\label{WWE}
\left\{
\begin{array}{l}
\partial_t h = G(h) \phi,
\\
\partial_t \phi = -g h  + \sigma \div \Big[ \dfrac{\nabla h}{ (1+|\nabla h|^2)^{1/2} } \Big]
  - \dfrac{1}{2} {|\nabla \phi|}^2 + \dfrac{{\left( G(h)\phi + \nabla h \cdot \nabla \phi \right)}^2}{2(1+{|\nabla h|}^2)}.
\end{array}
\right.
\end{equation}
Here
\begin{equation}
\label{defG0}
G(h) := \sqrt{1+{|\nabla h|}^2} \mathcal{N}(h),
\end{equation}
and $\mathcal{N}(h)$ is the Dirichlet-Neumann map associated to the domain $\Omega_t$.
Roughly speaking, one can think of $G(h)$ as a first order, non-local, linear operator that depends nonlinearly on the domain.
We refer to  \cite[chap. 11]{SulemBook} or the book of Lannes \cite{LannesBook} for the derivation of \eqref{WWE}.
For sufficiently small smooth solutions, this system admits the conserved energy
\begin{equation}\label{CPWHam}
\begin{split}
\mathcal{H}(h,\phi) &:= \frac{1}{2} \int_{\X^{n-1}} G(h)\phi \cdot \phi \, dx + \frac{g}{2} \int_{\X^{n-1}} h^2 \,dx
  + \sigma\int_{\X^{n-1}} \frac{{|\nabla h|}^2}{1 + \sqrt{1+|\nabla h|^2} } \, dx
  \\
  &\approx {\big\| |\nabla |^{1/2} \phi \big\|}_{L^2}^2 + {\big\| (g-\sigma\Delta)^{1/2}h \big\|}_{L^2}^2,
\end{split}
\end{equation}
which is the sum of the kinetic energy corresponding to the $L^2$ norm of the velocity field and the potential energy due to gravity and surface tension.
It was first observed by Zakharov \cite{Zak0} that \eqref{WWE} is the Hamiltonian flow associated to \eqref{CPWHam}.

The formal linearization of \eqref{WWE}-\eqref{defG0} around a flat and still interface is
\begin{align}
\label{WWElin}
\partial_t h = |\nabla| \phi, \qquad \partial_t \phi = -g h  + \sigma \Delta h.
\end{align}
By defining the linear dispersion relation
\begin{align}
\label{disprel}
\Lambda_{g,\sigma} := \sqrt{g|\nabla|+\sigma|\nabla|^3},
\end{align}
the identitites \eqref{WWElin} can be written as a single equation for a complex-valued unknown,
\begin{align}
\label{WWElin2}
\partial_t u + i \Lambda_{g,\sigma}u = 0, \qquad  u := \sqrt{g+\sigma|\nabla|^2}h + i|\nabla|^{1/2}\phi.
\end{align}

\subsection{Local regularity} 
Due to the complicated nature of the equations, the development of a basic local well-posedness theory for water waves
(existence and uniqueness of smooth solutions for the Cauchy problem) has proved to be highly non-trivial.
Early results include those by Nalimov \cite{Nalimov},
Yosihara \cite{Yosi}, and Craig \cite{CraigLim}, which deal with the case of small perturbations of a flat interface.
It was first proved by Wu \cite{Wu1,Wu2} that local-in-time solutions can be constructed with initial data of arbitrary size in Sobolev spaces, in the irrotational case. 
Following this breakthrough, the question of local well-posedness of the water waves
and free boundary Euler equations has been addressed by many authors.
See, for example, \cite{BG,CL,Lindblad,Ambrose,Lannes,CS2,ShZ3,CHS,ABZ1,ABZ2} 
for local regularity results in various physical situations, which may include vorticity,
surface tension, non-trivial finite bottom, two-fluid systems, or low regularity. 
See also the review paper \cite[Section 2]{IoPuRev} for a longer discussion on local regularity. 

Due to these contributions, the local well-posedness theory of water wave systems is presently well-understood in a variety of different scenarios.
In short, one can say that for sufficiently nice initial data one can construct classical smooth solutions
on a small time interval that depends on the size of the initial data and the arc-chord constant of the initial interface.
In particular, for small data of size $\e$ solutions exist and stay regular for times of $O(\e^{-1})$.

\subsection{Global regularity in Euclidean spaces}\label{sssecEuclidean}
In the Euclidean case $\X=\R$, it is sometimes possible to construct  global-in-time solutions to the Cauchy problem for \eqref{WWE}.
The main mecha\-nism is dispersion, which, combined with localization (decay at a spatial infinity) 
transfers the decay of linear solutions to the nonlinear problem, and gives control for long times.

For $n=3$ (2$d$ interfaces), the first global regularity results were proved  for the gravity problem ($g>0$, $\sigma=0$) by Ger\-main-Masmoudi-Shatah \cite{GMS2} and Wu \cite{Wu3DWW}. Global regularity in $3$d was also proved for the capillary problem ($g=0$, $\sigma>0$) in \cite{GMSC}.
For the case of a finite flat bottom see the works of Wang \cite{Wa2,Wa3}.
More recently, the more difficult question of global regularity for the full gravity-capillary problem ($g>0$, $\sigma>0$) 
was solved by Deng-Ionescu-Pausader-Pusateri \cite{DIPP}. 

In two dimensions (1$d$ interfaces) the first long-time result for \eqref{WWE} 
is due to Wu \cite{WuAG}, where almost-global existence for the gravity problem was obtained. 
This was improved to global regularity by the authors \cite{IoPu2} and Alazard-Delort \cite{ADa,ADb}. See also the refinements by  Hunter and Ifrim-Tataru \cite{HIT,IT} and Wang \cite{Wa1}.
For the capillary problem global regularity was proved independently by the authors \cite{IoPu3,IoPu4} and by Ifrim-Tataru \cite{IT2}. 

We emphasize that all the global regularity results for water waves proved so far require 3 basic assumptions: small data (small perturbations of the rest solution), trivial vorticity inside the fluid, and flat Euclidean geometry. See also the review paper \cite{IoPuRev} for a longer discussion of the main ideas involved in proving global regularity.

\subsection{Long-time regularity on tori}\label{sssecTori}

In the periodic case $\X = \T$, there are no dispersive effects that can lead to decay to control solutions for long times.
In addition, the quasilinear nature of the equations (and the lack of higher order conserved quantities)
prevent the effective use of semilinear techniques to construct global solutions. 
As a result, there are no global regularity results for water waves in periodic settings.

A partial substitute for global regularity is to prove {\it extended lifespan} results, 
which means to show that solutions can be extended smoothly beyond the time of existence predicted by the local theory.\footnote{A 
different mechanism to produce long-term or global solutions in the (one-dimensional) periodic setting is to prove 
the existence of large families of time-quasiperiodic solutions, see the recent papers \cite{BerMon,BBHM}.} 
In the case of small data of size $\varep$, this means extending solutions for times longer than $O(\varep^{-1})$. 
The main tool to prove such results is normal form transformations \cite{shatahKGE}.

\subsubsection*{Quartic energy inequalities}
Let us consider a generic equation of the form
\begin{align}
\label{tori1}
\partial_t u + i \Lambda u = Q(u,\overline{u}),
\end{align}
where $u$ is a solution (defined either on a torus or on a Euclidean space), $\Lambda = \Lambda(\nabla)$ 
is a suitable dispersion relation defined by a real-valued Fourier multiplier $\Lambda$, 
and $Q$ is a suitable quadratic (semilinear or quasi-linear) nonlinearity that may depend on $u$, $\overline{u}$, and their derivatives.
In certain cases one can start with energy estimates and then integrate by parts in time 
(the method of normal forms) to prove a {\it{quartic energy inequality}} of the form 
\begin{equation}\label{quen1}
\big|\mathcal{E}_N(t)-\mathcal{E}_N(0)\big|\lesssim \int_0^t\mathcal{E}_N(s)\cdot\|u(s)\|^2_{H^{N/2}}\,ds,
\end{equation}
for a suitable functional $\mathcal{E}_N(t)$ satisfying $\mathcal{E}_N(t)\approx \|u(t)\|_{H^N}^2$. 
The point is to get two factors of $\|u(s)\|_{H^{N/2}}$ in the right-hand side,
which leads to control of the energy increment over times of length $O(\e^{-2})$, 
and simultaneously avoid loss of derivatives. 

One expects to be able to prove such a quartic energy  inequality 
if the denominators produced by the normal forms do not vanish, in the quantitative form
\begin{align}
\label{torinores2}
\Big| \frac{q_{\pm\pm}(-\xi,\eta,\xi-\eta)}{\Lambda(-\xi)\pm\Lambda(\eta)\pm\Lambda(\xi-\eta)} \Big| 
  \lesssim \min(\langle \xi\rangle,\langle\eta\rangle,\langle\xi-\eta\rangle)^C,
\end{align}
for all frequencies $\xi$ and $\eta$. Here $q_{\pm\pm}$ are suitable multipliers that depend on 
the quadratic part of the original nonlinearity $Q$ in \eqref{tori1} and account for the symmetrization associated to energy estimates.

For water waves systems, quartic energy inequalities like \eqref{quen1} and extended regularity results up to times $T_\e = O(\e^{-2})$ 
have been proved in several $2$d models ($1$d interfaces), both in the Euclidean and in the periodic case: 
pure gravity \cite{WuAG,IoPu2,ADb,HIT}, pure capillarity \cite{IoPu4,IT2}, gravity over a flat bottom \cite{HGIT}, and constant vorticity \cite{ITG}. See \cite{TotzWuNLS} for a $3$d result.
Similar results were obtained earlier for quasi-linear Klein-Gordon equations, see \cite{DelortTAMS} and \cite{DelSzequasi}.
All of these results ultimately rely on the absence of non-trivial quadratic resonances, which is an algebraic condition like \eqref{torinores2}.

\subsubsection*{Iterated normal forms}
In certain cases one can repeat the procedure described above. 
Heuristically, a nonlinear term of homogeneity $\ell$ can be eliminated using normal forms provided there are no non-trivial {\it $(\ell+1)$-resonances}.
More precisely, the natural generalization of the condition \eqref{torinores2} is
\begin{align}
\label{nores}
\Big|\frac{q_{\pm\ldots\pm}(\xi_1,\ldots,\xi_{\ell+1})}{\pm\Lambda(\xi_1)\pm\ldots\pm\Lambda(\xi_{\ell+1})}\Big| 
  \lesssim \big(\text{third highest frequency among }\langle \xi_1\rangle,\ldots,\langle\xi_{\ell+1}\rangle\big)^{C},
\end{align}
for some $C\geq 1$ and for any frequencies $\xi_1,\ldots,\xi_{\ell+1}$ satisfying $\xi_1+\ldots+\xi_{\ell+1}=0$. One could reasonably expect to be able to prove longer $O(\e^{-\ell})$ extended lifespan for systems for which one could verify algebraic conditions such as \eqref{nores}. 


Despite the formal similarity, we remark that usually it is substantially more difficult to verify conditions like \eqref{nores} 
on the absence of high order resonances, than the simpler conditions \eqref{torinores2},
since the denominators might have many zeros.
Among these there are ``trivial'' resonances (for example when $\ell=3$ and $\Lambda$ is even, 
we have  $\Lambda(\xi)-\Lambda(-\xi)+\Lambda(\eta)-\Lambda(-\eta)\equiv 0$ for all frequencies $\xi,\eta$), 
so one has to understand precisely the values of the multipliers $q_{\pm\ldots\pm}$ 
in the numerators in \eqref{nores}, 
and simultaneously account for all the possible symmetrizations coming from energy estimates.

A redeeming feature that allows to carry out this type of arguments, and obtain long $O(\e^{-M})$ lifespan results,
is the presence of (physical) {\it external parameters}.
Indeed, if the dispersion relation $\Lambda$ depends in a non-degenerate way on parameters,
it might be possible to verify algebraic conditions such as \eqref{nores} generically, 
for almost all choices of these parameters. 
This usually works well in one dimension but typically fails on multi-dimensional tori, see the discussion below.
For works in this direction see, for example, the papers of Delort \cite{DelortS1,DelortSd} 
for quasi-linear Klein-Gordon equations on $\mathbb{S}^d$ 
with a mass parameter $m$. 
For works on semilinear PDEs see \cite{Bambusi,DelSzesemi,BaDeGreSze,BaGre,DelIme} and reference therein.

More recently, Berti-Delort \cite{BertiDelort} proved an $O(\e^{-M})$ existence 
result for $1$d periodic gravity-capillary waves in finite depth. This corresponds to $\Lambda(k) = \sqrt{\tanh|k|(\sigma|k|^3+g|k|)}$ 
in \eqref{tori1}, and the result of \cite{BertiDelort} applies for almost all values of the parameters $g,\sigma$. 

Our work in this paper is motivated by the natural question of extending such results
to the setting of water wave models in 3D (2D interfaces).
As we are now going to explain, the case of multi-dimensional tori is very different from the one-dimensional case.
The main reason is that there are a lot more lattice points in $\Z^2$, which lead to unavoidable ``small divisors'' 
(almost resonances).
In fact, in our case, for any choice of the parameters $g,\sigma\in(0,\infty)$ we can only prove degenerate bounds of the form
\begin{align}
\label{SDres}
\big| \Lambda_{g,\sigma}(-\xi)\pm\Lambda_{g,\sigma}(\eta)\pm\Lambda_{g,\sigma}(\xi-\eta) \big|^{-1}
  \lesssim \max(\langle \xi\rangle,\langle\eta\rangle,\langle\xi-\eta\rangle)^{3/2+}.
\end{align}
Compared to \eqref{torinores2}, these bounds on resonances are much weaker and lead to derivative losses 
in normal form transformations that cannot be handled by the general considerations described above. 
Indeed, to our knowledge, prior to this result, no long-time regularity theorems have been proved on multi-dimensional tori 
in the presence of non-trivial resonances.

\subsection*{Acknowledgments}
We would like to thank J.M. Delort for useful discussions on the topic, and for pointing out the work \cite{DelIme}
where extended lifespan results are obtained for strongly semilinear KG equations
in the presence of small divisors.

\subsection{The main theorem}
Our main theorem is a long-time regularity result for the water waves system \eqref{WWE} on $\T^2$,
for almost all choices of parameters $(g,\sigma)\in (0,\infty)^2$. 

\begin{theorem}\label{MainTheo}
Assume that $N_0=30$, $(g,\sigma)\in (0,\infty)^2$, and $g/\sigma\in(0,\infty)\setminus\mathcal{N}$, 
where $\mathcal{N}$ is a set of measure $0$; see \eqref{defN0} for the definition.
Assume that we are given initial data $(h_0,\phi_0)\in H^{N_0+1}(\T^2)\times H^{N_0+1/2}(\T^2)$ satisfying the assumptions
\begin{equation}
\label{mai1}
{\|\langle \nabla \rangle h_0\|}_{H^{N_0}} + {\| \,|\nabla|^{1/2} \phi_0\|}_{H^{N_0}} \leq \e \leq 1,
  \qquad \int_{\T^2}h_0(x)\,dx = 0.
\end{equation}
Then, there is a unique solution $(h,\phi)\in C([0,T_\e]:H^{N_0+1} \times H^{N_0+1/2})$ of the system \eqref{WWE} on $\T^2\times [0,T_\varep]$,
with initial data $(h(0),\phi(0)) = (h_0,\phi_0)$, where
\begin{equation}
\label{mai2}
T_\e \gtrsim_{g,\sigma}\varepsilon^{-5/3}[\log(2/\varepsilon)]^{-2}.
\end{equation}
Moreover, for any $t\in[0,T_\varep]$, 
\begin{equation}
\label{mai3}
{\|\langle \nabla \rangle h(t)\|}_{H^{N_0}} + {\| \,|\nabla|^{1/2} \phi(t) \|}_{H^{N_0}} \lesssim _{g,\sigma}\e,
  \qquad \int_{\T^2}h(x,t)\,dx=0.
\end{equation} 

\end{theorem}

The main point of our theorem is the nontrivial lifespan of solutions given by the lower bound \eqref{mai2}.
For $\varep\ll 1$, this goes beyond the usual time of existence $T\gtrsim_{g,\sigma} \varep^{-1}$ given by the local theory
which holds for all pairs $(g,\sigma)\in (0,\infty)^2$. We remark that the power of $\log(2/\varepsilon)$ in \eqref{mai2} 
can be improved, but it is not clear to us whether the main power $\e^{-5/3}$ can be improved.

\subsection{Main ideas and sketch of the proof}\label{ssecIdeas}

We outline first some of the main ideas of our proof on a model equation 
which retains some of the most relevant properties of the water waves system \eqref{WWE}.

\subsubsection{The model equation and a partial result}
Let us consider the evolution equation
\begin{align} 
\label{ModelWW}
\begin{split}
& \partial_t U + i\Lambda U = \nabla V\cdot\nabla U + \tfrac{1}{2}\Delta V\cdot U =: \mathcal{N}, 
\\ & \Lambda := \Lambda_{g,1} = \sqrt{g|\nabla|+|\nabla|^3}, \qquad V := P_{\leq 10}\Im U, \qquad U(0) = U_0,
\end{split}
\end{align}
for a scalar unknown $U:I_t\times \T^2_x \rightarrow \C$,
where the projection $P_{\leq 10}$ (and the usual standard Littlewood-Paley projections) 
are defined at the beginning of Subsection \ref{SecParaCalc}.

The model \eqref{ModelWW}, which was also used in \cite{DIPP}, is a good substitute for the full system \eqref{WWE}, since the linear part of \eqref{ModelWW} coincides with the linearization of the full system \eqref{WWE}, see \eqref{WWElin}--\eqref{WWElin2} with $\sigma =1$, and the solutions of \eqref{ModelWW} satisfy the $L^2$ conservation law
\begin{equation}
\label{modelcons}
{\|U(t)\|}_{L^2} = {\|U_0\|}_{L^2}, \qquad -\infty < t < \infty,
\end{equation}
which is a good substitute for the Hamiltonian structure of the original water wave systems. Moreover, the quasilinear nonlinearity is quadratic with a structure similar to the main term $iT_{V\cdot\zeta}$ in the paralinearized 
version of the full system \eqref{eqUN}.

The system \eqref{ModelWW} has the advantage of being algebraically simple, so we can use it to explain the main mechanism of the proof. To prove long-term regularity we implement a bootstrap argument: assume that $\e\ll 1$ and $U \in C([0,T_\e]:H^{N_0}(\T^2))$ is a solution of \eqref{ModelWW} satisfying 
\begin{equation}
\label{modeldata}
{\| U_0 \|}_{H^{N_0}} \leq \e \ll 1\qquad \text{ and }\qquad {\| U(t) \|}_{H^{N_0}} \leq \KK_g \e\text{ for any }t\in[0,T_\e],
\end{equation}
where $\KK_g\geq 2$ is a large constant. We would then like to prove the better bounds
\begin{equation}
\label{modelbootimp}
{\| U(t) \|}_{H^{N_0}} \leq \KK_g \e/2\text{  for any }t\in[0,T_\e].
\end{equation}

\smallskip
{\it{Step 1.}} As in some of our earlier work \cite{IoPu3,IoPu4}, we use a quasi-linear I-method. Let $W := \langle\nabla\rangle^N U$, and consider the basic energy functional, written in the Fourier space,
\begin{equation*}
\mathcal{E}_N(t) := \| \langle\nabla\rangle^N U(t) \|_{L^2}^2 = \frac{1}{(2\pi)^2} \sum_{\xi\in\Z^2} |\widehat{W}(\xi,t)|^2,
\end{equation*}
where $\what{W}(\xi)$ is the Fourier coefficient of $W$ at the frequency $\xi$. 
Using \eqref{ModelWW} we calculate
\begin{equation}
\label{evol1}
\frac{d}{dt}\mathcal{E}_N(t) = \Re \sum_{\xi,\eta\in\Z^2}m(\xi,\eta) \what{W}(\eta) \what{\overline{W}}(-\xi) \, i \what{U}(\xi-\eta),
\end{equation}
with
\begin{equation}
\label{evolsym1}
m(\xi,\eta) := c \, [(\xi-\eta)\cdot (\xi+\eta)] \frac{(1+|\eta|^2)^N - (1+|\xi|^2)^N}{(1+|\eta|^2)^{N/2}(1+|\xi|^2)^{N/2}} \varphi_{\leq 10}(\xi-\eta),
   \qquad c\in\R.
\end{equation}
We notice that $m(\xi,\eta)$ satisfies
\begin{equation}
\label{evolsym2}
m(\xi,\eta)=\mathfrak{d}(\xi,\eta)m'(\xi,\eta),\quad\text{ where }\quad\mathfrak{d}(\xi,\eta):=\frac{[(\xi-\eta)\cdot(\xi+\eta)]^2}{1+|\xi+\eta|^2},\quad |m'|\approx 1.
\end{equation}
The {\it{depletion factor}} $\mathfrak{d}$ plays an important role in our argument.
The presence of this factor is related to the exact conservation law \eqref{modelcons} and is the starting point for our proof.

We define the {\it{linear profiles}} $u$ and $w$ by the formulas
\begin{align}
\label{evol2}
U(t) = e^{-it\Lambda}u(t), \qquad W(t) = e^{-it\Lambda}w(t),
\end{align}
and rewrite the term in the right-hand side of \eqref{evol1} in terms of the profiles $u$ and $w$.
Thus, to prove the desired estimates \eqref{modelbootimp}, it suffices to bound the energy increment 
\begin{align}
\label{evolaim}
\Big| \Re \int_0^T  \sum_{\xi,\eta\in\Z^2} m(\xi,\eta) e^{it\Phi(\xi,\eta)} \what{w}(\eta) \what{\overline{w}}(-\xi) \, i \what{u}(\xi-\eta) \, dt \Big|
  \lesssim \e^2,
\end{align}
for $T\leq T_\e$, where $\Phi$ is the {\it cubic phase function}
\begin{align}\label{phiphase}
\Phi(\xi,\eta):= \Lambda(\xi) - \Lambda(\eta) - \Lambda(\xi-\eta).
\end{align}

\smallskip
{\it{Step 2.}} The contribution of large modulations, 
corresponding to inserting a cutoff function of the form $\varphi_{>1}(\Phi(\xi,\eta))$ in the left-hand side of \eqref{evolaim}, 
can be bounded according to the general normal form argument described in \eqref{quen1}--\eqref{torinores2}. 
We use the equations for the profiles
\begin{align}
\label{evol3}
\partial_t u = e^{it\Lambda}\mathcal{N}, \qquad \partial_t w = e^{it\Lambda} (\nabla V \cdot \nabla W + O(V \cdot W)),
\end{align}
integrate by parts in time, and then re-symmetrize the resulting quartic expressions to avoid loss of derivatives. 
The corresponding contribution is bounded by $CT_\e(\KK_g\e)^4$, which is acceptable.

\smallskip
{\it{Step 3.}} The contribution of small modulations $|\Phi(\xi,\eta)|\lesssim 1$ and high frequencies 
which, in general, is the most dangerous part, can be bounded using the frequency gain in \eqref{evolsym2}. 
More precisely, we observe that the symbol $m$ in \eqref{evolsym1}-\eqref{evolsym2} satisfies the bounds 
\begin{align}
\label{introcorre}
|m(\xi,\eta)| \lesssim \frac{[(\xi-\eta)\cdot(\xi+\eta)]^2}{1+|\xi+\eta|^2} \lesssim \frac{1 + |\Phi(\xi,\eta)|}{1 + |\xi+\eta|},
\end{align}
on the support of the sum in \eqref{evolaim}. Thus
\begin{align}
\label{evolaimlh}
\Big|\int_0^T  \sum_{\xi,\eta\in\Z^2} m(\xi,\eta) e^{it\Phi(\xi,\eta)} \what{w}(\eta) \what{\overline{w}}(-\xi) \, i \what{u}(\xi-\eta)\varphi_{\leq 0}(\Phi(\xi,\eta))\varphi_{\geq D}(\xi)\, dt \Big|\lesssim T_\e(\KK_g\e)^32^{-D},
\end{align}
where $D=D(\e)\gg 1$ is a suitable parameter to be fixed.

\smallskip
{\it{Step 4.}} Finally, we bound the contribution of small modulations $|\Phi(\xi,\eta)|\lesssim 1$ and low frequencies $|\xi|,|\eta|\lesssim 2^D$.
Here we need the parameter $g$ to be in the complement of a set of measure $0$ to avoid exact resonances.
More precisely, using a counting argument we show that for almost every $g\in(0,\infty)$ we have the sharp bounds
\begin{align}
\label{evolSD}
|\Phi(\xi,\eta)|^{-1} \lesssim \langle\xi\rangle^{3/2+},
\end{align}
for all $\xi,\eta\in\mathbb{Z}^2$, $\xi-\eta\neq 0$, $|\xi|\approx|\eta| \gg |\xi-\eta|\approx 1$. 
See Proposition \ref{prop1} for a precise statement. We integrate by parts in time, and bound the contribution of the sum by
\begin{align}
\label{evol15}
\begin{split}
C\Big|\Re \int_0^T  \sum_{\xi,\eta\in\Z^2} \frac{m(\xi,\eta)}{i\Phi(\xi,\eta)} \varphi_{\leq 0}(\Phi(\xi,\eta)) \varphi_{\leq D}(\xi)
  \, e^{it\Phi(\xi,\eta)} \what{\partial_t w}(\eta) \what{\overline{w}}(-\xi) \, i \what{u}(\xi-\eta) \, dt\Big|
\end{split}
\end{align}
plus similar or easier terms. 
In analyzing \eqref{evol15}, we notice that we have a $\langle\xi\rangle^{3/2+}$ 
loss from the denominator $|\Phi(\xi,\eta)|$, a $\langle\xi\rangle^{-1}$ gain from the symbol $m$ (see \eqref{introcorre}), 
a $\langle\xi\rangle$ loss from the derivative loss in the equation \eqref{evol3} for $w$, 
and an $\KK_g\e$ gain from the quadratic terms in \eqref{evol3}. Thus the expression in \eqref{evol15} is bounded by
\begin{equation*}
CT_\e2^{(3/2+)D}(\KK_g\e)^4.
\end{equation*}
Comparing now with \eqref{evolaimlh}, the two main conditions are $T_\e\e^32^{-D}\ll\e^2$ and $T_\e2^{(3/2+)D}\e^4\ll\e^2$.
These conditions are compatible if and only if $T_\e\ll\e^{-7/5+}$, by taking $2^D\approx\e^{-2/5}$. 

\subsubsection{Quartic resonances and the full result}
The argument outlined above would give us a partial extended existence result for times $T_\e$ of the order $O(\e^{-7/5+})$. 
To obtain the desired result with $T_\e=O(\e^{-5/3+})$ we need to expand substantially the analysis in Step 4. 

More precisely, we examine the formula \eqref{evol15} and use the equation \eqref{evol3} for the factor $\partial_t w$.
The main contribution comes from the term $e^{it\Lambda} \nabla V \cdot \nabla W$, 
and we can further express this contribution in terms of the profiles $u$ and $w$. 
We are thus led to consider quartic expression with oscillatory factors given by {\it quartic phase} or four-way modulation functions
\begin{align}
\label{evol17}
\Psi_\iota(\xi,\eta,\rho) := \Lambda(\xi) -  \Lambda(\rho)+\iota\Lambda(\eta-\rho) - \Lambda(\xi-\eta),\qquad\iota\in\{+,-\}.
\end{align}
A first helpful observation is that the term $\partial_t w \approx e^{it\Lambda} \nabla V \cdot \nabla W$ does not give a loss
of a full derivative, but only of a half derivative in the worst case scenario when $|\Psi_\iota|$ is small.
This is another manifestation of a depletion phenomenon similar to the one 
described after \eqref{evolsym2}.

To obtain the full result we have to integrate by parts in time again.
To do this we need good lower bounds on the quartic phase functions \eqref{evol17}.
Such bounds follow from Proposition \ref{prop2} which essentially states that, after removing another set of measure 0 of parameters $g$,
\begin{align}
\label{SDres2}
\big| \Psi_\iota(\xi,\eta,\rho) \big|^{-1}  \lesssim \langle \xi\rangle^{1/2},
\end{align}
when $|\xi|\approx|\rho| \gg |\xi-\eta|\approx|\rho-\eta| \approx 1$. These bounds hold provided that the cubic modulation $\Phi(\xi,\eta)$ is sufficiently small, 
and the frequency variables are not a ``trivial'' cubic resonance, i.e. $\xi=\rho$, $\iota=+$, with the notation in \eqref{evol17}.
We then show that the contributions from the trivial resonances vanish
due to the reality of the symbol $m$ in \eqref{evolsym1}; conceptually, this fact is connected to the time-reversibility of the equation.
Such property allows us to integrate by parts again and, after more analysis, eventually reach the result of Theorem \ref{MainTheo}.

\medskip
\subsubsection{The full system and the ``improved good variable''}
We return now to the proof of the main theorem, for the full water waves system \eqref{WWE}. 
This system is, of course, more complicated than the simplified model \eqref{ModelWW}. 
However, it has remarkable (Hamiltonian) structure, which we need to exploit. 

The starting point of the proof for the model equation \eqref{ModelWW} 
outlined above is the gain of one derivative in the low modulation region, see \eqref{introcorre}.
In the case of the full system \eqref{WWE}, we are able to achieve the same gain by constructing a new variable using paradifferential calculus. 
This new variable can be regarded as a refinement of the good unknown of Alinhac (used for the local existence theory by Alazard-Burq-Zuily \cite{ABZ1,ABZ2}).
The advantage over Alinhac's good unknown is that our variable not only avoids losses of derivatives, 
but it actually gives a gain one derivative in energy estimates in the most dangerous region of the phase space
where the modulation is small but the frequencies are large.

This type of ``improved good variable'' has been already used in our earlier work \cite{DIPP}.
Here we prove more precise bounds, with a better description of the nonlinearity, see Proposition \ref{proeqU}.
This precise structure of the nonlinearity is important,
in particular the reality of the symbol $a_{++}$ in the quadratic term $\mathcal{Q}_U$,
and the gain of $3/2$ derivatives of the quadratic term $\mathcal{R}_U$.

\subsection{Organization}
The rest of the paper is organized as follows. In Section \ref{section2} we introduce our main notation, prove two propositions on  lower bounds for cubic and quartic phase functions, and state our main bootstrap proposition. In Section \ref{SecEqs} we use paradifferential calculus to construct our main variable $U$ and derive suitable evolution equations for this variable and its derivatives. In Section \ref{SecEn1} we start the proof of our main energy estimates, and show how to control the contribution of large modulations. Finally, in Section \ref{SecEn2} we bound the contributions from small modulations, using the lower bounds for the cubic and the quartic phase functions.

\section{Definitions, lemmas, and the main bootstrap proposition}\label{section2}

\subsection{Small divisor estimates}\label{SecSD}

Our proof of the main theorem depends in a critical way on understanding the structure of the resonances of the evolution.
In this subsection we show how to estimate from below quadratic phases (or three-way resonance functions) and certain cubic phases (four-way resonance functions),  
for almost all parameters $(g,\sigma)\in (0,\infty)^2$.
The set of parameters $g/\sigma$ that we need to exclude for our main result to hold is
\begin{align}
\label{defN0}
\mathcal{N} = \mathcal{N}_{1/2} \cup \mathcal{R},
\end{align}
where $\mathcal{N}_{1/2}$ and $\mathcal{R}$ are defined respectively in Propositions \ref{prop1} and \ref{prop2}.

\subsubsection{3-way resonances}\label{SecSD1}

The main conclusion of this subsection can be stated as follows: for all $\kappa >0$ and almost all parameters $g/\sigma$ the quadratic phase functions
\begin{align}\label{moddef}
\Phi_{g,\sigma}(\xi,\eta) = \Lambda_{g,\sigma}(\xi) \pm \Lambda_{g,\sigma} (\xi-\eta) \pm \Lambda_{g,\sigma}(\eta) ,\qquad \Lambda_{g,\sigma}(v)=\sqrt{g|v|+\sigma|v|^3},
\end{align}
satisfy the lower bounds
\begin{align}
\label{SDbound0}
|\Phi_{g,\sigma}(\xi,\eta)|  \gtrsim_{g,\sigma,\kappa} |\xi|^{-3/2} (\log(2+|\xi|))^{-1-\kappa} \, |\xi-\eta|^{-4},
\end{align}
for all $\xi \neq \eta \in \Z^2$ with $|\xi-\eta|\leq \min(|\xi|,|\eta|)$, and all choices of the signs $+$ or $-$.

To prove this precisely, we start with a lemma:

\begin{lemma}\label{lem1} 
Assume that $a,b,c\in[1,\infty)$ and $a\leq b\leq c\leq a+b$. Let $F:[0,\infty)\to\mathbb{R}$,
\begin{equation}\label{ml1}
F(x):=\sqrt{ax+a^3}+\sqrt{bx+b^3}-\sqrt{cx+c^3}.
\end{equation}

(i) If $|F(x_0)|\leq 1/10$ for some $x_0\in[0,\infty)$ then 
\begin{equation}\label{ml2}
F'(x_0)\geq \frac{a}{10\sqrt{ax_0+a^3}}.
\end{equation}

(ii) As a consequence, the function $F$ vanishes at most at one point $p=p(a,b,c)\in[0,\infty)$. 
Moreover, for any $B\geq 1$ and $\delta\in(0,1/20]$ the set of points
\begin{equation}\label{ml3}
X_{B,\delta}:=\{x\in (0,B):\,|F(x)|<\delta\}
\end{equation}
is an interval (or the empty set) with $|X_{B,\delta}|\leq 20\delta\sqrt{a+B}$. Finally,
\begin{equation}\label{ml13.1}
X_{B,\delta}=\emptyset\quad\text{ unless }\quad c-b\lesssim_B a^{3/2}c^{-1/2}. 
\end{equation} 
\end{lemma}


\begin{proof} (i) We write, using $a+b\geq c$,
\begin{equation*}
\begin{split}
2F'(x)&=\frac{a}{\sqrt{ax+a^3}}+\frac{b}{\sqrt{bx+b^3}}-\frac{c}{\sqrt{cx+c^3}}\\
&\geq \Big(\frac{a}{\sqrt{ax+a^3}}-\frac{a}{\sqrt{cx+c^3}}\Big)+\Big(\frac{b}{\sqrt{bx+b^3}}-\frac{b}{\sqrt{cx+c^3}}\Big).
\end{split}
\end{equation*}
Since $1\leq a\leq b$, if $|F(x_0)|\leq 1/10$ then $\sqrt{cx_0+c^3}\geq \sqrt{bx_0+b^3}$, $\sqrt{cx_0+c^3}\geq (3/2)\sqrt{ax_0+a^3}$. Thus
\begin{equation*}
2F'(x_0)\geq \frac{a}{\sqrt{ax_0+a^3}}-\frac{a}{\sqrt{cx_0+c^3}}\geq \frac{a}{3\sqrt{ax_0+a^3}},
\end{equation*} 
as claimed.

(ii) Let 
\begin{equation}\label{ml4}
I:=\{x\in [0,\infty):\,|F(x)|<1/20\}.
\end{equation}
Clearly $I$ is an open set in $[0,\infty)$, thus a union of open intervals. We claim that $I$ consists, in fact, of at most one such open interval. 
Indeed, assume that $(y_0,z_0)\subseteq I$ is an open interval and $z_0\notin I$. 
Since $F'(z_0)>0$ (see \eqref{ml2}) we have $F(z_0)=1/20$ and there is $\delta>0$ such that $[z_0,z_0+\delta)\cap I=\emptyset$.
In fact we claim that $[z_0,\infty)\cap I=\emptyset$. 
Otherwise let $p:=\inf\,[z_0,\infty)\cap I$, so $p>z_0$, $F(p)=1/20$, and $F'(p)>0$ (due to \eqref{ml2}); 
in particular there is $p'<p$ close to $p$ such that $F(p')<F(p)$, thus $p'\in I$, contradicting the definition of $p$.

Thus $I\subset[0,\infty)$ is an open interval and, in view of \eqref{ml2},
\begin{equation*}
F'(x)\geq \frac{1}{10\sqrt{x+a}}\qquad\text{ for any }\qquad x\in I.
\end{equation*} 
It follows that $|X_{B,\delta}|\leq 20\delta\sqrt{a+B}$, as claimed. Moreover, since $\sqrt{c^3+cx}-\sqrt{b^3+bx}\gtrsim_B (c-b)c^{1/2}$ for all $x\in[0,B]$ and $1\leq b\leq c$, the conclusion \eqref{ml13.1} follows as well. 
\end{proof}

For $y\in(0,\infty)$ we consider now the dispersion relations $\Lambda_y:\mathbb{Z}^2\to\mathbb{R}$
\begin{equation}\label{ml10}
\Lambda_y(v):=\sqrt{y|v|+|v|^3}.
\end{equation}
For $\iota_1,\iota_2,\iota_3\in\{+,-\}$ and $y\in(0,\infty)$ we define also the resonance functions
\begin{equation}\label{ml11}
\Psi_y^{\iota_1,\iota_2,\iota_3}(v_1,v_2,v_3):=\iota_1\Lambda_y(v_1)+\iota_2\Lambda_y(v_2)+\iota_3\Lambda_y(v_3),
\end{equation}
where $v_1,v_2,v_3\in\mathbb{Z}^2$, $v_1+v_2+v_3=0$. 
For $\kappa\in(0,1]$ let
\begin{equation}\label{ml12}
K_\kappa(v_1,v_2,v_3):=\frac{1}{\langle v\rangle_{max}^{3/2}\log(1+\langle v\rangle_{max})^{1+\kappa}}
  \frac{1}{\min(\langle v_1\rangle,\langle v_2\rangle,\langle v_3\rangle)^4},
\end{equation} 
where $\langle v\rangle_{max}:=\max(\langle v_1\rangle,\langle v_2\rangle,\langle v_3\rangle)$. For $B,j\in[5,\infty)\cap\mathbb{Z}$ let
\begin{equation}\label{ml15}
\begin{split}
\mathcal{N}_{j,\kappa}^{B}:=&\big\{y\in(0,B):\,|\Psi_y^{\iota_1,\iota_2,\iota_3}(v_1,v_2,v_3)|< 2^{-j}K_\kappa(v_1,v_2,v_3)
  \text{ for some }\\
&\iota_1,\iota_2,\iota_3\in\{+,-\}\text{ and some }v_1,v_2,v_3\in\mathbb{Z}^2_\ast\text{ with }v_1+v_2+v_3=0\big\},
\end{split}
\end{equation}
where $\mathbb{Z}^2_\ast:=\mathbb{Z}^2\setminus\{(0,0)\}$. Let
\begin{equation}\label{ml16}
\mathcal{N}_\kappa:=\bigcup_{B\geq 5}\big(\bigcap_{j\geq 5}\mathcal{N}_{j,\kappa}^B\big).
\end{equation}

\begin{proposition}\label{prop1} 
The set $\mathcal{N}_\kappa\subseteq (0,\infty)$ has measure $0$, for any $\kappa\in(0,1]$. Moreover, for any $y\notin\mathcal{N}_\kappa$ there is a constant $c_y>0$ such that
\begin{equation}\label{ml17}
|\Psi_y^{\iota_1,\iota_2,\iota_3}(v_1,v_2,v_3)|\geq c_yK_\kappa(v_1,v_2,v_3)
\end{equation}
for any $\iota_1,\iota_2,\iota_3\in\{+,-\}$ and any $v_1,v_2,v_3\in\mathbb{Z}^2_\ast$ with $v_1+v_2+v_3=0$.
\end{proposition} 

\begin{proof}
The claim \eqref{ml17} is just  a consequence of the definitions. 
It suffices to prove that for any $B\in[1,\infty)\cap\mathbb{Z}$ the measure of the set $\cap_{j\geq 0}\mathcal{N}_{j,\kappa}^B$ is $0$. 
For this is suffices to prove that
\begin{equation}\label{ml20}
|\mathcal{N}_{j,\kappa}^B|\lesssim _{B,\kappa}2^{-j}\qquad\text{ for any }B,j\in[5,\infty)\cap\mathbb{Z}.
\end{equation} 

We examine the definition \eqref{ml15} of the sets $\mathcal{N}_{j,\kappa}^B$. 
Since $\Lambda_y(v)\geq 1$ for all $v\in\mathbb{Z}^2_\ast$ and $y>0$, the inequality $|\Psi_y^{\iota_1,\iota_2,\iota_3}(v_1,v_2,v_3)|\leq 1$ 
can hold only if not all the signs are equal and for certain configurations of the variable $v_1,v_2,v_3$. 
More precisely,
\begin{equation}\label{ml21}
\begin{split}
&\mathcal{N}_{j,\kappa}^B=\bigcup_{\xi,\eta\in\mathbb{Z}^2_\ast,\,|\eta|\leq|\xi|\leq|\xi+\eta|}\mathcal{N}_{j,\kappa}^B(\xi,\eta);\\
&\mathcal{N}_{j,\kappa}^B(\xi,\eta):=\big\{y\in(0,B):\,|\Lambda_y(\eta)+\Lambda_y(\xi)-\Lambda_y(\xi+\eta)|< 2^{-j}K_\kappa(\xi,\eta,-\xi-\eta)\big\}.
\end{split}
\end{equation}

We can use Lemma \ref{lem1} with $a:=|\eta|$, $b:=|\xi|$, $c:=|\xi+\eta|$ to analyze the sets $\mathcal{N}_{j,\kappa}^B(\xi,\eta)$.
These sets are either empty or they are open intervals of length $\lesssim 2^{-j}K_\kappa(\xi,\eta,-\xi-\eta)\sqrt{\langle \eta\rangle+B}$.
In view of Lemma \ref{lem1} (ii), for \eqref{ml20} it suffices to prove that
\begin{equation}\label{ml30}
\sum_{\xi,\eta\in\mathbb{Z}^2_\ast,\,|\eta|\leq|\xi|\leq|\xi+\eta|,\,|\xi+\eta|-|\xi|
\lesssim_B|\eta|^{3/2}|\xi|^{-1/2}}|\eta|^{1/2}K_\kappa(\xi,\eta,-\xi-\eta)\lesssim_{B,\kappa}1. 
\end{equation}

We divide the sum dyadically, over $|\eta|\approx 2^l$ and $|\xi|\approx 2^k$. The sum in the left-hand side of \eqref{ml30} is dominated by 
\begin{equation*}
C\sum_{l,k\geq 0,\,k\geq l-2}\sum_{|\eta|\in[2^{l-1},2^{l+1}],\,|\xi|\in [2^{k-1},2^{k+1}],\,\big||\xi+\eta|-|\xi|\big|
\lesssim_B 2^{3l/2}2^{-k/2}}|\eta|^{1/2}K_\kappa(\xi,\eta,-\xi-\eta).
\end{equation*}
The condition $\big||\xi+\eta|-|\xi|\big|\lesssim_B 2^{3l/2}2^{-k/2}$ implies that $|\xi\cdot\eta|\lesssim_B 2^{3l/2}2^{k/2}$.
Using also the definition \eqref{ml12} the sum above is bounded by
\begin{equation*}
C\sum_{l,k\geq 0,\,k\geq l-2}\sum_{|\eta|\in[2^{l-1},2^{l+1}],\,|\xi|\in [2^{k-1},2^{k+1}],\,|\xi\cdot\eta|
\lesssim_B 2^{3l/2}2^{k/2}}2^{l/2}\frac{1}{2^{3k/2}(1+k)^{1+\kappa}2^{4l}}.
\end{equation*}
For $l,k,\eta$ fixed the sum over $\xi$ is bounded by $C_B(1+k)^{-1-\kappa}2^{-3l}$,
so the entire sum is bounded by $C_{B,\kappa}$, as claimed in \eqref{ml30}.
\end{proof}

\subsubsection{4-way resonances}\label{SecSD2}
In this section we control four-way interactions, or iterated resonances.
This step is needed to get a better power of $\varep$ in \eqref{mai2},
as it allows us to perform iterated (partial) normal forms in certain ranges of frequencies. 
Roughly speaking, the main conclusion of this subsection, Proposition \ref{prop2}, can be interpreted as follows:
if $\xi \neq \eta \in \Z^2$, $|\xi-\eta|\leq \min(|\xi|,|\eta|)$, and $|\Phi_{g,\sigma}(\xi,\eta)|$ is sufficiently small, where  
$\Phi_{g,\sigma}(\xi,\eta)$ is the quadratic phase defined in \eqref{moddef}, then the cubic phase
\begin{align}
\Psi_{g,\sigma}(\xi,\eta,\rho) =  \Lambda_{g,\sigma}(\xi) \pm \Lambda_{g,\sigma} (\xi-\eta) \pm \Lambda_{g,\sigma}(\eta-\rho) \pm \Lambda_{g,\sigma}(\rho)
\end{align}
satisfies a good lower bound of the form
\begin{align}
\label{SDbound02}
|\Psi_{g,\sigma}(\xi,\eta,\rho)|  \gtrsim_{g,\sigma} |\xi|^{-1/2} (|\xi-\eta|+|\eta-\rho|)^{-2},
\end{align}
for all $\eta, \xi, \rho \in \Z^2$ with $\xi\neq \rho$, $|\xi-\eta|^{16}+|\eta-\rho|^{16} \ll\min(|\xi|,|\eta|,|\rho|)$, provided that $g/\sigma$ is outside a set of measure zero. To define this set precisely, we begin with a lemma: 

\begin{lemma}\label{quart1}
There is a set $\mathcal{R}\subseteq (0,\infty)$ of measure $0$ with the property that for any $y\notin\mathcal{R}$ there is $b_y>0$ such that
\begin{equation}\label{ml50}
\Big|\frac{\Lambda_y(\xi)}{|\xi|}-\frac{\Lambda_y(\eta)}{|\eta|}\Big|\geq b_y|\xi|^{-6}
\end{equation}
for any $\xi,\eta\in\Z^2_\ast$ satisfying $\eta=\lambda\xi$, $\lambda\in(0,1)$.
\end{lemma}

\begin{proof} For $B,j\in[5,\infty)\cap \Z$ we define the sets
\begin{equation}\label{ml51}
\begin{split}
\mathcal{R}_j^B:=\Big\{y\in(0,B):\,\Big|&\frac{\Lambda_y(\xi)}{|\xi|}-\frac{\Lambda_y(\eta)}{|\eta|}\Big|<2^{-j}|\xi|^{-6}\\
&\text{ for some }\xi,\eta\in\Z^2_\ast\text{ with }\eta=\lambda\xi,\,\lambda\in(0,1)\Big\}.
\end{split}
\end{equation}
As before, let
\begin{equation}\label{ml52}
\mathcal{R}:=\bigcup_{B\geq 5}\big(\bigcap_{j\geq 5}\mathcal{R}_j^B\big).
\end{equation}

The inequality \eqref{ml50} is clearly satisfied for $y\in(0,\infty)\setminus\mathcal{R}$, so we only need to prove that the measure of the set $\mathcal{R}$ is equal to $0$. For this is suffices to prove that
\begin{equation}\label{ml53}
|\mathcal{R}_{j}^B|\lesssim _{B}2^{-j}\qquad\text{ for any }B,j\in[5,\infty)\cap\mathbb{Z}.
\end{equation} 

Assume $y\in \mathcal{R}_{j}^B$ and assume $\xi,\eta$ are points in $\Z^2_\ast$, $\eta=\lambda\xi$, $\lambda\in(0,1)$, such that the inequality in \eqref{ml51} is satisfied. In particular
\begin{equation*}
\Big|\frac{y|\xi|+|\xi|^3}{|\xi|^2}-\frac{y|\eta|+|\eta|^3}{|\eta|^2}\Big|\lesssim_B2^{-j}|\xi|^{-5.5}.
\end{equation*}
After algebraic simplifications, this shows that
\begin{equation*}
\big|(|\eta|-|\xi|)(y-|\xi||\eta|)\big|\lesssim_B2^{-j}|\xi|^{-3.5}.
\end{equation*}
Since $\xi,\eta$ have integer coordinates and $\eta=\lambda \xi$, $\lambda\in (0,1)$, we have $\big||\eta|-|\xi|\big|\gtrsim 1$. In particular, the inequality above shows that
\begin{equation}\label{ml54}
\big|y-|\xi||\eta|\big|\lesssim_B2^{-j}|\xi|^{-3.5}.
\end{equation}
Therefore $y$ belongs to a union over $k\geq 0$ and $\xi,\eta\in\Z^2_\ast$ of open intervals of lengths $\lesssim_B2^{-j}2^{-3.5k}$ centered around numbers of the form $|\xi||\eta|$, where $|\xi|\approx 2^k$ and $\eta=\lambda\xi$, $\lambda\in (0,1)$. For every $k\geq 0$ fixed, $\xi$ can have up to $C2^{2k}$ possible locations, and for every such location there are no more than $C2^k$ possible locations for $\eta$ along the line from $0$ to $\xi$. The desired conclusion \eqref{ml53} follows.
\end{proof}

We are now ready to prove our main result in this subsection.

\begin{proposition}\label{prop2} 
Assume that $\mathcal{R}\subseteq (0,\infty)$ is defined as in \eqref{ml51}--\eqref{ml52} and $y\in(0,\infty)\setminus \mathcal{R}$.
Then there is a small constant $b'_y>0$ such that
\begin{equation}\label{ml42}
|\Lambda_y(v+\xi)-\Lambda_y(v)-\iota_1\Lambda_y(\xi)|+|\Lambda_y(v+\eta)-\Lambda_y(v)-\iota_2\Lambda_y(\eta)|
  \geq b'_y\langle v\rangle^{-1/2}(\langle\xi\rangle+\langle\eta\rangle)^{-2}
\end{equation}
for any $\iota_1,\iota_2\in\{+,-\}$ and any $v,\xi,\eta\in\Z^2_\ast$ with $(\xi,\iota_1)\neq(\eta,\iota_2)$ and $(|\xi|+|\eta|)^{16}\leq b'_y|v|$.
\end{proposition}

\begin{proof} 
The main point is to prove that at least one of the two modulations
 $|\Lambda_y(v+\xi)-\Lambda_y(v)-\iota_1\Lambda_y(\xi)|$ or $|\Lambda_y(v+\eta)-\Lambda_y(v)-\iota_2\Lambda_y(\eta)|$ is bounded from below, 
essentially, by the large frequency to the power $-1/2$. 
This is better than the power $-3/2$ obtained in Proposition \ref{prop1}, which is the optimal power in the case when one considers only one modulation.

Assume, for contradiction, that the conclusion of the lemma fails. 
Thus for any constant $\delta_y>0$ there are signs $\iota_1,\iota_2\in\{+,-\}$ and lattice points $v,\xi,\eta\in\Z^2_\ast$ 
satisfying $|\xi|+|\eta|\leq R$, $(\xi,\iota_1)\neq(\eta,\iota_2)$ and $|v|\geq \delta_y^{-1}R^{16}$ such that
\begin{equation}\label{ml43}
|\Lambda_y(v+\xi)-\Lambda_y(v)-\iota_1\Lambda_y(\xi)|+|\Lambda_y(v+\eta)-\Lambda_y(v)-\iota_2\Lambda_y(\eta)|\leq \delta_y|v|^{-1/2}R^{-2}.
\end{equation}

\medskip
{\bf{Step 1.}} We show first that the vectors $0$, $\xi$, and $\eta$ are aligned, 
i.e. $\xi_1\eta_2=\xi_2\eta_1$, where $\xi=(\xi_1,\xi_2)$ and $\eta=(\eta_1,\eta_2)$. For this we notice that
\begin{equation*}
\Lambda_y(w)=P_y(|w|^2)\qquad\text{ where }\qquad P_y(\rho):=\big(y\rho^{1/2}+\rho^{3/2}\big)^{1/2}.
\end{equation*}
Notice that $P_y(\rho)=\rho^{3/4}+O_y(\rho^{-1/4})$, thus
\begin{equation*}
|\partial_\rho P_y(\rho)-(3/4)\rho^{-1/4}|+|\partial_\rho^2 P_y(\rho)|\lesssim_y\rho^{-5/4}
\end{equation*}
for $\rho\in[1,\infty)$. In particular, for $\xi,\eta,v$ as in \eqref{ml43} and $\mu\in\{\xi,\eta\}$, we have
\begin{equation}\label{ml44}
\Big|\Lambda_y(v+\mu)-\Lambda_y(v)-(|v+\mu|^2-|v|^2)(3/4)|v|^{-1/2}\Big|\lesssim_y(|v+\mu|^2-|v|^2)^2|v|^{-5/2}.
\end{equation}
In particular, using \eqref{ml43}, it follows that $\big||v+\mu|^2-|v|^2\big|\lesssim_y R^{3/2}|v|^{1/2}$ and
\begin{equation*}
\Big|\frac{3}{4}(|v+\mu|^2-|v|^2)|v|^{-1/2}-\iota_\mu\Lambda_y(\mu)\Big|\leq 2\delta_y|v|^{-1/2}R^{-2}
\end{equation*}
for $\mu\in\{\xi,\eta\}$, where $\iota_\xi=\iota_1$ and $\iota_\eta = \iota_2$. Therefore
\begin{equation}\label{ml45}
\Big|\frac{3}{2}(v\cdot\mu)|v|^{-1/2}-\iota_\mu\Lambda_y(\mu)
  + \frac{3}{4}|\mu|^2|v|^{-1/2}\Big|\leq 2\delta_y|v|^{-1/2}R^{-2}.
\end{equation}

Letting $v=|v|e$, where $e=(e_1,e_2)\in\mathbb{S}^1$ is a unit vector in $\R^2$, 
it follows from \eqref{ml45} that $|e\cdot \mu|\lesssim_y R^{3/2}|v|^{-1/2}$ for $\mu\in\{\xi,\eta\}$. 
This shows that $\xi_1\eta_2=\xi_2\eta_1$. Indeed, otherwise $|\xi_1\eta_2-\xi_2\eta_1|\geq 1$ 
and we could solve the system (in $e_1$ and $e_2$)
\begin{equation*}
e_1\xi_1+e_2\xi_2=a,\qquad e_1\eta_1+e_2\eta_2=b.
\end{equation*}
Since $|a|+|b|\lesssim_y R^{3/2}|v|^{-1/2}$ it would follow that $|e_1|+|e_2|\lesssim_y R^{5/2}|v|^{-1/2}$, 
in contradiction with the assumptions $|e|=1$ and $|v|\geq\delta_y^{-1}R^{16}$.

To summarize, we showed that $\eta=\lambda\xi$ for some $\lambda\in\mathbb{R}$ 
and that the inequalities \eqref{ml45} hold for $\mu\in\{\xi,\eta\}$.

\medskip
{\bf{Step 2.}} We show now that $\eta=\pm\xi$. Indeed, using \eqref{ml45} 
and keeping only the first two terms we have
\begin{equation*}
\Big|\frac{3}{2}(v\cdot\mu)|v|^{-1/2}-\iota_\mu\Lambda_y(\mu)\Big|\lesssim_y|v|^{-1/2}R^{2},
\end{equation*}
for $\mu\in\{\xi,\eta\}$. In particular
\begin{equation*}
\Big|\frac{3}{2}|v\cdot\mu||v|^{-1/2}-|\Lambda_y(\mu)|\Big|\lesssim_y|v|^{-1/2}R^{2}.
\end{equation*}
It follows that
\begin{equation}\label{ml47}
\Big|\frac{\Lambda_y(\xi)}{|\xi|}-\frac{3|v\cdot\xi|}{2|\xi|}|v|^{-1/2}\Big|+\Big|\frac{\Lambda_y(\eta)}{|\eta|}-\frac{3|v\cdot\eta|}{2|\eta|}|v|^{-1/2}\Big|
  \lesssim_y|v|^{-1/2}R^{2}.
\end{equation}
Since $\eta=\lambda\xi$ for some $\lambda\in\mathbb{R}$ we have $|v\cdot\xi|/|\xi|=|v\cdot\eta|/|\eta|$. 
We apply now \eqref{ml50}. Since $|v|^{-1/2}\leq \delta_y^{1/2}R^{-8}$, 
the inequalities \eqref{ml47} can only be satisfied if $|\eta|=|\xi|$, as claimed.

\medskip
{\bf{Step 3.}} 
Finally, we show that $(\xi,\iota_\xi)=(\eta,\iota_\eta)$, in contradiction with the original assumptions. 
Indeed, using again \eqref{ml45} and keeping only the first two terms we have
\begin{equation*}
\Big|\frac{3}{2}(\iota_{\mu}\mu\cdot v)|v|^{-1/2}-\Lambda_y(\mu)\Big|\lesssim_y|v|^{-1/2}R^{2},
\end{equation*}
for $\mu\in\{\xi,\eta\}$. Since $\eta=\pm\xi$, this implies that $\iota_{\xi}\xi=\iota_\eta\eta$ 
(otherwise one would have $\iota_{\xi}\xi=-\iota_\eta\eta$, and the inequalities 
above would give $|\Lambda_y(\xi)|\lesssim_y|v|^{-1/2}R^{2}$, in contradiction with the original assumptions). 
Moreover, if $\eta=-\xi$ and $\iota_\eta=-\iota_\xi$ then we use again 
the inequalities \eqref{ml45}, this time with all the terms. It follows that $|\xi|^2|v|^{-1/2}\lesssim \delta_y|v|^{-1/2}R^{-2}$, 
in contradiction with the original assumptions. 
The remaining case $\eta=\xi$ and $\iota_\eta=\iota_\xi$ was excluded in the original assumptions, which gives the contradiction. 
This completes the proof of the proposition.
\end{proof}

\subsection{Elements of paradifferential calculus}\label{SecParaCalc} 
Our proof of the main theorem is based on establishing suitable energy estimates for solutions of the system \eqref{WWE}. 
For this we use a paradifferential formulation. In this subsection we summarize the main definitions and properties of the Weyl paradifferential calculus.

We fix $\varphi: \R \to [0,1]$ an even smooth function supported in $[-8/5,8/5]$ and equal to $1$ in $[-5/4,5/4]$. 
For simplicity of notation, we also let $\varphi: \R^2 \to [0,1]$ denote the corresponding radial function on $\R^2$. 
Let
\begin{equation}
\label{LP}
\begin{split}
& \varphi_k(x):=\varphi(|x|/2^k)-\varphi(|x|/2^{k-1})\,\,\text{ for any }\,\,k\in\mathbb{Z},
\\
& \varphi_I:=\sum_{m\in I\cap\mathbb{Z}}\varphi_m\text{ for any }I\subseteq\mathbb{R},
\\
& \varphi_{\leq B}(x):=\varphi(x/2^B),\qquad \varphi_{>B}:=1-\varphi_{\leq B}\,\,\text{ for any }\,\,B\in\mathbb{R}.
\end{split}
\end{equation}
We denote by $P_k$, $k \in \Z$, the Littlewood--Paley projection operators defined by the Fourier multipliers $\xi \to \varphi_k(\xi)$. 
Notice that, in our periodic setting, $P_k\equiv 0$ if $k\leq -1$.  
We let $P_{\leq B}$, respectively $P_{>B}$, denote the operators defined by the Fourier  
multipliers $\xi\to \varphi_{\leq B}(\xi)$, respectively $\xi \to \varphi_{>B}(\xi)$. 
Finally, we let $P_{-\infty}$ denote the average operator on $\T^2$, $P_{-\infty}f(x):=(2\pi)^{-2}\int_{\T^2}f(y)\,dy$.

We recall now the definition of paradifferential operators on $\T^2$ (Weyl formulation):
given a symbol $a = a(x,\zeta): \T^2 \times \R^2 \to \mathbb{C}$, we define the operator $T_a$ by
\begin{equation}
\label{Taf}
\begin{split}
\mathcal{F} \big( T_a f \big)(\xi) := \frac{1}{4\pi^2} 
  \sum_{\eta\in\Z^2} \chi \Big(\frac{|\xi-\eta|}{|\xi+\eta|} \Big) \widetilde{a}(\xi-\eta,(\xi+\eta)/2) \widehat{f}(\eta),
\end{split}
\end{equation}
where $\widetilde{a}$ denotes the partial Fourier transform of $a$ in the first coordinate and $\chi=\varphi_{\leq -20}$.  

By convention, if $\xi=0$ then $\chi\big(\frac{|\xi-\eta|}{|\xi+\eta|} \big)\equiv 0$ for all $\eta\in\Z^2$. 
In particular, $T_ag\equiv 0$ if $g$ is a constant function. 
Moreover, for any $f$, the function $T_af$ depends only on the values of the symbol $a(x,\zeta)$ for $|\zeta|\geq 1$. 
Thus we may assume that the symbols $a$ are defined only for $|\zeta|>1/2$, and write $a=a'$ if $a(x,\zeta)=a'(x,\zeta)$ for $x\in\T^2$ and $|\zeta|>1/2$.

To estimate symbols we will use the spaces $\mathcal{M}^l_r$ defined by the norms
\begin{equation}
\label{symclass}
{\| a \|}_{\mathcal{M}^{l}_{r}} := \sup_{|\alpha| + |\beta| \leq r} \sup_{|\zeta|>1/2} \,
    \langle \zeta\rangle^{-l} {\|\, \langle\zeta\rangle^{|\beta|} \partial^\beta_\zeta \partial_x^\alpha a(x,\zeta) \|}_{L^2_x}.
\end{equation}
The number $l\in\mathbb{R}$ indicates the order of the operator while $r\in[4,\infty)\cap\mathbb{Z}$ measures its (less important) differentiability. 
For simplicity, in the periodic case considered here we measure all functions and symbols in $L^2_x$ based spaces.

Given two symbols $a$ and $b$, recall the definition of their Poisson bracket,
\begin{align}
\label{poisson}
\{ a,b \} := \nabla_x a \nabla_\zeta b - \nabla_\zeta a \nabla_x b.
\end{align}

We record now several simple properties of paradifferential operators, which follow mostly from definitions. See also  \cite[Appendix A]{DIPP} for similar proofs in the Euclidean case. 

\begin{lemma}[Basic Properties]\label{PropSym} 
(i) For any $r\geq 6$ and $l_1,l_2\in\R$ we have
\begin{equation}
\label{propsymb1}
{\| a b \|}_{\mathcal{M}_r^{l_1+l_2}} + {\| \{a,b\} \|}_{\mathcal{M}_{r-2}^{l_1+l_2-1}}  \lesssim {\| a \|}_{ \mathcal{M}_r^{l_1}} {\| b \|}_{\mathcal{M}_r^{l_2}}.
\end{equation}

(ii) If $a\in\mathcal{M}^l_6$ is a symbol then
\begin{equation}\label{LqBdTa}
\begin{split}
{\| T_a f \|}_{H^{m-l}} & \lesssim {\| a \|}_{\mathcal{M}_{6}^l} {\| f \|}_{H^m}.
\end{split}
\end{equation}

(iii) Given symbols $a\in\mathcal{M}^{l_1}_{12}$, $b\in\mathcal{M}^{l_2}_{12}$, $|l_1|,|l_2|\leq 10$, we have
\begin{equation}
\label{lemComp1}
T_a T_b = T_{ab} + \frac{i}{2}T_{\{a,b\}} + E(a,b)
\end{equation}
where $E$ is an error term such that 
\begin{equation}
\label{RemBounds}
{\| E(a,b)f\|}_{H^{m-l_1-l_2+2}}\lesssim {\| a \|}_{\mathcal{M}^{l_1}_{12}} {\| b \|}_{\mathcal{M}^{l_2}_{12}}
  {\| f \|}_{H^m}.
\end{equation}
Moreover, if $a,b\in H^{12}$ are functions then $T_aT_b-T_{ab}$ is a smoothing operator
\begin{equation}
\label{RemBounds2}
{\|(T_aT_b-T_{ab})f\|}_{H^{m+6}}\lesssim {\| a \|}_{H^{12}} {\| b \|}_{H^{12}}
  {\| f \|}_{H^m}.
\end{equation}

(iv) If $a\in \mathcal{M}_{6}^0$ is real-valued then $T_a$ is a bounded self-adjoint operator on $L^2$. Moreover, if $a\in \mathcal{M}_{6}^0$ and $f\in L^2$ then
\begin{equation}\label{Alu2}
\overline{T_af}=T_{a'}\overline{f},\quad\text{ where }\quad a'(y,\zeta):=\overline{a(y,-\zeta)}.
\end{equation} 
\end{lemma}

In some of our calculations in section \ref{SecEqs} we will need more precise formulas for the error terms $E(a,b)$. The definitions show easily that that if $a=a(\zeta)$ is a Fourier multiplier, $b$ is a symbol, and $f$ is a function, then 
\begin{equation}
\label{ExEab}
\begin{split}
\widehat{E(a,b)f}(\xi) = \frac{1}{4\pi^2}\sum_{\eta\in\Z^2} \chi(\frac{\vert\xi-\eta\vert}{\vert\xi+\eta\vert})
  &\Big( a(\xi)-a(\frac{\xi+\eta}{2})-\frac{\xi-\eta}{2}\cdot\nabla a(\frac{\xi+\eta}{2}) \Big)\widetilde{b}(\xi-\eta,\frac{\xi+\eta}{2})\widehat{f}(\eta),
\end{split}
\end{equation}
and
\begin{equation}
\label{ExEab2}
\begin{split}
\widehat{E(b,a)f}(\xi) = \frac{1}{4\pi^2}\sum_{\eta\in\Z^2} \chi(\frac{\vert\xi-\eta\vert}{\vert\xi+\eta\vert})
&\Big( a(\eta)-a(\frac{\xi+\eta}{2})-\frac{\eta-\xi}{2}\cdot\nabla a(\frac{\xi+\eta}{2}) \Big)\widetilde{b}(\xi-\eta,\frac{\xi+\eta}{2})\widehat{f}(\eta).
\end{split}
\end{equation}

Our last lemma in this subsection concerns paralinearization of products. 

\begin{lemma}[Paralinearization]\label{lemHH}
(i) If $f,g\in L^2$ then
\begin{equation*}
fg = T_fg + T_gf + \mathcal{H}(f,g)
\end{equation*}
where $\mathcal{H}$ is smoothing in the sense that
\begin{equation}\label{al7}
\widehat{\mathcal{H}(f,g)}(\xi)=\frac{1}{4\pi^2}\sum_{\eta\in\Z^2}a_p(\xi-\eta,\eta)\widehat{f}(\xi-\eta)\widehat{g}(\eta),\qquad |a_p(v,w)|\lesssim \Big(\frac{1+\min(|v|,|w|)}{1+\max(|v|,|w|)}\Big)^{40}.
\end{equation}
As a consequence, if $f,g \in H^{m}$, $m\geq 20$, then
\begin{equation}\label{aux1.1}
{\|\mathcal{H}(f,g) \|}_{H^{m+10}} \lesssim {\|f\|}_{H^{m}} {\|g\|}_{H^{m}}.
\end{equation}

(ii) Assume that $F(z) = z + h(z)$, where $h$ is analytic for $|z| <1/2$ and satisfies $| h(z)| \lesssim |z|^3$. 
Then, for all  $m \geq 20$, and $u\in H^m$ with ${\| u \|}_{H^{m}}\leq 1/100$, one has
\begin{equation}
\label{Paracomp}
F(u) = T_{F^\prime(u)} u + E(u), \qquad {\| E(u) \|}_{H^{m+10}} \lesssim {\| u \|}_{H^{m}}^3.
\end{equation}
\end{lemma}

\subsection{The main bootstrap proposition}\label{MainToProve}

The following is our main bootstrap proposition:

\begin{proposition}\label{MainBootstrap}
Assume that $N_0=30$, $\sigma=1$, and $g\in(0,\infty)\setminus\mathcal{N}$, where $\mathcal{N}=\mathcal{N}_{1/2}\cup\mathcal{R}$, 
see definitions \eqref{ml16} and \eqref{ml52}. Then there is a sufficiently large constant $\mathcal{K}_{g}\in [10,\infty)$ with the following property: 
assume that $\varep\in (0,\mathcal{K}_g^{-4}]$,
\begin{equation}\label{al0}
1\leq T\leq T_\e:= \KK^{-4}_g \varep^{-5/3}[\log(2/\varep)]^{-2},
\end{equation}
and $(h,\phi)\in C([0,T]:H^{N_0+1}\times H^{N_0+1/2})$ is a solution of the system \eqref{WWE} 
with $\sigma=1$ on $\T^2\times [0,T]$ satisfying the assumptions
\begin{equation}\label{al01}
{\|\langle \nabla \rangle h(0)\|}_{H^{N_0}} + {\| \,|\nabla|^{1/2} \phi(0) \|}_{H^{N_0}} \leq \varep,\qquad \int_{\T^2}h(x,0)\,dx=0,
\end{equation}
and
\begin{equation}\label{al1}
\sup_{t\in[0,T]}\big\{{\|\langle \nabla \rangle h(t)\|}_{H^{N_0}} + {\| \,|\nabla|^{1/2} \phi(t) \|}_{H^{N_0}}\big\} \leq \KK_g\varep.
\end{equation} 
Then $\int_{\T^2}h(x,t)\,dx=0$ for any $t\in[0,T]$ and the solution satisfies the improved bounds
\begin{equation}\label{al2}
\sup_{t\in[0,T]}\big\{{\|\langle \nabla \rangle h(t)\|}_{H^{N_0}} + {\| \,|\nabla|^{1/2} \phi(t) \|}_{H^{N_0}}\big\} \leq \KK_g\varep/2.
\end{equation}  
\end{proposition}

We notice that solutions of the water waves system \eqref{WWE} have the following rescaling pro\-per\-ty:
if $(h,\phi)$ is a solution of the system \eqref{WWE} on $\T^2\times [0,T]$ with parameters $(g,\sigma)$, then the pair $(h',\phi')$ defined by
\begin{equation*}
h'(x,t):=h\big(x,t/\sqrt{\sigma}\big),\qquad \phi'(x,t):=\frac{1}{\sqrt{\sigma}}\phi\big(x,t/\sqrt{\sigma}\big)
\end{equation*}
is a solution of \eqref{WWE} on $\T^2\times [0,\sqrt{\sigma}T]$ with parameters $(g',\sigma'):=(g/\sigma,1)$. 

As a consequence, it is clear that Theorem \ref{MainTheo} follows from Proposition \ref{MainBootstrap}
and the local regularity theorem for the system \eqref{WWE}. 
The rest of the paper is concerned with the proof of Proposition \ref{MainBootstrap}.

\section{Paralinearization and the ``improved good variable''}\label{SecEqs}
In this section we assume that $N_0=30$ and $(h,\phi)\in C([0,T]:H^{N_0+1}\times H^{N_0+1/2})$ is a solution of \eqref{WWE} satisfying
\begin{equation}
\label{Sob1}
{\|\langle \nabla \rangle h(t)\|}_{H^{N_0}} + {\| \,|\nabla|^{1/2} \phi(t) \|}_{H^{N_0}} \leq \oe\ll 1,\qquad \int_{\T^2}h(x,t)\,dx=0
\end{equation}
for any $t\in[0,T]$. Our goal in this section is to express the system \eqref{WWE} as a scalar equation for a suitably constructed complex-valued function.
The variable we construct here, which we call the ``improved good variable'', 
is more suitable for long time analysis than the standard ``good variable'' constructed in the local theory. 
The main advantage of this variable is that it allows us to use of a special structure in the equation, 
connected to the original Hamiltonian structure of the system, 
which effectively leads to a gain of one derivative in the region close to the resonant sets.

\subsection{Paralinearization} Before stating and proving our main Proposition \ref{proeqU}, we state two lemmas 
about the paralinearization of the Dirichlet-Neumann operator $G(h)\phi$ and the water waves system \eqref{WWE}. 

\begin{lemma}[Paralinearization of the Dirichlet-Neumann operator]\label{DNmainLem}
Assume that $(h,\phi)$ is a solution of the system \eqref{WWE} satisfying \eqref{Sob1}. Define
\begin{align}
\label{DefBV}
B := \frac{G(h)\phi+\nabla_xh\cdot\nabla_x\phi}{1+\vert\nabla h\vert^2}, 
  \qquad V := \nabla_x \phi - B\nabla_x h, \qquad \omega := \phi - T_B h.
\end{align}
Then $G(h)\phi,B,V\in \oe H^{N_0-1/2}$ and
\begin{equation}\label{int0}
\int_{\T^2}[G(h)\phi](x,t)\,dx=0\qquad\text{ for any }t\in[0,T].
\end{equation}
Moreover, we can paralinearize the Dirichlet-Neumann operator as
\begin{equation}
\label{DNmainformula}
G(h)\phi = T_{\lambda}\omega- \mathrm{div} (T_V h) + G_2 + \oe^3 H^{N_0+3/2},
\end{equation}
where
\begin{align}
\label{deflambda}
\begin{split}
\lambda &:= \lambda^{(1)}+\lambda^{(0)}, 
\\
\lambda^{(1)}(x,\zeta) &:= \sqrt{(1+|\nabla h|^2) |\zeta|^2-(\zeta\cdot\nabla h)^2},
\\
\lambda^{(0)}(x,\zeta) &:= \frac{(1+|\nabla h|^2)^2}{2\lambda^{(1)}} 
  \Big\{ \frac{\lambda^{(1)}}{1+|\nabla h|^2}, \frac{\zeta \cdot \nabla h}{1+|\nabla h|^2} \Big\} + \frac{1}{2}\Delta h.
\end{split}
\end{align}
The quadratic terms in \eqref{DNmainformula} are given by
\begin{equation}
 \label{DNquadratic}
G_2 = G_2(h, |\nabla|^{1/2}\omega) \in \oe^2 H^{N_0+5/2}, 
  \qquad \what{G_2}(\xi) = \frac{1}{4\pi^2}\sum_{\eta\in\Z^2} g_2(\xi,\eta) \what{h}(\xi-\eta) |\eta|^{1/2} \what{\omega}(\eta) \,d\eta,
\end{equation}
where $g_2$ is a symbol satisfying
\begin{align}
\begin{split}
 \label{DNquadraticsym}
|g_2(\xi,\eta)| & \lesssim |\xi|\min(|\eta|,|\xi-\eta|)^{1/2} \Big(\frac{1+\min(|\eta|,|\xi-\eta|)}{1+\max(|\eta|,|\xi-\eta|)}\Big)^{7/2}.
\end{split}
\end{align}

Moreover we have
\begin{equation}\label{dtg}
\partial_t(G(h)\phi) - |\nabla|\partial_t\phi\in \oe^2 H^{N_0-2}.
\end{equation}

\end{lemma}

A similar result in the Euclidean case is proved in Proposition B.1 of \cite{DIPP}.
It is based on the use of paradifferential calculus along the lines of the works of Alazard-Burq-Zuily \cite{ABZ1,ABZ2}.
Here we stated all the formulas using the simpler setting of Sobolev spaces of periodic functions, 
while a more elaborate setting (including decay information and vector-fields) is used in \cite{DIPP}.
Paralinearization formulas like \eqref{DNmainformula} have been crucially used in many previous works
on the local \cite{ABZ1,ABZ2} and global \cite{ADa,ADb} theory.

As a consequence of Lemma \ref{DNmainLem} one can obtain the following paralinearization of the water waves system:

\begin{lemma}[Paralinearization of the system]\label{SysparaLem}
Let $\Lambda := \sqrt{g|\nabla|+\sigma|\nabla|^3}$, and define
\begin{align}
\label{defell}
\begin{split}
\ell(x,\zeta) & := L_{ij}(x) \zeta_i \zeta_j - \Lambda^2 h,  
  \qquad L_{ij} := \frac{\sigma }{\sqrt{1+|\nabla h|^2}} \Big(\delta_{ij} - \frac{\partial_i h \partial_j h}{1 + |\nabla h|^2} \Big),
\end{split}
\end{align}
to be the mean curvature operator coming from the surface tension.

With the notation of Lemma \ref{DNmainLem} we can rewrite the system \eqref{WWE} as
\begin{equation}\label{WWpara}
\left\{
\begin{array}{l}
\partial_t h = T_{\lambda} \omega - \mathrm{div}( T_V h) + G_2 + \oe^3 H^{N_0+1},
\\
\\
\partial_t \omega = - g h  - T_\ell h - T_V \nabla\omega + \Omega_2 + \oe^3 H^{N_0+1},
\end{array}
\right.
\end{equation}
where 
\begin{align}
\label{d_tomegaquadratic}
& \Omega_2 := \frac{1}{2} \mathcal{H}(|\nabla|\omega,|\nabla|\omega) 
  - \frac{1}{2} \mathcal{H}(\nabla\omega,\nabla\omega) \in \oe^2 H^{N_0+2}.
\end{align}
\end{lemma}

\subsection{The improved good variable} We show now that the system \eqref{WWpara}
can be symmetrized by introducing an improved version of Alinhac's ``good unknown'', 
which reveals a key structure in the system, related to smoothing of resonant interactions.
This is done in Proposition \ref{proeqU} below, whose proof follows the ideas in \cite{DIPP}.
However, here we uncover additional special structures in the equations that have not been observed before.
We will comment on these structures and their relevance, after the statement of the proposition, in Remark \ref{Remark1}.

\begin{proposition}\label{proeqU}
Assume that $(h,\phi)\in C([0,T]:H^{N_0+1}\times H^{N_0+1/2})$ is a solution of the water waves system \eqref{WWE} 
satisfying \eqref{Sob1}. With $\lambda$ and $\ell$ as above, we define the symbol 
\begin{align}
\label{defSigma}
\Sigma(x,\zeta) & :=\sqrt{\lambda(x,\zeta)(g+\ell(x,\zeta))},
\end{align}
and the complex-valued variable
\begin{align}
\label{defscalarunk}
\begin{split}
&U :=  T_{\sqrt{g+\ell}}h + iT_{\Sigma}T_{1/\sqrt{g+\ell}}\omega+iT_{m^\prime}\omega,\qquad
m^\prime(x,\zeta) :=\frac{i}{2}\frac{(\div V)(x)}{\sqrt{g+\ell(x,\zeta)}},
\end{split}
\end{align}
where $V$ and $\omega$ are defined in \eqref{DefBV}. 
Then 
\begin{align}
\label{Uhorel} 
U = \sqrt{g + \sigma|\nabla|^2} h + i |\nabla|^{1/2} \omega + \oe^2 H^{N_0},
\end{align}
and $U$ satisfies the equation
\begin{align}
\label{eqU}
& (\partial_t + i T_{\Sigma} + iT_{V\cdot\zeta})U = \mathcal{N}_U + \mathcal{Q}_U + \mathcal{R}_U + \mathcal{C}_U,
\end{align}
where 

\setlength{\leftmargini}{2em}
\begin{itemize}

\medskip
\item  The quadratic term $\mathcal{N}_U$ has the special explicit structure
\begin{equation}
\label{eqUN}
\mathcal{N}_U := c_1 T_\gamma U + c_2 T_\gamma\overline{U}, 
\end{equation}
for some constants $c_1,c_2 \in \R$, 
where 
\begin{equation}
\label{gamma}
\gamma(x,\zeta) := \frac{\zeta_i\zeta_j}{ |\zeta|^2} |\nabla|^{-1/2} \partial_i\partial_j(\Im U)(x).
\end{equation}

\medskip
\item The quadratic terms $\mathcal{Q}_U$ are of the form
\begin{align}
\label{eqUquadspec}
& \mathcal{Q}_U = A_{++}(\Im U,U) + A_{+-}(U,\bar{U}) + A_{--}(\bar{U},\bar{U})\in \oe^2 H^{N_0+1},
\end{align}
with symbols $a_{\iota_1\iota_2}$, that is
\begin{align}
\label{bilnot}
\mathcal{F}[A_{\iota_1\iota_2} (f,g)] (\xi) := \frac{1}{4\pi^2} \sum_{\eta\in\Z^2} a_{\iota_1\iota_2}(\xi,\eta) \what{f}(\xi-\eta) \what{g}(\eta),
\end{align}
satisfying for all $\xi,\eta \in \mathbb{Z}^2$, and $(\iota_1\iota_2) \in \{ (++), (+-), (--)\}$
\begin{align}
\label{symbolseqUspec}
\begin{split}
|a_{\iota_1\iota_2}(\xi,\eta)| & \lesssim \frac{[1+|\xi-\eta|]^4}{1+|\eta|} \varphi_{\leq -5}(|\xi-\eta|/|\eta|),
\qquad a_{++}(\xi,\eta) \in \R.
\end{split}
\end{align}

\medskip
\item The quadratic terms $\mathcal{R}_U$ have a gain of $3/2$ derivatives, i.e. they are of the form
\begin{align}
\label{eqUquad}
& \mathcal{R}_U = B_{++}(U,U) + B_{+-}(U,\bar{U}) + B_{--}(\bar{U},\bar{U})\in \oe^2 H^{N_0+3/2},
\end{align}
with symbols $b_{\iota_1\iota_2}$ satisfying, for all $\xi,\eta \in \mathbb{Z}^2$, and $(\iota_1\iota_2) \in \{ (++), (+-), (--)\}$,
\begin{align}
\label{symbolseqU}
\begin{split}
|b_{\iota_1\iota_2}(\xi,\eta)|\lesssim\frac{[1+\min(|\eta|,|\xi-\eta|)]^4}{[1+\max(|\eta|,|\xi-\eta|)]^{3/2}}.
\end{split}
\end{align}

\medskip
\item $\mathcal{C}_U$ is a cubic term, i.e. it satisfies for any $t \in [0,T]$
\begin{align}
\label{cubiceqU}
{\| \mathcal{C}_U \|}_{H^{N_0}} \lesssim \oe^3.
\end{align}
\end{itemize}

\end{proposition}

Notice that in Proposition \ref{proeqU} we keep the parameters $g$ and $\sigma$ due to their physical significance. 
Later on we will reduce, without loss of generality, to the case $\sigma = 1$.

\begin{definition}\label{def1}
We denote a finite linear combinations of terms of the form \eqref{eqUquadspec}--\eqref{symbolseqUspec}, with $a_{++}(\xi,\eta) \in \R$ 
by $\oe^2 H^{N_0+1}_{r}$. Similarly, we denote a finite linear combinations of terms of the form \eqref{eqUquad}--\eqref{symbolseqU} by $\oe^2 H_\ast^{N_0+3/2}$. 
\end{definition}

\begin{remark}\label{Remark1} 
We examine the structure of the main equation \eqref{eqU}. In the left-hand side we have 
the ``quasilinear'' part $(\partial_t + i T_{\Sigma} + iT_{V\cdot\zeta})U$. 
Here $\Sigma$ is a real-valued operator of order $3/2$, and $V$ is the velocity vector defined in \eqref{DefBV}. 
In the right-hand side of \eqref{eqU} we have four types of terms: 

\begin{itemize}

\medskip
\item[(1)] An explicit quadratic term $\mathcal{N}_U$ which involves the symbol $\gamma$.
The key property, which is a consequence of the choice of the symbol $m^\prime$,
is an effective gain of one derivative in the region close to the resonant hypersurfaces. 
Indeed, we notice that
\begin{equation}\label{gamma2}
\widetilde{\gamma}(\rho,\zeta) = -\frac{\zeta_i\zeta_j}{|\zeta|^2} \frac{\rho_i \rho_j}{|\rho|^{1/2}} 
  \widehat{\Im U}(\rho)
\end{equation}
and remark that the presence of the angle $(\zeta \cdot \rho)^2$ in this expression will eventually give us 
the derivative gain in the resonant region (see also the related factor $\mathfrak{d}$ and \eqref{ProBulk2sym}).



\medskip
\item[(2)] A quadratic term $\mathcal{Q}_U$, whose symbols always gain one derivative, 
and have a special reality property for the terms where the high frequency is on $U$ (which is the most relevant interaction in the energy estimates)
as opposed to $\bar{U}$.
This reality property is needed in order to take care of the ``trivial resonances'' in the quartic bulk terms.

\medskip
\item[(3)] A ``{\it{strongly semilinear}}'' quadratic term $\mathcal{R}_U$, which gains $3/2$ derivatives; 
due to this large gain, the contribution of this term to the energy increment will be easier to treat.

\medskip
\item[(4)] A semilinear cubic term $\mathcal{C}_U \in \oe^3 H^{N_0}$, whose contribution is straightforward to estimate. 
\end{itemize}
\end{remark}

Let us make one more remark about the term $\mathcal{N}_U$.
\begin{remark}
The structure of $\mathcal{N}_U$ is natural for the equation \eqref{eqU}. Indeed, in order to obtain high-order energy estimates on $U$
we will apply the operator $T_\Sigma$ multiple times to the equation, and derive an equation for $W = (T_\Sigma)^n U$, see \eqref{defW_n}
and Proposition \ref{proeqW}.
Then, terms of the form $T_\gamma W$ will appear when commuting $T_\Sigma$ and $T_{V\cdot \zeta}$.


\end{remark}

For easy reference we collect below some bounds for the various functions and symbols that appear in Proposition \ref{proeqU}.

\begin{lemma}\label{symbbou} (i) We have
\begin{equation}\label{SymApprox0}
\begin{split}
\lambda^p &=|\zeta|^p\Big(1+\frac{p\lambda_1^{(0)}(x,\zeta)}{\vert\zeta\vert} + \oe^2 \mathcal{M}^{0}_{N_0-4}\Big),
\\
(g+\ell)^p& = (g+\sigma\vert\zeta\vert^2)^p \Big(1 -  \frac{p\Lambda^2h}{g + \sigma|\zeta|^2} + \oe^2 \mathcal{M}^{0}_{N_0-4}\Big),
\\
\Sigma&=\Lambda(\zeta)\Big(1+\frac{\Sigma_1(x,\zeta)}{\Lambda(\zeta)} + \oe^2 \mathcal{M}^{0}_{N_0-4}\Big),
\end{split}
\end{equation}
for $p\in[-2,2]$, where $\Lambda(\zeta)=\sqrt{|\zeta|(g+\sigma|\zeta|^2)}$, 
\begin{equation}\label{Exps}
\begin{split}
\lambda_1^{(0)}(x,\zeta) &:=\frac{|\zeta|^2 \Delta h - \zeta_j\zeta_k\partial_j\partial_kh}{2|\zeta|^2}\in\oe \mathcal{M}^{0}_{N_0-4},
\\
\Sigma_1 &:= \Big\{\frac{1}{4}\frac{\Lambda(\zeta)}{|\zeta|} \Big[ \Delta h - \frac{\zeta_i\zeta_j}{|\zeta|^2}\partial_{i}\partial_jh \Big]
  - \frac{1}{2} \frac{|\zeta|}{\Lambda(\zeta)} \Lambda^2 h\Big\} \in \oe \mathcal{M}^{1/2}_{N_0-4}.
\end{split}
\end{equation}

(ii) We have
\begin{equation}\label{ExpFun}
\begin{split}
&V,\,G(h)\phi,\,B\in\oe H^{N_0-1/2},\qquad m'\in\oe\mathcal{M}^{-1}_{N_0-4},\qquad|\nabla|^{1/2}\phi,|\nabla|^{1/2}\omega\in\oe H^{N_0},
\\
&\phi-\omega\in\oe^2H^{N_0+1},\qquad G(h)\phi-|\nabla|\omega\in\oe^2H^{N_0-1/2}.
\end{split}
\end{equation}
Moreover, we have
\begin{equation}
\label{ExpV0}
\begin{split}
&\Re (U)=\sqrt{g+\sigma|\nabla|^2} h+\oe^2 H^{N_0},\qquad\Im (U)=|\nabla|^{1/2}\omega+\oe^2 H^{N_0},\\
\end{split}
\end{equation}
and
\begin{equation}
\begin{split}
\label{ExpV}
&V = V_1 + \oe^2 H^{N_0-1/2}, \qquad V_1 := |\nabla|^{-1/2} \nabla \Im (U),
\\
&m'= m'_1 + \oe^2 \mathcal{M}^{-1}_{N_0-4}, \qquad 
  m_1'(x,\zeta):=\frac{i}{2}\frac{|\nabla|^{3/2}\Im(U)}{\sqrt{g+\sigma|\zeta|^2}},
\end{split}
\end{equation}

(iii) We have
\begin{equation}
\label{SymApprox2}
\begin{split}
\partial_t\sqrt{g+\ell} & = (g+\sigma|\zeta|^2)^{-1/2} [\Delta(g-\sigma\Delta)\omega/2] + \oe^2 \mathcal{M}^{1}_{N_0-6}
  \in \oe \mathcal{M}^{-1}_{N_0-6} + \oe^2 \mathcal{M}^{1}_{N_0-6},
\\
\partial_t \sqrt{\lambda} & = \frac{1}{2\sqrt{|\zeta|}} \partial_t\lambda_1^{(0)} + \oe^2 \mathcal{M}^{1/2}_{N_0-6}
  \in \oe \mathcal{M}^{-1/2}_{N_0-6} + \oe^2 \mathcal{M}^{1/2}_{N_0-6},
\\
\partial_t \Sigma & = \partial_t\Sigma_1 + \oe^2 \mathcal{M}^{3/2}_{N_0-6} \in \oe \mathcal{M}^{1/2}_{N_0-6} + \oe^2 \mathcal{M}^{3/2}_{N_0-6}.
\end{split}
\end{equation}

\end{lemma}

\begin{proof} 
The bounds in (i) follow directly from definitions, the basic bounds \eqref{propsymb1}, and the bootstrap assumption $h\in \oe H^{N_0+1}$. 
The bounds in (ii) follow using also Lemma \ref{DNmainLem}. The bounds in (iii) follow using also the identity 
$\partial_th = G(h)\phi = |\nabla|\omega + \oe^2 H^{N_0-1/2}$ in \eqref{WWE}.
\end{proof}

\begin{proof}[Proof of Proposition \ref{proeqU}] 

We subdivide the proof in a few steps.

\medskip
{\it Step 1: the symmetrizing variables $(H,\Psi)$ and symbols rules}. 
To diagonalize the principal part of the system \eqref{WWpara} we define the variables
\begin{equation}
\label{HPsi}
H := T_{\sqrt{g+\ell}}h, \qquad \Psi := T_{\Sigma}T_{1/\sqrt{g+\ell}}\omega+T_{m^\prime}\omega,
\end{equation}
where $m^\prime$ is the symbol in \eqref{defscalarunk}, so that the new variable is $U = H+i\Psi$. In view of Lemma \ref{symbbou} we have the following relations between the variables
\begin{align}
 \label{unkrelations}
\begin{split}
& H = \Re(U) + \oe^2 H^{N_0}, \qquad \sqrt{g -\sigma\Delta} h = \Re(U) + \oe^2 H^{N_0},
\\
& \Psi = \Im(U )+ \oe^2 H^{N_0}, \qquad |\nabla|^{1/2}\omega = \Im(U) + \oe^2 H^{N_0}.
\end{split}
\end{align}

As a consequence, if $s=s(\xi,\eta)$ is a symbol satisfying \eqref{symbolseqU}
and we denote by $S$ the associated bilinear multiplier, then
\begin{equation}
\label{Acceptable}
S[T_1,T_2] \in \oe^2 H^{N_0+3/2}_\ast + \oe^3 H^{N_0}\,\,\text{ for any }\,\,T_1,T_2\in\{U,\bar{U},H,\Psi,(g-\sigma\Delta)^{1/2}h,|\nabla|^{1/2}\omega\}.
\end{equation}
compare the definition \ref{def1}.

If $r=r(\xi,\eta)$ is a real-valued symbol satisfying the estimate in \eqref{symbolseqUspec}
and we denote by $R$ the associated bilinear multiplier, then
\begin{equation}
\label{Acceptable2a}
\begin{split}
R[Z_1,Z_2] \in \oe^2 H^{N_0+1}_r + \oe^3 H^{N_0}
  \quad & \text{ for any } \quad Z_1 \in\{ \Im U, |\nabla|^{1/2}\omega, \Psi\} 
  \\ & \text{ and } \quad Z_2 \in \{U, \, H, \, i\Psi, \, (g-\sigma\Delta)^{1/2}h, \, i|\nabla|^{1/2}\omega\}.
\end{split}
\end{equation}

If $q=q(\xi,\eta)$ is a (general) symbol satisfying the estimate in \eqref{symbolseqUspec}
and we denote by $Q$ the associated bilinear multiplier, then
\begin{equation}
\label{Acceptable2b}
\begin{split}
Q[Z_1,\bar{U}] \in \oe^2 H^{N_0+1}_r + \oe^3 H^{N_0}
  \quad & \text{ for any } \quad Z_1 \in \{U,\, \bar{U}, \, H, \, \Psi, \, (g-\sigma\Delta)^{1/2}h, \, |\nabla|^{1/2}\omega\},
\end{split}
\end{equation}
compare the Definition \ref{def1}.
 
\medskip
{\it Step 2: The evolution equation for $H$}.
We first show that
\begin{align}
\label{EqdtH}
\begin{split}
\partial_t H - T_{\Sigma} \Psi+iT_{V\cdot\zeta}H = -T_\gamma H -\dfrac{1}{2}T_{\sqrt{g+\ell} \, \mathrm{div} V}h - 
  T_{ m^\prime\Sigma}\omega + \oe^2 H^{N_0+1}_r + \oe^2 H^{N_0+3/2}_\ast + \oe^3 H^{N_0}.
\end{split}
\end{align}
Indeed, using the first equation in \eqref{WWpara}, the identity 
$\div T_V h = \tfrac{1}{2} T_{\div V} h + iT_{V\cdot \zeta}h$ 
(see the definition in \eqref{Taf}), and the definitions in \eqref{HPsi}, 
we can arrange the terms in \eqref{EqdtH} as follows:
\begin{equation}\label{EqdtH1}
\begin{split}
\partial_tH-&T_\Sigma\Psi+iT_{V\cdot\zeta}H+T_\gamma H
 + \frac{1}{2}T_{\sqrt{g+\ell} \, \div V}h+T_{ m^\prime\Sigma}\omega
\\
&= (T_{\sqrt{g+\ell}}T_{\lambda}-T_{\Sigma}T_{\Sigma}T_{1/\sqrt{g+\ell}})\omega - (T_\Sigma T_{m^\prime}-T_{ m^\prime\Sigma})\omega
\\
& + i(T_{V\cdot\zeta}T_{\sqrt{g+\ell}}h-T_{\sqrt{g+\ell}}T_{V\cdot\zeta}h-iT_\gamma T_{\sqrt{g+\ell}}h)
\\
& -\frac{1}{2}(T_{\sqrt{g+\ell}}T_{\div V}-T_{\sqrt{g+\ell} \, \div V})h + T_{\partial_t\sqrt{g+\ell}}h+ T_{\sqrt{g+\ell}}G_2
  + T_{\sqrt{g+\ell}} \oe^3 H^{N_0+1}.
\end{split}
\end{equation}
We now look at each of the the terms in the right-hand side of the equality above.

For the first term we use Lemma \ref{PropSym} to write
\begin{align}
\nonumber
\big( T_{\sqrt{g+\ell}} &T_{\lambda} - T_{\Sigma}T_{\Sigma}  T_{1/\sqrt{g+\ell}} \big) \omega
\\
\label{EqdtH2}
& = \Big( T_{ \lambda \sqrt{g+\ell}} + \frac{i}{2} T_{ \{ \sqrt{g+\ell}, \lambda \} }
  -\big( T_{\Sigma^2/\sqrt{g+\ell}} + \frac{i}{2} T_{ \{ \Sigma^2, 1/\sqrt{g+\ell} \} } \big) \Big) \omega 
\\
\label{EqdtH3}
& + [E(\sqrt{g+\ell},\lambda)-E(\Sigma,\Sigma)T_{1/\sqrt{g+\ell}} - E(\Sigma^2,1/\sqrt{g+\ell})] \omega.
\end{align}
Since
\begin{equation*}
\begin{split}
\lambda\sqrt{g+\ell}=\Sigma^2/\sqrt{g+\ell}\,\,\text{ and }\,\,\{\sqrt{g+\ell},\lambda\}=\{\Sigma^2,1/\sqrt{g+\ell}\}\,\,\text{ on }\,\,\T^2\times(\R^2\setminus B(1/2)),
\end{split}
\end{equation*}
the expression in \eqref{EqdtH2} vanishes.
Using the expansions in \eqref{SymApprox0} and \eqref{RemBounds} in Lemma \ref{PropSym}, we see that the expression in \eqref{EqdtH3} can be rewritten as 
\begin{equation}\label{Errors1}
\begin{split}
&\Big[ E(\sqrt{g + \sigma|\zeta|^2},\lambda_1^{(0)}) + E(-\frac{\Lambda^2 h}{2\sqrt{g+\sigma |\zeta|^2}},|\zeta|)\Big]\omega- [E(\Lambda,\Sigma_1) + E(\Sigma_1,\Lambda)](g-\sigma\Delta)^{-1/2}\omega
\\
& - \Big[E(\Lambda^2, \frac{\Lambda^2 h}{2(g+\sigma|\zeta|^2)^{3/2}}) + 2E(\Lambda\Sigma_1,\frac{1}{\sqrt{g+\sigma|\zeta|^2}})\Big]\omega+ \oe^3 H^{N_0}.
\end{split}
\end{equation}
It follows from \eqref{RemBounds} that all the quadratic terms in \eqref{Errors1} are in $\oe^2 H^{N_0+3/2}$. 
To show that they are in fact in $\oe^2 H^{N_0+3/2}_\ast + \oe^3H^{N_0}$ we need to express them in terms 
of $U$ and $\overline{U}$. 
For this we use the identities \eqref{ExEab}--\eqref{ExEab2}. Using also \eqref{Acceptable} (with $T_1=(g-\sigma\Delta)^{1/2}h$ 
and $T_2=|\nabla|^{1/2}\omega$) it is easy to see that all the terms in \eqref{Errors1} are acceptable $\oe^2 H^{N_0+3/2}_\ast + \oe^3 H^{N_0}$ terms.
For example, in view of \eqref{ExEab} and \eqref{Exps}, the first term in \eqref{Errors1} is given by
\begin{equation*}
\begin{split}
&\mathcal{F}\{E(\sqrt{g + \sigma|\zeta|^2},\lambda_1^{(0)}) \omega\}(\xi)= \frac{1}{4\pi^2}\sum_{\eta\in\Z^2\setminus\{0\}} b(\xi,\eta)\mathcal{F}\{(g-\sigma\Delta)^{1/2}h\}(\xi-\eta)\widehat{|\nabla|^{1/2}\omega}(\eta),\\
&b(\xi,\eta):=\chi(\frac{\vert\xi-\eta\vert}{\vert\xi+\eta\vert})\Big[a(\xi)-a(\frac{\xi+\eta}{2})-\frac{\xi-\eta}{2}\cdot\nabla a(\frac{\xi+\eta}{2}) \Big]\frac{[|\xi|^2-|\eta|^2]^2-|\xi+\eta|^2|\xi-\eta|^2}{2|\xi+\eta|^2(g+\sigma|\xi-\eta|^2)^{1/2}|\eta|^{1/2}},
\end{split}
\end{equation*}
where $a(\zeta):=\sqrt{g + \sigma|\zeta|^2}$. 
This is acceptable, since the multiplier $b$ satisfies \eqref{symbolseqU}. The other terms can be written in a similar way. 

With $m'_1$ is as in \eqref{ExpV}, we write
\begin{equation*}
\begin{split}
(T_\Sigma T_{m^\prime}-T_{ m^\prime\Sigma})\omega&=\frac{i}{2}T_{\{\Sigma,m'\}}\omega
  + E(\Sigma,m')\omega = \frac{i}{2}T_{\{\Lambda(\zeta),m'_1\}}\omega + E(\Lambda(\zeta),m'_1)\omega + \oe^3 H^{N_0}.
\end{split}
\end{equation*}
Using the formula \eqref{ExEab} for $E(a,b)$, we see that the term $E(\Lambda(\zeta),m'_1)\omega$
is an acceptable $\oe^2 H^{N_0+3/2}_\ast + \oe^3 H^{N_0}$ term.
Using \eqref{unkrelations}, Lemma \ref{PropSym}, and the formula \eqref{ExpV} for $m'_1$, we write
\begin{align}
\label{Errors1.5}
\frac{i}{2}T_{\{\Lambda(\zeta),m'_1\}}\omega = \frac{-i}{2} T_{\nabla_\ze \Lambda(\zeta) \cdot \nabla_x m'_1} |\nabla|^{-1/2} \Im U + \oe^3 H^{N_0}
= P[\Im U, U - \bar{U}] + \oe^3 H^{N_0},
\end{align}
for a real-valued symbol $p$ satisfying the bound \eqref{symbolseqUspec}. 
According to \eqref{Acceptable2a} and \eqref{Acceptable2b} this is an acceptable term.


Next we look at the second line in the right-hand side of \eqref{EqdtH1}, and write
\begin{equation}
\label{Errors2}
\begin{split}
& i\big(T_{V\cdot\zeta}T_{\sqrt{g+\ell}}h - T_{\sqrt{g+\ell}}T_{V\cdot\zeta}h - iT_\gamma T_{\sqrt{g+\ell}}h\big) = 
  - (T_{\{V\cdot\zeta,\sqrt{g+\ell}\}} - T_{\gamma\sqrt{g+\ell}})h
\\
& + i E(V\cdot\zeta,\sqrt{g+\ell})h - i E(\sqrt{g+\ell},V\cdot\zeta)h + \frac{i}{2}T_{\{\gamma,\sqrt{g+\ell}\}}h + E(\gamma,\sqrt{g+\ell})h.
\end{split}
\end{equation}
Using \eqref{SymApprox0} and \eqref{ExpV} we calculate
\begin{align*}
\{V\cdot\zeta,\sqrt{g+\ell}\} =\frac{\zeta_j \partial_k V_j(x)\cdot \sigma\zeta_k}{\sqrt{g+\sigma|\zeta|^2}}
  +\oe^2\mathcal{M}^1_{N_0-4}
  = \frac{\sigma\zeta_j\zeta_k\cdot \partial_j\partial_k|\nabla|^{-1/2}\Im(U)(x)}{\sqrt{g+\sigma|\zeta|^2}} 
  + \oe^2\mathcal{M}^1_{N_0-4}.
\end{align*}
It follows that, see also the definition \eqref{gamma} of $\gamma$,
\begin{align*}
\{V\cdot\zeta,\sqrt{g+\ell}\} - \gamma \sqrt{g+\ell} = - \frac{g \, \gamma(x,\zeta)}{\sqrt{g+\sigma|\zeta|^2}} + \oe^2\mathcal{M}^1_{N_0-4},
\end{align*}
and therefore, also in view of \eqref{unkrelations}, $(T_{\{V\cdot\zeta,\sqrt{g+\ell}\}} - T_{\gamma\sqrt{g+\ell}})h 
\in \oe^2H^{N_0+3/2}_\ast + \oe^3H^{N_0}$. 

To verify that the terms in the second line of \eqref{Errors2} are also acceptable contributions, we first write them in the form
\begin{align*}
E(i V_1\cdot\zeta,\sqrt{g+\sigma|\zeta|^2})h - E(\sqrt{g+\sigma|\zeta|^2}, i V_1\cdot\zeta)h 
  + \frac{i}{2}T_{\{\gamma,\sqrt{g+\sigma|\zeta|^2}\}}h 
\\ + E(\gamma,\sqrt{g+\sigma|\zeta|^2})h + \oe^3H^{N_0}.
\end{align*}
Using the formula for $V_1$ in \eqref{ExpV}, the formulas \eqref{ExEab}--\eqref{ExEab2} and \eqref{Acceptable2a}, we see that
\begin{equation*}
E(i V_1\cdot\zeta,\sqrt{g+\sigma|\zeta|^2})h, \quad  E(\sqrt{g+\sigma|\zeta|^2}, iV_1\cdot\zeta)h \in \oe^2 H^{N_0+1}_r + \oe^3H^{N_0}.
\end{equation*}
Similarly, from the formula \eqref{gamma} we have $iT_{\{\gamma,\sqrt{g+\sigma|\zeta|^2}\}}h  \in \oe^2 H^{N_0+1}_r + \oe^3H^{N_0}$. Indeed, since
$i \{\gamma,\sqrt{g+\sigma|\zeta|^2}\} = i\nabla_x \gamma \cdot \sigma \frac{\zeta}{\sqrt{g+\sigma|\zeta|^2}}$ we have
\begin{align}
\label{rep1}
\begin{split}
& \mathcal{F} \big( i T_{\{\gamma,\sqrt{g+\sigma|\zeta|^2}\}}h \big)(\xi)= \frac{i \sigma}{4\pi^2} \sum_{\eta\in\mathbb{Z}^2} 
   \widetilde{\gamma}\big(\xi-\eta,\frac{\xi+\eta}{2}\big)\frac{i (\xi-\eta)\cdot (\xi+\eta)/2}{\sqrt{g+(\sigma/4)|\xi+\eta|^2}} 
  \chi\Big(\frac{\xi-\eta}{|\xi+\eta|}\Big) \widehat{h}(\eta)
  \\
  & = \frac{\sigma}{8\pi^2} \sum_{\eta\in\mathbb{Z}^2} \chi\Big(\frac{\xi-\eta}{|\xi+\eta|}\Big) 
  \frac{(|\xi|^2 - |\eta|^2)^3}{\sqrt{g+(\sigma/4)|\xi+\eta|^2}}
  \frac{1}{|\xi+\eta|^2} \frac{1}{|\xi-\eta|^{1/2}}\ \widehat{\Im U}(\xi-\eta) \widehat{h}(\eta),
\end{split}
\end{align}
having used \eqref{gamma2}. We then see that the bilinear symbol in \eqref{rep1} is real-valued. 
Since $\sqrt{g-\sigma\Delta}h = \Re U + \overline{\varepsilon}^2H^{N_0}$ (see \eqref{unkrelations}), 
the expression in \eqref{rep1} is of the claimed form. 
See also \eqref{Acceptable2a}. Notice that even though some of the symbols may appear purely imaginary at first, 
they actually have the claimed reality property, consistently with \eqref{EqdtH} and Definition \ref{def1}, 
after we express all the relevant quantities in terms of $U$ and $\Im U$.

Using \eqref{ExEab}, \eqref{gamma}, \eqref{unkrelations}, we see that, according to \eqref{Acceptable},
\begin{equation*}
E(\gamma,\sqrt{g+\sigma|\zeta|^2})h \in \oe^2 H^{N_0+3/2}_\ast + \oe^3H^{N_0}.
\end{equation*}
Therefore all the terms in the second line of \eqref{Errors2} give acceptable contributions.

For the first two terms in the last line in \eqref{EqdtH1}, we use Lemmas \ref{symbbou} and \ref{PropSym} to obtain
\begin{equation*}
\big(T_{\sqrt{g+\ell}}T_{\div V}-T_{\div V \cdot \sqrt{g+\ell}} \big)h 
  = \big(\frac{i}{2}T_{\{\sqrt{g + \sigma|\zeta|^2}, \div V_1\}}
  + E(\sqrt{g + \sigma |\zeta|^2},\div V_1) \big)h + \oe^3 H^{N_0}.
\end{equation*}
Then, the formulas \eqref{ExpV} and \eqref{Acceptable2a} 
give that $iT_{\{\sqrt{g + \sigma|\zeta|^2}, \div V_1\}}h \in \oe^2 H^{N_0+1}_r + \oe^3 H^{N_0}$ (compare with \eqref{rep1}). 
Moreover, from \eqref{ExEab} we can see that 
$E(\sqrt{g + \sigma |\zeta|^2},\div V_1)h \in  \oe^2 H^{N_0+3/2}_\ast$.

The remaining terms in \eqref{EqdtH1} are (see \eqref{SymApprox2})
\begin{align*}
& T_{\partial_t\sqrt{g+\ell}}h = T_{\frac{1}{2}\frac{\Delta(g-\sigma\Delta)\omega}{\sqrt{g + \sigma|\zeta|^2}}}h + \oe^3 H^{N_0},
\\
& T_{\sqrt{g+\ell}}G_2 = T_{\sqrt{g + \sigma|\zeta|^2}}G_2 + \oe^3 H^{N_0}, 
\\
& T_{\sqrt{g+\ell}} \oe^3 H^{N_0+1} = \oe^3 H^{N_0}.
\end{align*}
According to \eqref{Acceptable2a}-\eqref{Acceptable2b} the first expression belongs to 
$\oe^2 H^{N_0+1}_r + \oe^3 H^{N_0}$.
Using the properties of $G_2$ in \eqref{DNquadratic}-\eqref{DNquadraticsym}, 
the term in the second line above is in $\oe^2 H^{N_0+3/2}_\ast + \oe^3 H^{N_0}$.
We have therefore obtained the desired identity \eqref{EqdtH}.

\medskip
{\it Step 3: The evolution equation for $\Psi$}.
We now show that the following equation holds:
\begin{equation}
\label{EqdtPsi}
 \begin{split}
\partial_t \Psi + T_{\Sigma } H + iT_{V\cdot\zeta}\Psi 
  &= -\frac{1}{2}T_{\gamma}\Psi- T_{m^\prime(g+\ell)}h +\dfrac{1}{2}T_{\sqrt{\lambda} \, \div V}\omega\\
  &+ \oe^2 i H^{N_0+1}_r + \oe^2 H^{N_0+3/2}_\ast + \oe^3 H^{N_0}.
\end{split}
\end{equation}
Note the extra factor of $i$ in front of $H^{N_0+1}_r$, consistent with our main variable being $U=H+i\Psi$.
 
Using the definition of $\Psi$ in \eqref{HPsi}, 
the second equation in \eqref{WWpara}, and the identity $T_V\nabla\omega = i T_{V\cdot\ze}\omega-(1/2)T_{\div V}\omega$, we compute
\begin{equation}\label{EqdtPsi1}
\begin{split}
\partial_t\Psi + T_{\Sigma}H + &iT_{V\cdot\zeta}\Psi +\frac{1}{2}T_{\gamma}\Psi+T_{m^\prime(g+\ell)}h-\frac{1}{2}
  T_{\sqrt{\lambda} \div V}\omega
\\
& = (T_{\Sigma}T_{\sqrt{g+\ell}}-T_{\Sigma}T_{1/\sqrt{g+\ell}}T_{g+\ell})h+(T_{m^\prime(g+\ell)}-T_{m^\prime}T_{g+\ell})h
\\
& + i\big(T_{V\cdot\zeta}\Psi-\frac{i}{2}T_{\gamma}\Psi-(T_\Sigma T_{1/\sqrt{g+\ell}} + T_{m^\prime})T_{V\cdot\zeta}\omega\big)
\\
& + \frac{1}{2}(T_{\Sigma}T_{1/\sqrt{g+\ell}}T_{\div V}-T_{\sqrt{\lambda}\, \div V })\omega + \frac{1}{2}T_{m^\prime}T_{\div V}\omega
\\
& + [\partial_t,T_{\Sigma}T_{1/\sqrt{g+\ell}} + T_{m^\prime}]\omega + ( T_{\Sigma}T_{1/\sqrt{g+\ell}} + T_{m^\prime} ) (\Omega_2 + \oe^3 H^{N_0+1}).
\end{split}
\end{equation}
Again, we proceed to verify that all the terms in the right-hand side are acceptable remainders.

For the terms in the first line, using Lemma \ref{PropSym} and the expansions in \eqref{SymApprox0}, we have
\begin{align*}
\big(T_{\Sigma}  T_{\sqrt{g+\ell}} &-T_{\Sigma}T_{1/\sqrt{g+\ell}} T_{g+\ell} \big)h = -T_\Sigma E(1/\sqrt{g+\ell},g+\ell)h
\\
& = -T_{\varphi_{>-4}(\zeta)\Lambda(\zeta)} \big[ E(\frac{\Lambda^2h}{2(g+\sigma |\zeta|^2)^{3/2}}, g + \sigma |\zeta|^2) 
  - E( 1/\sqrt{g + \sigma |\zeta|^2}, \Lambda^2h)\big] h + \oe^3 H^{N_0}.
\end{align*}
Using also \eqref{ExEab}--\eqref{ExEab2}, this gives acceptable contributions in $\oe^2 H^{N_0+3/2}_\ast + \oe^3 H^{N_0}$. 
In addition, using Lemma \ref{symbbou}, we have
\begin{equation*}
\begin{split}
(T_{m^\prime}T_{g+\ell}-T_{m^\prime(g+\ell)})h = \frac{i}{2}T_{\{m^\prime,g+\ell\}}h + E(m^\prime,g+\ell)h 
\\ = i\sigma T_{\zeta\cdot\nabla_x m'_1}h + E(m'_1,g+\sigma|\zeta|^2)h + \oe^3 H^{N_0}.
\end{split}
\end{equation*}
From the explicit formula \eqref{ExpV} for $m'_1 \in \mathcal{M}^{-1}_{N_0-4}$, 
and in view of \eqref{Acceptable2a}-\eqref{Acceptable2b} we see that
\begin{equation*}
T_{\zeta\cdot\nabla_x m'_1}h \in \oe^2 H^{N_0+1}_r + \oe^3 H^{N_0},
\end{equation*}
as desired (compare with \eqref{rep1}).
Moreover, from \eqref{ExEab2}, \eqref{unkrelations}, and \eqref{Acceptable}
we have that $E(m'_1, g+\sigma|\zeta|^2)h$ is also an acceptable contribution in 
$\oe^2 H^{N_0+3/2}_\ast + \oe^3 H^{N_0}$.

For the terms in the second line of the right-hand side of \eqref{EqdtPsi1} 
we use Lemma \ref{symbbou} and \eqref{HPsi} to write
\begin{align*}
i \big( T_{V\cdot\zeta}\Psi- \frac{i}{2}&T_{\gamma}\Psi-(T_\Sigma T_{1/\sqrt{g+\ell}}+T_{m^\prime})T_{V\cdot\zeta}\omega \big)
\\
& = i \big(T_{V\cdot\zeta}T_\Sigma T_{1/\sqrt{g+\ell}} - T_\Sigma T_{1/\sqrt{g+\ell}}T_{V\cdot\zeta} \big)\omega
  + \frac{1}{2}T_{\gamma}T_\Sigma T_{1/\sqrt{g+\ell}}\omega + \oe^3H^{N_0}
\\
& = i \big(T_{V_1\cdot\zeta}T_{|\zeta|^{1/2}} - T_{|\zeta|^{1/2}}T_{V_1\cdot\zeta}\big)\omega
  + \frac{1}{2}T_{\gamma}T_{|\zeta|^{1/2}}\omega + \oe^3H^{N_0}.
\end{align*}
Using Lemma \ref{PropSym} this equals
\begin{align}
\label{Errors5}
\begin{split}
-T_{\{V_1\cdot\zeta, |\zeta|^{1/2}\}}\omega + \frac{1}{2} T_{\gamma |\zeta|^{1/2}}\omega 
  + i E(V_1\cdot\zeta,|\zeta|^{1/2})\omega - i E(|\zeta|^{1/2},V_1\cdot\zeta)\omega 
\\
+ \frac{i}{4} T_{\{\gamma, |\zeta|^{1/2}\}}\omega + \frac{1}{2} E(\gamma,|\zeta|^{1/2})\omega + \oe^3H^{N_0}. 
\end{split}
\end{align}
Using the definitions \eqref{ExpV} and \eqref{gamma}, we notice that for $p\in[0,1]$
\begin{equation}\label{gamprop}
\{V_1\cdot\zeta,|\zeta|^{p}\}=\gamma\cdot p|\zeta|^{p} \qquad \mbox{ on } \qquad \T^2\times(\R^2\setminus B(1/2)),
\end{equation}
and therefore the first two terms in \eqref{Errors5} cancel out.
Using \eqref{ExEab}--\eqref{ExEab2}, the formula for $V_1$ in \eqref{ExpV},
and the rules \eqref{Acceptable2a}-\eqref{Acceptable2b}, we see that
\begin{align*}
E(V_1\cdot\zeta,|\zeta|^{1/2})\omega - E(|\zeta|^{1/2},V_1\cdot\zeta)\omega \in \oe^2 H^{N_0+1}_r + \oe^3 H^{N_0},
\end{align*}
which is our desired property. Using also \eqref{gamma} we see that
\begin{align*}
i T_{\{\gamma, |\zeta|^{1/2}\}}\omega = \frac{i}{2} T_{\zeta \cdot \nabla_x \gamma \, |\zeta|^{-3/2}} |\nabla|^{-1/2}\Im U + \oe^3 H^{N_0}
  \in \oe^2 H^{N_0+1}_r + \oe^3 H^{N_0}.
\end{align*}
For the last term in \eqref{Errors5}, we use \eqref{ExEab2} 
to see that $E(\gamma,|\zeta|^{1/2})\omega \in \oe^2 H^{N_0+3/2}_\ast + \oe^3 H^{N_0}$. 
Thus all the terms in the second line of the right-hand side of \eqref{EqdtPsi1} are acceptable contributions.

Using Lemma \ref{PropSym} and the definitions, we see that the terms on the third line 
in the right-hand side of \eqref{EqdtPsi1} are acceptable cubic term in $\oe^3 H^{N_0}$.

For the terms in the last line in \eqref{EqdtPsi1}, we observe that
\begin{align}
\nonumber 
& [\partial_t,T_\Sigma T_{1/\sqrt{g+\ell}}+T_{m'}]\omega=T_{\partial_t\Sigma}T_{1/\sqrt{g+\ell}}\omega 
  + T_{\Sigma}T_{\partial_t(1/\sqrt{g+\ell})}\omega + T_{\partial_tm'}\omega
\\
\label{Errors6} 
& = T_{\partial_t\Sigma_1}T_{(g+\sigma|\zeta|^2)^{-1/2}}\omega 
  - \Lambda T_{\frac{\Delta(g-\sigma\Delta)\omega}{2(g+\sigma|\zeta|^2)^{3/2}}} \omega 
  + i T_{\frac{\partial_t(\mathrm{div}V)}{2(g + \sigma|\zeta|^2)^{1/2}}}\omega + \oe^3 H^{N_0},
\end{align}
where we used \eqref{SymApprox0} and \eqref{SymApprox2}. 
Using Lemma \ref{PropSym}, the formula for $\Sigma_1$ in \eqref{Exps}, $\partial_t h =|\nabla|\omega + \oe^2 H^{N_0-1/2}$
(see Lemmas \ref{SysparaLem} and \ref{DNmainLem}), 
and \eqref{ExEab2} we have
\begin{align*}
T_{\partial_t\Sigma_1} T_{(g+\sigma|\zeta|^2)^{-1/2}}\omega = T_{\partial_t\Sigma_1 \, (g+\sigma|\zeta|^2)^{-1/2}}\omega
  + \oe^2 H^{N_0+3/2}_\ast + \oe^3 H^{N_0}
  \\ 
  \in i \oe^2 H^{N_0+1}_r + \oe^2 H^{N_0+3/2}_\ast + \oe^3 H^{N_0}.
\end{align*}
The remaining terms in \eqref{Errors6} are easily seen to be in $\oe^2 H^{N_0+3/2}_\ast + \oe^3 H^{N_0}$, 
since $\partial_t V = -\nabla(g+\sigma|\nabla|^2)h + \oe^2 H^{N_0-2}$
(see Lemma \ref{SysparaLem} and Lemma \ref{DNmainLem}).

Finally, using the property \eqref{d_tomegaquadratic} in Lemma \ref{SysparaLem}, and the usual \eqref{unkrelations} and \eqref{Acceptable},
we have
\begin{equation*}
\begin{split}
& (T_{\Sigma}T_{1/\sqrt{g+\ell}} + T_{m^\prime})(\Omega_2) = \oe^2 H^{N_0+3/2}_\ast+ \oe^3 H^{N_0},
\\
& (T_{\Sigma}T_{1/\sqrt{g+\ell}} + T_{m^\prime})(\oe^3 H^{N_0+1}) = \oe^3 H^{N_0}.
\end{split}
\end{equation*}
Therefore, all the terms in the right-hand side of \eqref{EqdtPsi1} are acceptable, and \eqref{EqdtPsi} is verified.

\medskip
{\it Step 4}.
From the definition of $U = H+i\Psi$, see \eqref{defscalarunk} and \eqref{HPsi}, 
combining the equations \eqref{EqdtH} and \eqref{EqdtPsi}, and using \eqref{unkrelations}, we see that
\begin{align*}
\begin{split}
& \partial_t U + i T_{\Sigma} U +iT_{V\cdot\zeta}U= Q_U+ \mathcal{N}_U + \oe^2H^{N_0+1}_r + \oe^2 H^{N_0+3/2}_\ast + \oe^3 H^{N_0},
\\
& Q_U :=(-\frac{1}{2}T_{\sqrt{g+\ell}\, \div V}-iT_{m^\prime (g+\ell)})h + (-T_{m^\prime\Sigma}+\frac{i}{2}T_{\sqrt{\lambda}\, \div V})\omega \equiv 0,
\\
& \mathcal{N}_U := -T_{\gamma}H-\frac{i}{2}T_{\gamma}\Psi=-\frac{1}{4}T_{\gamma}(3U+\overline{U})+ \oe^3 H^{N_0},
\end{split}
\end{align*}
where $Q_U$ vanishes in view of our choice of $m^\prime$, see \eqref{defscalarunk}. 
Thus $\mathcal{N}_U$ has the structure claimed in the statement, according to Definition \ref{def1}.
\end{proof}

\subsection{Higher order derivatives}\label{higherder}

To derive high order energy estimates for $U$, and hence for $h$ and $|\nabla|^{1/2}\phi$,
we need to apply derivatives to the equation \eqref{eqU}.
We will consider differentiated variables of the form 
\begin{align}
 \label{defW_n}
W_n := (T_\Sigma)^n U\in\oe H^{N_0-3n/2}, \quad n\in[0,2N_0/3],
\end{align}
for $U$ as in \eqref{defscalarunk} and $\Sigma$ as in \eqref{defSigma}. 
We have the following consequence of Proposition \ref{proeqU}:

\begin{proposition}\label{proeqW} 

(i) Recalling the assumptions \eqref{Sob1}, for any $t\in [0,T]$ we have
\begin{equation}
\label{EnEquiv}
\big[{\|\langle \nabla \rangle h(t)\|}_{H^{N_0}} + \| \,|\nabla|^{1/2} \phi(t) \|_{H^{N_0}}\big]\approx \sum_{n\in[0,2N_0/3]}\|W_n(t)\|_{L^2}.
\end{equation}

(ii) With $\gamma$ as in \eqref{gamma} and $n\in[0,2N_0/3]$, we have
\begin{align}
\label{eqW}
\partial_t W_n + i T_\Sigma W_n + iT_{V\cdot\zeta}W_n =  T_\gamma (a_n W_n + b_n\overline{W_n}) + \mathcal{A}_{W_n}
  + \mathcal{B}_{W_n} + \mathcal{C}_{W_n},
\end{align}
for some numbers $a_n,b_n \in \R$ 
and $\gamma$ as in \eqref{gamma}, where:

\setlength{\leftmargini}{2em}
\begin{itemize}

\item The quadratic terms $\mathcal{A}_{W_n}$ have the form
\begin{align}
\label{eqW1}
& \mathcal{A}_{W_n} = A_{++}^n(\Im U,U) + A_{+-}^n(U,\bar{U}) + A_{--}^n(\bar{U},\bar{U}),
\end{align}
where the the symbols $a_{\iota_1\iota_2}^n$ of the bilinear operators $A_{\iota_1\iota_2}^{n}$ satisfy,
for all $\xi,\eta \in \mathbb{Z}^2$ and $(\iota_1\iota_2) \in \{ (++), (+-), (--)\}$, the bounds
\begin{align}
\label{symbolseqWspec}
\begin{split}
|a_{\iota_1\iota_2}^n(\xi,\eta)| & \lesssim \langle\eta\rangle^{3n/2-1} \langle\xi-\eta\rangle^4 \varphi_{\leq -5}(|\xi-\eta|/|\eta|),
\end{split}
\end{align}
and the additional ``reality condition'' $a_{++}^n(\xi,\eta) \in \R$.

\medskip
\item The quadratic terms $\mathcal{B}_{W_n}$ have the form
\begin{align}
\label{eqWSsemi}
\mathcal{B}_{W_n} = \sum_{\iota_1\iota_2 \in \{+,-\}} S_{\iota_1\iota_2}^{n}[U_{\iota_1},U_{\iota_2}] ,
\end{align}
where $U_+:=U$, $U_-:=\overline{U}$, and the symbols $s_{\iota_1\iota_2}^n$ of the bilinear operators $S_{\iota_1\iota_2}^{n}$ satisfy
\begin{align}
\label{eqWsymbols}
|s_{\iota_1,\iota_2}^n(\xi,\eta)| \lesssim \max(\langle\eta\rangle,\langle\xi-\eta\rangle)^{(3n/2-3/2)} \min(\langle\eta\rangle,\langle\xi-\eta\rangle)^{4}.
\end{align}
In other words, these terms gain $3/2$ derivatives over the regularity of $W_n$.

\medskip
\item The cubic terms $\mathcal{C}_{W_n}$ satisfy the bounds
\begin{align}\label{eqWrem}
\| \mathcal{C}_{W_n} \|_{H^{N_0-3n/2}} \lesssim \oe^3.
\end{align} 

\end{itemize}

\end{proposition}

This proposition is the analogue of Proposition \ref{proeqU} at a higher level of derivatives.
The quadratic terms $\mathcal{A}_{W_n}$ in \eqref{eqW} have the same structure as the quadratic terms $\mathcal{Q}_U$ in \eqref{eqUquadspec}
and, in particular, they gain $1$ derivative over the regularity of $W_n$, and satisfy the reality condition $a_{++}^n \in \R$.
The quadratic terms $\mathcal{B}_{W_n}$ in \eqref{eqWSsemi} have the same structure as the quadratic terms $\mathcal{R}_U$ in \eqref{eqUquad},
and gain $3/2$ derivatives over the regularity of $W_n$.

\medskip
\begin{proof}[Proof of Proposition \ref{proeqW}]
Recall that, see \eqref{SymApprox0},
\begin{align}
\label{proeqW1}
\begin{split}
\Sigma - \Lambda \in \oe\mathcal{M}^{3/2}_{N_0-4}\qquad 
  \Sigma - \Lambda-\Sigma_1 \in \oe^2\mathcal{M}^{3/2}_{N_0-4}.
\end{split}
\end{align}
It follows from Lemma \ref{PropSym} that, for any $n\in[0,2N_0/3]$,
\begin{equation}\label{Sigma-Lambda}
\begin{split}
&T_\Sigma^nU\in\oe H^{N_0-3n/2},\qquad T_\Sigma^nU-\Lambda^nU=\sum_{l=0}^{n-1}\Lambda^{n-1-l}(T_{\Sigma-\Lambda})T_\Sigma^lU,
\\
&\|W_n(t)-\Lambda^nU(t)\|_{L^2}\lesssim \oe\|U(t)\|_{H^{3n/2}} \quad\text{ for any }\quad t\in[0,T].
\end{split}
\end{equation}
In particular
\begin{equation*}
\sum_{n\in[0,2N_0/3]}\|W_n(t)\|_{L^2}\approx \|U(t)\|_{H^{N_0}}\quad\text{ for any }\quad t\in[0,T],
\end{equation*}
The bounds in \eqref{EnEquiv} follow using also the definition of $U$ in \eqref{defscalarunk} and the identity $\widehat{h}(0,t)=0$.

For (ii) we use induction over $n$. The case $n=0$ follows from Proposition \ref{proeqU}. 
Assuming that the conclusion holds for some $n\leq 2N_0/3-1$ and applying $T_\Sigma$, we find that
\begin{equation*}
\begin{split}
(\partial_t+iT_\Sigma +iT_{V\cdot\zeta})W_{n+1} & = T_{\gamma}(a_{n}W_{n+1} + b_{n}\overline{W_{n+1}})
  + [\partial_t,T_{\Sigma}]W_n + i[T_{V\cdot\zeta},T_{\Sigma}]W_{n}
\\
& + [T_\Sigma,T_\gamma](a_nW_n+b_n\overline{W_n})+T_\Sigma\mathcal{A}_{W_n}+T_\Sigma\mathcal{B}_{W_n}+T_\Sigma\mathcal{C}_{W_n}.
\end{split}
\end{equation*}
Let us examine the terms in the right-hand side  above.
Using \eqref{proeqW1}--\eqref{Sigma-Lambda} and \eqref{SymApprox2} we have
\begin{equation*}
[\partial_t,T_\Sigma]W_n = T_{\partial_t\Sigma_1}\Lambda^n U + \oe^3H^{N_0-3(n+1)/2},
\end{equation*}
From the definition of $\Sigma_1$ in \eqref{Exps}, $\partial_t h = |\nabla|\omega + \oe^2 H^{N_0-1/2}$,
the relations \eqref{unkrelations}, one can directly verify that
$T_{\partial_t\Sigma_1}\Lambda^n U$ is a term of the form $A_{++}^{n+1}[\Im U, U]$, 
for a real symbol $a_{++}^{n+1}$ as in \eqref{eqW1}--\eqref{symbolseqWspec}.

Using also \eqref{ExpV} and \eqref{gamprop} we have,
\begin{equation*}
\begin{split}
i [T_{V\cdot\zeta},T_{\Sigma}]W_{n}& = i [T_{V_1\cdot\zeta},T_{\Lambda}]W_{n} + \oe^3H^{N_0-3(n+1)/2}
\\
& = - \frac{3}{2} T_\gamma W_{n+1} + \mathcal{N}'(\Im U,\Lambda^nU)+\oe^3H^{N_0-3(n+1)/2},
\end{split}
\end{equation*}
where $\mathcal{N}'(\Im U,\Lambda^nU)$ is an acceptable quadratic term as in \eqref{eqW1}--\eqref{symbolseqWspec}. 
Moreover
\begin{equation}\label{proeqW5}
\begin{split}
[T_\Sigma,T_\gamma]&(a_nW_n + b_n\overline{W_n}) = [T_\Lambda,T_\gamma](a_n\Lambda^n U + b_n\Lambda^n \overline{U}) + \oe^3H^{N_0-3(n+1)/2}\\
& =-iT_{\nabla_\zeta \Lambda \cdot \nabla_x \gamma}(a_n\Lambda^n U + b_n\Lambda^n \overline{U})+ E(\Lambda,\gamma)(a_n\Lambda^n U + b_n\Lambda^n \overline{U})\\
&-E(\gamma,\Lambda)(a_n\Lambda^n U + b_n\Lambda^n \overline{U})+ \oe^3H^{N_0-3(n+1)/2}.
\end{split}
\end{equation}
Recalling the definition of $\gamma$ in \eqref{gamma}, 
we see that the first term in the right-hand side of \eqref{proeqW5} is of the desired form \eqref{eqW1}--\eqref{symbolseqWspec},
with $n$ replaced by $n+1$ (compare with \eqref{rep1}).
Moreover, from \eqref{ExEab}--\eqref{ExEab2} we see that the other quadratic terms in the right-hand \eqref{proeqW5} are also acceptable contributions
of the form \eqref{eqWSsemi}--\eqref{eqWsymbols}. 

Using again \eqref{proeqW1}--\eqref{Sigma-Lambda} we see that
\begin{equation*}
T_\Sigma\mathcal{A}_{W_n} = \Lambda \mathcal{A}_{W_n} + \oe^3H^{N_0-3(n+1)/2},\qquad T_\Sigma\mathcal{B}_{W_n} = \Lambda \mathcal{B}_{W_n} + \oe^3H^{N_0-3(n+1)/2}.
\end{equation*}
Since we also have $\oe^3 T_\Sigma H^{N_0-3n/2} \in \oe^3 H^{N_0-3(n+1)/2}$, the induction step is completed.
\end{proof}

\section{Energy estimates I: setup and control of large modulations} 
\label{SecEn1}

We turn now to the proof of the main bootstrap Proposition \ref{MainBootstrap}, which will be performed in this and the next section. 
Set
\begin{equation*}
\sigma=1\qquad\text{ and }\qquad\oe:=\KK_g\e.
\end{equation*}
The implied constants in inequalities such as $A\lesssim B$ are allowed to depend on $g$ in this section,
and may deteriorate as $g\to 0$, or $g\to\infty$,  or, more subtly, as $g$ is close to the set $\mathcal{N}$ (see the constants $c_y$ in \eqref{ml17} and $b'_y$ in \eqref{ml42}).
The constant $\KK_g$ is assumed to be taken large relative to these implied constants, and then $\e$ is taken small relative to $\KK_g^{-1}$.

\subsection{Energy identities}
The hypothesis \eqref{Sob1} holds, due to the bootstrap assumption in Proposition \ref{MainBootstrap} 
and the identities $\partial_th=G(h)\phi$ and \eqref{int0}.  We define the energy functional
\begin{equation}
\label{Efun1}
\mathcal{E}_{N_0}(t):=\frac{1}{2}\sum_{n\in[0,2N_0/3]}\|W_n(t)\|_{L^2}^2,
\end{equation}
where the functions $W_n$ are defined in \eqref{defW_n} and satisfy the equation \eqref{eqW}.
We start with a lemma describing the evolution of $\mathcal{E}_{N_0}$:

\begin{lemma}\label{ProBulk}
For any $t\in[0,T]$ we have
\begin{equation}
\label{ProBulk10}
{\|\langle \nabla \rangle h(t)\|}_{H^{N_0}}^2 + {\| \,|\nabla|^{1/2} \phi(t) \|}_{H^{N_0}}^2 \approx \mathcal{E}_{N_0}(t).
\end{equation}
Moreover
\begin{equation}
\label{ProBulk1}
\frac{d}{dt}\mathcal{E}_{N_0}(t) = \mathcal{B}_{0}(t) + \mathcal{B}_1(t) + \mathcal{B}_2(t) + \mathcal{B}_E(t)
\end{equation}
where:

\setlength{\leftmargini}{2em}
\begin{itemize}

\medskip
\item  The bulk term $\mathcal{B}_0$ has the form
\begin{equation}
\label{ProBulk2}
\mathcal{B}_0 := \Re \sum_{\iota \in \{+,-\}} \sum_{\xi,\eta\in\Z^2}
  c_\iota \mu_0(\xi,\eta)\widehat{\Im U}(\xi-\eta) \widehat{W_\iota^0}(\eta) \widehat{\overline{W^0}}(-\xi) 
\end{equation}
where $W^0:=T^{2N_0/3}_\Sigma U$, $W_+^0 = W^0$, $W^0_-=\overline{W^0}$, $c_+ \in \R$, $c_- \in \C$, and
\begin{equation}
\label{ProBulk2sym}
\begin{split}
& \mu_0(\xi,\eta) = |\xi-\eta|^{3/2}\mathfrak{d}(\xi,\eta), 
  \qquad \mathfrak{d}(\xi,\eta) := \chi\Big(\frac{|\xi-\eta|}{|\xi+\eta|}\Big) \Big(\frac{\xi-\eta}{|\xi-\eta|}\cdot\frac{\xi+\eta}{|\xi+\eta|}\Big)^2.
\end{split}
\end{equation}

\item The bulk term $\mathcal{B}_1$ can be written as
\begin{align}
\label{ProBulk3}
\begin{split}
\mathcal{B}_1 = \mathcal{B}_1^+ + \mathcal{B}_1^-,
\qquad & \mathcal{B}_1^+ := \Re \sum_{\xi,\eta\in\Z^2} \mu_1^+(\xi,\eta) \widehat{\Im U}(\xi-\eta) \widehat{W^0}(\eta) \widehat{\overline{W^0}}(-\xi),
\\
& \mathcal{B}_1^- := \Re \sum_{\iota \in \{+,-\}} \sum_{\xi,\eta\in\Z^2} 
  \mu_1^-(\xi,\eta) \what{U_{\iota}}(\xi-\eta) \what{\overline{W^0}}(\eta) \widehat{\overline{W^0}}(-\xi),
\end{split}
\end{align}
where the symbols $\mu_1^+$ and $\mu^-_1 = \mu^-_{1;U_\iota}$ satisfy
\begin{equation}
\label{ProBulk3sym}
\begin{split}
& |\mu_1^{\pm} (\xi,\eta)| \lesssim (\langle\xi\rangle+\langle\eta\rangle)^{-1} \langle \xi-\eta \rangle^6 \varphi_{\leq -6}(|\xi-\eta|/|\eta|), \qquad \mu_1^+ \in \R.
\end{split}
\end{equation}

\medskip
\item The bulk term $\mathcal{B}_2$ is a finite linear combination of the form
\begin{align}
\label{ProBulk5}
\begin{split}
& \mathcal{B}_2:= \sum_{W',W'' \in \mathcal{W}} \sum_{\iota \in \{+,-\}} \sum_{\xi,\eta\in\Z^2}
  \mu_2(\xi,\eta) \widehat{U_{\iota}}(\xi-\eta) \what{W''}(\eta) \what{\overline{W'}}(-\xi)   
\end{split}
\end{align}
where $U$ and $\Sigma$ are defined as in Proposition \ref{proeqU}, 
and $\mathcal{W} := \{ T_\Sigma^m U_{\pm} : \, m \leq 2N_0/3 \}$ and the symbols $\mu_2 = \mu_{2;(U_{\iota},W',W'')}$ satisfy
\begin{equation}
\label{ProBulk5sym}
\begin{split}
& |\mu_2 (\xi,\eta)| \lesssim (\langle\xi\rangle+\langle\eta\rangle)^{-3/2} \langle \xi-\eta \rangle^6.
\end{split}
\end{equation}

\medskip
\item $\mathcal{B}_E$ are cubic remainder terms satisfying
\begin{equation}\label{ProBulk4}
 \qquad |\mathcal{B}_E(t)| \lesssim \oe^4\qquad\text{ for any }t\in[0,T].
\end{equation}
 
\end{itemize}
\end{lemma}


Notice that the energy estimates we prove in Lemma \ref{ProBulk} are stronger than standard energy estimates, 
since the cubic terms $\mathcal{B}_0,\mathcal{B}_1$ and $\mathcal{B}_2$
are better than generic cubic terms, due to the special form of the multipliers $\mu_0,\mu_1^{\pm},\mu_2$.
Also notice that the sum in \eqref{ProBulk2} corresponding to $\iota = +$ is already real valued.
The decomposition of the different bulk terms above is dictated by the different ways in which we will handle them.

\medskip
\begin{proof}[Proof of Lemma \ref{ProBulk}]
The bounds \eqref{ProBulk10} follow from \eqref{EnEquiv}.
To prove the remaining claims we start from \eqref{eqW}. 
The symbols $\Sigma$ and $V\cdot\ze$ are real-valued. Therefore we have
\begin{equation}\label{bin3.3}
\frac{d}{dt}\frac{1}{2}\Vert W_n\Vert_{L^2}^2 = \Re\langle T_\gamma (a_n W_n + b_n\overline{W_n} ),W_n\rangle 
  + \Re \langle\mathcal{A}_{W_n},W_n\rangle + \Re \langle\mathcal{B}_{W_n},W_n\rangle + \Re\langle \mathcal{C}_{W_n},W_n\rangle,
\end{equation}
since, as a consequence of Lemma \ref{PropSym} (iv), 
\begin{equation*}
\Re\langle i T_\Sigma W_n + iT_{V\cdot\zeta}W_n,W_n\rangle = 0.
\end{equation*}
Clearly, $|\langle \mathcal{C}_{W_n},W_n\rangle|\lesssim \oe^4$, so the last term can be placed in $\mathcal{B}_E$.

Recall, see \eqref{Sigma-Lambda} that $\Lambda^n U - W_n \in \oe^ 2L^2$.
Then, using \eqref{gamma2} and the definitions, $\langle T_\gamma (c_n W_n + d_n\overline{W_n} ),W_n\rangle$ 
can be written in the Fourier space as the term $\mathcal{B}_0$ in \eqref{ProBulk2}--\eqref{ProBulk2sym} when $n = 2N_0/3$. 
When $0 \leq n \leq 2N_0/3 -1$, these terms can be absorbed into \eqref{ProBulk5}--\eqref{ProBulk5sym}, up to acceptable remainders of the form
$\mathcal{B}_E$. 

Next, we look at the terms $\langle \mathcal{A}_{W_n},W_n \rangle$. 
When $n=2N_0/3$ we use the expressions in \eqref{eqW1}--\eqref{symbolseqWspec} and write
\begin{equation*}
\qquad \mu_1^+(\xi,\eta) := \frac{a_{++}^n(\xi,\eta)}{1+\Lambda(\eta)^n} \in \R,
  \qquad \mu^-_{1,\iota}(\xi,\eta) := \frac{a_{\iota-}^n(\xi,\eta)}{1+\Lambda(\eta)^n},
\end{equation*}
to deduce the form \eqref{ProBulk3}--\eqref{ProBulk3sym} up to acceptable terms of the form $\mathcal{B}_2$ and $\mathcal{B}_E$.
The terms with $n \leq 2N_0/3-1$ are also absorbed into $\mathcal{B}_2$.

Similarly, $\langle\mathcal{B}_{W_n},W_n \rangle $ can be written in the Fourier space as part of the term $\mathcal{B}_2$ 
in \eqref{ProBulk5}--\eqref{ProBulk5sym} plus acceptable errors.
Indeed, given a symbol $s$ as in \eqref{eqWsymbols}, one can write 
\begin{equation*}
s(\xi,\eta) = \mu_2(\xi,\eta)\cdot [(1+\Lambda(\xi-\eta)^n)+(1+\Lambda(\eta)^n)],
  \qquad \mu_2(\xi,\eta):=\frac{s(\xi,\eta)}{2+\Lambda(\xi-\eta)^n+\Lambda(\eta)^n}.
\end{equation*}
The symbol $\mu_2$ satisfies the required estimate in \eqref{ProBulk5sym}. The factors $1+\Lambda(\xi-\eta)^n$ and $1+\Lambda(\eta)^n$ 
can be combined with the functions $\widehat{U_{\iota_1}}(\xi-\eta)$ and $\widehat{U_{\iota_2}}(\eta)$ respectively. 
Recalling that $\Lambda^n U - W_n \in\oe^2L^2$, see \eqref{Sigma-Lambda}, the desired representation \eqref{ProBulk5} follows, 
up to acceptable errors that can be incorporated into $\mathcal{B}_E$.
\end{proof}

\subsection{Setup and strategy}
The identity \eqref{ProBulk1} shows that the energy increment can be controlled by estimating the space-time contribution 
the bulk terms $\mathcal{B}_0,\mathcal{B}_1,\mathcal{B}_2$ and $\mathcal{B}_E$. 
In view of Lemma \ref{ProBulk}, for \eqref{al2} it suffices to prove that, for any $t\in[0,T]$
\begin{equation*}
\Big|\int_0^t\mathcal{B}_0(s)\,ds\Big| + \Big|\int_0^t\mathcal{B}_1(s)\,ds\Big|
  + \Big|\int_0^t\mathcal{B}_2(s)\,ds\Big| + \Big|\int_0^t\mathcal{B}_E(s)\,ds\Big|\lesssim \oe^2\KK_g^{-1}.
\end{equation*}
The contribution of $\mathcal{B}_E$ is easy to bound, using the estimates \eqref{al0} and $|\mathcal{B}_E(s)| \lesssim \oe^4$. 

To bound the remaining contributions, it suffices to prove the following:

\begin{proposition}\label{strategy}
For any $t\in[0,T]$ and $\iota_1,\iota_2\in\{+,-\}$ we have
\begin{equation}
\label{gar0}
\Big| \Re \int_0^t\sum_{\xi,\eta\in\Z^2}
  \mu_0(\xi,\eta)\widehat{iU_{\iota_1}}(\xi-\eta,s) \widehat{W_{\iota_2}^0}(\eta,s)\widehat{\overline{W^0}}(-\xi,s)\,ds\Big| \lesssim \oe^2 \KK_g^{-1},
\end{equation}
where $\mu_0$ as in \eqref{ProBulk2sym}. Moreover, for $\mu_1^\pm$ as in \eqref{ProBulk3sym} we have
\begin{equation}\label{gar1a}
\Big| \Re \int_0^t \sum_{\xi,\eta\in\Z^2}
  \mu_1^{+}(\xi,\eta)\widehat{iU_{\iota_1}}(\xi-\eta,s) \widehat{W^0}(\eta,s) \widehat{\overline{W^0}}(-\xi,s)\,ds\Big| \lesssim \oe^2 \KK_g^{-1},
\end{equation}
\begin{equation}\label{gar1b}
\Big|\int_0^t \sum_{\xi,\eta\in\Z^2}
  \mu_1^{-}(\xi,\eta)\widehat{U_{\iota_1}}(\xi-\eta,s) \widehat{W^0_{-}}(\eta,s) \widehat{\overline{W^0}}(-\xi,s)\,ds\Big| \lesssim \oe^2 \KK_g^{-1}.
\end{equation}
Finally, for $\mu_2$ as in \eqref{ProBulk5sym} and any 
$W'_+\in \mathcal{W}_+$, $W_{\iota_2} \in \mathcal{W}_{\iota_2}$ we have
\begin{equation}
\label{gar2}
\Big|\int_0^t\sum_{\xi,\eta\in\Z^2}
  \mu_2(\xi,\eta) \widehat{U_{\iota_1}}(\xi-\eta,s) \widehat{W_{\iota_2}}(\eta,s)\widehat{\overline{W'_+}}(-\xi,s)\,ds\Big|\lesssim \oe^2\KK_g^{-1},
\end{equation}
where
\begin{equation}\label{gar2.1}
\mathcal{W}_{\iota} := \{T_\Sigma^m U_\iota:\,m\in[0,2N_0/3]\}, \qquad \iota\in\{+,-\}.
\end{equation}
\end{proposition}

To prove \eqref{gar0}--\eqref{gar2} we often need to integrate by parts in time (the method of normal forms). 
Notice that, at the linear level, the solutions $U$ and $W$ satisfy the equations $(\partial_t+i\Lambda)U=0$ and $(\partial_t+i\Lambda)W=0$, 
where $\Lambda = \Lambda_g=\sqrt{g|\nabla|+|\nabla|^3}$. 
Thus, at the linear level, the solutions are of the form $U(t)\approx e^{-it\Lambda}U(0)$ and $W(t)\approx e^{-it\Lambda}W(0)$ and the natural time oscillation of cubic expressions like those in \eqref{gar0}--\eqref{gar2} is given by the {\it{modulation}} or {\it{phase function}} 
\begin{align}
\label{defPhipm}
\Phi_{\iota_1\iota_2}(\xi,\eta) := \Lambda(\xi) - \iota_1 \Lambda(\xi-\eta) -\iota_2 \Lambda(\eta).
\end{align}

We use these phase functions to decompose the cubic expressions in \eqref{gar0}-\eqref{gar2}. 
More precisely, for any $\iota_1,\iota_2\in\{+,-\}$ and $\mu \in\{\mu_0,\mu_1^{\iota_2},\mu_2\}$ we define the trilinear operators
\begin{equation}\label{gar3}
\begin{split}
&\mathcal{A}_{\mu;\iota_1\iota_2}^{>0}(F,G,H) := \sum_{\xi,\eta\in\Z^2}
  \mu(\xi,\eta)\widehat{F}(\xi-\eta) \widehat{G}(\eta)\widehat{\overline{H}}(-\xi)\cdot \varphi_{>0}(\Phi_{\iota_1\iota_2}(\xi,\eta)),
\\
& \mathcal{A}_{\mu;\iota_1\iota_2}^{\leq 0}(F,G,H) := \sum_{\xi,\eta\in\Z^2}
 \mu(\xi,\eta)\widehat{F}(\xi-\eta) \widehat{G}(\eta)\widehat{\overline{H}}(-\xi)\cdot \varphi_{\leq 0}(\Phi_{\iota_1\iota_2}(\xi,\eta)).
\end{split}
\end{equation}

We will prove the main bounds \eqref{gar0}-\eqref{gar2} in the rest of the paper,
by analyzing separately the different contributions. In particular, in subsection \ref{LargeMod} we deal with the contributions from large modulations $\mathcal{A}^{>0}_{\mu;\iota_1\iota_2}$ using integration by parts in time (normal forms) and symmetrization. In section \ref{SecEn2} we deal with the contributions of small modulations $\mathcal{A}_{\mu;\iota_1\iota_2}^{\leq 0}$ in two steps: first we control the contributions of large frequencies, using the smallness of the multipliers $\mu_0,\mu_1^\pm,\mu_2$, and then we control the contributions of small frequencies, using iterated normal forms and the lower bounds in Propositions \ref{prop1} and \ref{prop2}.

We will use several times the following:

\begin{remark}\label{RemFactor}
If $\xi,\eta\in\mathbb{Z}^2$ and $0<|\xi-\eta|<2^{-4}|\xi+\eta|$ and $\iota\in\{+,-\}$ then
\begin{equation}\label{Factor}
\Big(\frac{\xi-\eta}{|\xi-\eta|}\cdot\frac{\xi+\eta}{|\xi+\eta|}\Big)^2 
  \lesssim \frac{(\Lambda(\xi)-\Lambda(\eta))^2+\langle\xi-\eta\rangle^3}{(1+|\xi|+|\eta|)\langle\xi-\eta\rangle^2}
  \lesssim \frac{\Phi_{\iota+}(\xi,\eta)^2+\langle\xi-\eta\rangle^3}{(1+|\xi|+|\eta|) \langle\xi-\eta\rangle^2}.
\end{equation}
This provides an important correlation between the factor $\mathfrak{d}$
in \eqref{ProBulk2sym} and the phases $\Phi_{\iota+}$,
which we will use at various levels in the proof. 
\end{remark}

We will also often use the following simple trilinear estimate, which follows from the Cauchy--Schwarz inequality: 
if $|\nu(\xi,\eta)|\lesssim 1$ for any $\xi,\eta\in\Z^2$ then 
\begin{equation}\label{Fact2}
\Big| \sum_{\xi,\eta\in\Z^2}\nu(\xi,\eta)\widehat{F}(\xi-\eta) \widehat{G}(\eta) \widehat{\overline{H}}(-\xi)\Big| 
  \lesssim\|\widehat{F}\|_{L^1} \|\widehat{G}\|_{L^2} \|\widehat{H}\|_{L^2} \lesssim \|F\|_{H^2} \|G\|_{L^2}\|H\|_{L^2}.
\end{equation}

\subsection{Large modulations}\label{LargeMod}
In this subsection we consider the operators $\mathcal{A}_{\mu;\iota_1\iota_2}^{>0}$, see \eqref{gar3}, which correspond to large modulations. We will prove the following:

\begin{lemma}\label{gar7}
For any $\iota_1,\iota_2\in\{+,-\}$, $W'_+\in\mathcal{W}_+$, $W_{\iota_2}\in\mathcal{W}_{\iota_2}$, $\mu\in\{\mu_0,\mu_1^{\iota_2},\mu_2\}$, and $t\in[0,T]$ we have
\begin{equation}\label{gar9}
\Big| \int_0^t\mathcal{A}_{\mu;\iota_1\iota_2}^{>0}(U_{\iota_1}(s), W_{\iota_2}(s),W'_+(s))\,ds\Big| \lesssim \oe^2\KK_g^{-1}.
\end{equation}
\end{lemma}

The rest of this subsection is concerned with the proof of Lemma \ref{gar7}. Notice that the bounds \eqref{gar9} are slightly stronger than what is needed to prove \eqref{gar0}--\eqref{gar2}, in the sense that we do not use some of the symmetries in the integrals in \eqref{gar0}--\eqref{gar1b}. In other words, the case of large modulation is more robust. There are no small divisors issues, 
and we can integrate by parts in $s$ and re-symmetrize to avoid derivative loss and prove the desired estimates. 

Notice that, for any $s\in[0,T]$,  
\begin{equation}\label{gar10}
\|U_{\iota_1}(s)\|_{H^{N_0}} \lesssim \oe,\qquad (\partial_s+i\iota_1\Lambda)U_{\iota_1}(s)\in\oe^2 H^{N_0-2},
\end{equation}
This follows from Proposition \ref{proeqU}. Moreover, the functions $W\in\mathcal{W}_+$ satisfy the bounds 
\begin{equation}\label{gar11}
{\| W(s) \|}_{L^2}\lesssim \oe
\end{equation}
and satisfy equations of the form (see Proposition \ref{proeqW}) 
\begin{equation}\label{gar12}
(\partial_t + i\Lambda)W = - iT_{V\cdot\zeta}W -iT_{\Sigma - \Lambda}W + \mathcal{E}_W,
\end{equation}
for some term $\mathcal{E}_W$ satisfying the bound in \eqref{gar12.5} below.
Moreover, using Lemma \ref{PropSym} (ii) and Lemma \ref{symbbou} we see that, for all $k \in \Z$ and $s\in[0,T]$,
\begin{align}
\label{gar12.5}
\begin{split}
\|\mathcal{E}_W(s)\|_{L^2}&\lesssim\oe^2,\\
{\| (P_k T_{V\cdot\zeta}W)(s) \|}_{L^2} & \lesssim \oe 2^{k} {\|P_{[k-2,k+2]}W(s)\|}_{L^2},
\\
{\| (P_k T_{\Sigma-\Lambda-\Sigma_1}W)(s) \|}_{L^2} & \lesssim \oe^2 2^{3k/2} {\|P_{[k-2,k+2]}W(s)\|}_{L^2},\\
{\| (P_k T_{\Sigma_1}W)(s) \|}_{L^2} & \lesssim \oe 2^{k/2} {\|P_{[k-2,k+2]}W(s)\|}_{L^2}.
\end{split}
\end{align}

To integrate by parts in $s$, for $\mu\in\{\mu_0,\mu_1^{\iota_2},\mu_2\}$ we define the trilinear operators ,  
\begin{align}
\label{LMnot1}
& \mathcal{T}_{\mu;\iota_1\iota_2}^{>0}(F,G,H):=\sum_{\xi,\eta\in\Z^2}\mu(\xi,\eta)\widehat{F}(\xi-\eta) \widehat{G}(\eta)\widehat{\overline{H}}(-\xi)\cdot  
  \frac{\varphi_{> 0}(\Phi_{\iota_1\iota_2}(\xi,\eta))}{i\Phi_{\iota_1\iota_2}(\xi,\eta)}.
\end{align}
Note how the operators in \eqref{LMnot1} are related to the first formula in \eqref{gar3},
and are obtained diving the symbol in \eqref{gar3} by the phase function $\Phi_{\iota_1\iota_2}$ in \eqref{defPhipm}.
We will use these types of operators several times below in connection to normal forms transformations, 
see the identity \eqref{LM1}.

For simplicity of notation, in the rest of the paper we remove the subscripts ${\iota_1\iota_2}$ 
in the operators $\mathcal{T}_{\mu;\iota_1\iota_2}^{>0}$ and $\mathcal{A}_{\mu;\iota_1\iota_2}^{>0}$, 
and in the phase functions $\Phi_{\iota_1\iota_2}$. 
The bounds \eqref{gar9} are harder in the case $\mu=\mu_0$ because we need additional symmetrization arguments.

\begin{proof}[Proof of \eqref{gar9} when $\mu=\mu_0$]
{\bf{Step 1.}} Our starting point is the normal form transformation formula
\begin{align}
\label{LM1}
\begin{split}
&\int_0^t\mathcal{A}_{\mu_0}^{>0}(U_{\iota_1}(s),W_{\iota_2}(s),W'_+(s))\,ds 
  = - \int_0^t \mathcal{T}_{\mu_0}^{>0}((\partial_s+i\Lambda_{\iota_1})U_{\iota_1}(s),W_{\iota_2}(s),W'_+(s))\,ds
\\
& - \int_0^t \big[\mathcal{T}_{\mu_0}^{>0}(U_{\iota_1}(s),(\partial_s+i\Lambda_{\iota_2})W_{\iota_2}(s),W'_+(s))
  + \mathcal{T}_{\mu_0}^{>0}(U_{\iota_1}(s),W_{\iota_2}(s),(\partial_s+i\Lambda)W'_+(s))\big]\,ds
\\
& + \mathcal{T}_{\mu_0}^{>0}(U_{\iota_1}(t),W_{\iota_2}(t),W'_+(t))-\mathcal{T}_{\mu_0}^{>0}(U_{\iota_1}(0),W_{\iota_2}(0),W'_+(0)),
\end{split}
\end{align} 
where $\Lambda_+= \Lambda$, $\Lambda_-:=-\Lambda$. 
\eqref{LM1} follows by writing
\begin{equation*}
\mathcal{T}(U_{\iota_1}(t),W_{\iota_2}(t),W'_+(t))-\mathcal{T}(U_{\iota_1}(0),W_{\iota_2}(0),W'_+(0))=\int_0^t\frac{d}{ds}\{\mathcal{T}(U_{\iota_1}(s),W_{\iota_2}(s),W'_+(s))\}\,ds,
\end{equation*}
where $\mathcal{T}=\mathcal{T}_{\mu_0}^{>0}$ and expanding suitably the terms in the right-hand side
using the definitions of the operators \eqref{gar3} and \eqref{LMnot1}, and of the phase \eqref{defPhipm}.

All the terms in \eqref{LM1} except the ones in the second line in the case $\iota_2=+$ are easy to estimate. It follows from \eqref{Fact2} that
\begin{equation}\label{LM1.5}
|\mathcal{T}_{\mu_0}^{>0}(F,G,H)|\lesssim\|F\|_{H^4}\|G\|_{L^2}\|H\|_{L^2}
\end{equation}
for any functions $F,G,H$. Then, using just \eqref{gar10}--\eqref{gar11}, we can bound
\begin{align}\label{LM1.6}
|\mathcal{T}_{\mu_0}^{>0}(U_{\iota_1}(s),W_{\iota_2}(s),W'_+(s))|\lesssim\oe^3
\end{align} 
for any $s\in[0,T]$, and
\begin{align}\label{LM1.7}
\Big|\int_0^t \mathcal{T}_{\mu_0}^{>0}(&(\partial_s+i\Lambda_{\iota_1})U_{\iota_1}(s),W_{\iota_2}(s),W'_+(s))\,ds\Big|\lesssim T_\e\oe^4.
\end{align} 
These bounds suffice, due to the assumption \eqref{al0}. 

Similarly, if $\iota_2=-$ then $\Phi(\xi,\eta)=\Lambda(\xi)+\Lambda(\eta)-\iota_1\Lambda(\xi-\eta)$ and we have the bound
\begin{equation*}
\Big|\frac{\varphi_{> 0}(\Phi(\xi,\eta))}{i\Phi(\xi,\eta)}\Big|\lesssim \frac{(1+|\xi-\eta|)^{3/2}}{(1+|\xi|+|\eta|)^{3/2}}.
\end{equation*}
It follows from \eqref{gar12}--\eqref{gar12.5} that $\|\langle\nabla\rangle^{-3/2}(\partial_s+i\Lambda)W(s)\|_{L^2}\lesssim\oe^2$ for any $s\in[0,T]$ and $W\in\mathcal{W}_+$. Therefore, as before, we can estimate the absolute value of the term in the second line of \eqref{LM1} by $T_\e\oe^4$, which suffices.

{\bf{Step 2.}} To bound the remaining terms, which are 
\begin{align*}
\begin{split}
Y(s):= \mathcal{T}_{\mu_0}^{>0}(U_{\iota_1}(s),(\partial_s+i\Lambda)W_{+}(s),W'_+(s))+\mathcal{T}_{\mu_0}^{>0}(U_{\iota_1}(s),W_{+}(s),(\partial_s+i\Lambda)W'_+(s)),
\end{split}
\end{align*} 
we need an additional symmetrization argument to avoid loss of derivative. Using \eqref{gar12}, we can write $Y = Y_1 + Y_2 + Y_3$ where
\begin{align}
\label{LM21}
\begin{split}
Y_1(s) & := -i\big[\mathcal{T}_{\mu_0}^{>0}(U_{\iota_1}(s),T_{\Sigma-\Lambda}W_{+}(s),W'_+(s))-\mathcal{T}_{\mu_0}^{>0}(U_{\iota_1}(s),W_{+}(s),T_{\Sigma-\Lambda}W'_+(s))\big] ,\\
Y_2(s) &: = -i\big[\mathcal{T}_{\mu_0}^{>0}(U_{\iota_1}(s),T_{V\cdot\ze}W_{+}(s),W'_+(s))-\mathcal{T}_{\mu_0}^{>0}(U_{\iota_1}(s),W_{+}(s),T_{V\cdot\ze}W'_+(s))\big] ,\\
Y_3(s) & :=\mathcal{T}_{\mu_0}^{>0}(U_{\iota_1}(s),\mathcal{E}_{W_{+}}(s),W'_+(s))+\mathcal{T}_{\mu_0}^{>0}(U_{\iota_1}(s),W_{+}(s),\mathcal{E}_{W'_+}(s)) .
\end{split}
\end{align} 
Using again \eqref{LM1.5}, \eqref{gar10}, and \eqref{gar12} it is easy to see that
\begin{equation*}
\int_0^t|Y_3(s)|\,ds\lesssim T_\varep\oe^4, 
\end{equation*}
which is acceptable. After these reductions, for \eqref{gar9} it suffices to prove that
\begin{equation}
\label{gar30.0}
\big|\mathcal{T}_{\mu_0}^{>0}(U_{\iota_1}(s),T_{\sigma}W(s),W'(s))-\mathcal{T}_{\mu_0}^{>0}(U_{\iota_1}(s),W(s),T_{\sigma}W'(s))\big|\lesssim\oe^4
\end{equation}
for any $s\in[0,T]$ and $\sigma\in\{\Sigma-\Lambda,V\cdot\ze\}$, where, for simplicity of notation, $W:=W_+$, $W'=W'_+$.

Using the definition \eqref{Taf} we have
\begin{equation*}
\begin{split}
\mathcal{T}_{\mu_0}^{>0}(U_{\iota_1},T_{\sigma}W,W')=\frac{1}{4\pi^2}&\sum_{\xi,\eta,\rho\in\Z^2}\mu_0(\xi,\eta)\frac{\varphi_{> 0}(\Phi(\xi,\eta))}{i\Phi(\xi,\eta)}\widehat{U_{\iota_1}}(\xi-\eta)\\
&\times\chi\Big(\frac{|\eta-\rho|}{|\eta+\rho|}\Big)\widetilde{\sigma}\Big(\eta-\rho,\frac{\eta+\rho}{2}\Big)\widehat{W}(\rho)\overline{\widehat{W'}}(\xi)
\end{split}
\end{equation*}
and
\begin{equation*}
\begin{split}
\mathcal{T}_{\mu_0}^{>0}(U_{\iota_1},W,T_{\sigma}W')=\frac{1}{4\pi^2}&\sum_{\xi,\eta,\rho\in\Z^2}\mu_0(\xi,\eta)\frac{\varphi_{> 0}(\Phi(\xi,\eta))}{i\Phi(\xi,\eta)}\widehat{U_{\iota_1}}(\xi-\eta)\\
&\times \widehat{W}(\eta)\chi\Big(\frac{|\xi-\rho|}{|\xi+\rho|}\Big)\overline{\widetilde{\sigma}}\Big(\xi-\rho,\frac{\xi+\rho}{2}\Big)\overline{\widehat{W'}}(\rho).
\end{split}
\end{equation*}
Since $\sigma$ is real-valued we have $\overline{\widetilde{\sigma}}\Big(\xi-\rho,\frac{\xi+\rho}{2}\Big)=\widetilde{\sigma}\Big(-\xi+\rho,\frac{\xi+\rho}{2}\Big)$. Therefore, after changes of variables we have
\begin{equation}\label{LM26}
\begin{split}
\mathcal{T}_{\mu_0}^{>0}(U_{\iota_1},T_{\sigma}W,W')&-\mathcal{T}_{\mu_0}^{>0}(U_{\iota_1},W,T_{\sigma}W')\\
&=\frac{1}{4\pi^2} \sum_{\xi,\eta,\rho\in\Z^2} H_\sigma(\xi,\eta,\rho) 
  \, \widehat{U_{\iota_1}}(\xi-\eta-\rho) \what{W}(\eta) \bar{\what{W'}}(\xi)
\end{split}
\end{equation} 
where
\begin{align}
\label{M}
\begin{split}
H_\sigma(\xi,\eta,\rho) :=  \mu_0(\xi,\rho+\eta)\frac{\varphi_{>0}(\Phi(\xi,\rho+\eta))}{i\Phi(\xi,\rho+\eta)} \,
  \, \wt{\sigma}\big(\rho,\frac{2\eta+\rho}{2}\big) \chi\big(\frac{|\rho|}{|2\eta+\rho|}\big)
\\
- \mu_0(\xi-\rho,\eta)\frac{\varphi_{>0}(\Phi(\xi-\rho,\eta))}{i\Phi(\xi-\rho,\eta)} \,
   \, \wt{\sigma}\big(\rho,\frac{2\xi-\rho}{2}\big) \chi\big(\frac{|\rho|}{|2\xi-\rho|}\big).
\end{split}
\end{align}
Since $|\widehat{U_{\iota_1}}(\xi-\eta-\rho)|\lesssim\oe\langle\xi-\eta-\rho\rangle^{-12}$ (see \eqref{gar10}), 
and using also \eqref{gar11}, for \eqref{gar30.0} it suffices to prove that, for any $\xi,\eta\in\Z^2$ with $|\xi-\eta|\leq 2^{-10}(1+|\xi|+|\eta|)$ we have
\begin{equation}\label{gar30}
\sum_{\rho\in\Z^2,\,|\rho|\leq 2^{-10}(1+|\xi|+|\eta|)} \frac{|H_\sigma(\xi,\eta,\rho)|}{\langle\xi-\eta-\rho\rangle^{12}} \lesssim \oe\langle\xi-\eta\rangle^{-3}.
\end{equation}

{\bf{Step 3.}} To prove \eqref{gar30} we define the symbol
\begin{align}
\label{msigma}
\begin{split}
s_\sigma(\xi,\eta,\rho; a, b) := \mu_0(\xi-a,\rho+\eta-a) 
  & \frac{\varphi_{>0}(\Phi(\xi-a,\rho+\eta-a))}{i\Phi(\xi-a,\rho+\eta-a)}
  \\
  &\times\wt{\sigma}\big(\rho,\eta+\frac{\rho}{2}+b\big) \chi\big(\frac{\rho}{|2\eta + \rho + 2b|}\big),
\end{split}
\end{align}
and notice that
\begin{align}\label{gar31}
H_\sigma(\xi,\eta,\rho) =  s_\sigma(\xi,\eta,\rho;0,0) - s_\sigma(\xi,\eta,\rho; \rho, \xi-\rho-\eta).
\end{align}

Recalling that $\sigma\in\oe\mathcal{M}^{3/2}_{N_0-4}$ and letting $Y:=1+|\xi|+|\eta|$, we bound
\begin{equation*}
\begin{split}
|s_\sigma(\xi,\eta,\rho;\rho,0)&- s_\sigma(\xi,\eta,\rho; \rho, \xi-\rho-\eta)|
\\
& \lesssim \mu_0(\xi-\rho,\eta) \frac{\varphi_{>0}(\Phi(\xi-\rho,\eta))}{|\Phi(\xi-\rho,\eta)|}
  \cdot \oe\langle\rho\rangle^{-12}\langle\xi-\eta-\rho\rangle Y^{1/2}.
\end{split}
\end{equation*}
To estimate this expression we recall the definition of $\mu_0$ in \eqref{ProBulk2sym}, 
use the estimate \eqref{Factor} for the factor $\mathfrak{d}$, 
and notice that $|\Phi(\xi-\rho,\eta)|\lesssim\langle\xi-\eta-\rho\rangle Y^{1/2}$ for $|\rho|\leq 2^{-10}Y$. Thus
\begin{align}
\label{gar37}
\begin{split} 
& |s_\sigma(\xi,\eta,\rho;\rho,0) - s_\sigma(\xi,\eta,\rho; \rho, \xi-\rho-\eta)| 
  \\ 
  & \lesssim \frac{\Phi(\xi-\rho,\eta)^2 + \langle \xi-\eta-\rho\rangle^3}{Y \, \langle \xi-\eta-\rho\rangle^2} 
  \frac{\varphi_{>0}(\Phi(\xi-\rho,\eta))}{|\Phi(\xi-\rho,\eta)|}
  \cdot \oe\langle\rho\rangle^{-12}\langle\xi-\eta-\rho\rangle Y^{1/2}
\\
  & \lesssim \oe\langle\xi-\eta-\rho\rangle^{4}\langle\rho\rangle^{-12},
\end{split}
\end{align}
for any $\xi,\eta,\rho\in\mathbb{Z}^2$ with $|\xi-\eta|+|\rho|\leq 2^{-8}(1+|\xi|+|\eta|)$.

To bound $|s_\sigma(\xi,\eta,\rho;0,0) - s_\sigma(\xi,\eta,\rho; \rho,0)|$ we estimate first
\begin{equation}\label{gar38}
\begin{split}
&|s_\sigma(\xi,\eta,\rho;0,0) - s_\sigma(\xi,\eta,\rho; \rho,0)|\\
&\lesssim \oe\langle\rho\rangle^{-18}Y^{3/2}\Big|\mu_0(\xi,\rho+\eta)\frac{\varphi_{>0}(\Phi(\xi,\rho+\eta))}{\Phi(\xi,\rho+\eta)}-\mu_0(\xi-\rho,\eta)\frac{\varphi_{>0}(\Phi(\xi-\rho,\eta))}{\Phi(\xi-\rho,\eta)}\Big|.
\end{split}
\end{equation}
For $\alpha\in[0,1]$ let
\begin{equation}\label{gar39}
\begin{split}
G(\alpha)&:=\mu_0(\xi-\rho+\al\rho,\eta+\al\rho)\frac{\varphi_{>0}(\Phi(\xi-\rho+\al\rho,\eta+\al\rho))}{\Phi(\xi-\rho+\al\rho,\eta+\al\rho)}\\
&=c|\xi-\eta-\rho|^{3/2}\chi\Big(\frac{|\xi-\eta-\rho|}{|\xi+\eta-\rho+2\al\rho|}\Big)\Big(\frac{\xi-\eta-\rho}{|\xi-\eta-\rho|}\cdot\frac{\xi+\eta-\rho+2\al\rho}{|\xi+\eta-\rho+2\al\rho|}\Big)^2\cdot\frac{\varphi_{>0}(P(\alpha))}{P(\alpha)}
\end{split}
\end{equation}
where $P(\alpha)=P_{\xi,\eta,\rho}(\alpha):=\Lambda(\xi-\rho+\al\rho)-\Lambda(\eta+\al\rho)-\iota_1\Lambda(\xi-\eta-\rho)$, and the identity follows from definitions. Notice that, if $|\xi-\eta|+|\rho|\leq 2^{-8}(1+|\xi|+|\eta|)$ and $\xi-\eta-\rho\neq 0$ then
\begin{equation*}
|\partial_\al P(\alpha)|\lesssim\langle\rho\rangle\langle\xi-\eta-\rho\rangle(1+|\xi|+|\eta|)^{-1/2},
\end{equation*}
\begin{equation*}
\Big|\frac{\xi-\eta-\rho}{|\xi-\eta-\rho|}\cdot\frac{\xi+\eta-\rho+2\al\rho}{|\xi+\eta-\rho+2\al\rho|}\Big|\lesssim \frac{P(\al)+\langle\xi-\eta-\rho\rangle^{3/2}}{\langle\xi-\eta-\rho\rangle(1+|\xi|+|\eta|)^{1/2}},
\end{equation*}
for all $\al\in[0,1]$ (see also \eqref{Factor}).  Therefore we can take the $\alpha$ derivative in \eqref{gar39} and estimate
\begin{equation*}
|\partial_\al G(\alpha)|\lesssim \langle\rho\rangle^2\langle\xi-\eta-\rho\rangle^4(1+|\xi|+|\eta|)^{-3/2},
\end{equation*}
for all $\al\in[0,1]$. Using \eqref{gar38}, we have 
\begin{equation*}
|s_\sigma(\xi,\eta,\rho;0,0) - s_\sigma(\xi,\eta,\rho; \rho,0)|\lesssim\oe\langle\xi-\eta-\rho\rangle^{4}\langle\rho\rangle^{-12}.
\end{equation*}
Thus $|H_\sigma(\xi,\eta,\rho)|\lesssim\oe\langle\xi-\eta-\rho\rangle^{4}\langle\rho\rangle^{-12}$, using also \eqref{gar37} and \eqref{gar31}. The desired conclusion \eqref{gar30} follows. This completes the proof of \eqref{gar9}.
\end{proof}

\begin{proof}[Proof of \eqref{gar9} when $\mu\in\{\mu_1^{\iota_2},\mu_2\}$] 
With $2^K = \oe T_\e \KK_g$, 
we estimate first the contribution of high frequencies using the symbol
bounds \eqref{ProBulk3sym}, \eqref{ProBulk5sym}, and \eqref{gar10}--\eqref{gar11}
\begin{equation}\label{gar39.5}
\begin{split}
\Big|\int_0^t\mathcal{A}_{\mu}^{>0}&(U_{\iota_1}(s),W_{\iota_2}(s),P_{>K}W'_+(s))\,ds\Big|\\
&\lesssim T2^{-K}\sup_{s\in[0,T]}\big[\|U_{\iota_1}(s)\|_{H^{10}}\|W_{\iota_2}(s)\|_{L^2}\|P_{>K}W'_+(s))\|_{L^2}\big]\lesssim\oe^2\KK_g^{-1},
\end{split}
\end{equation}
as desired. For the low frequencies we integrate by parts in $s$, as in \eqref{LM1},
\begin{align*}
\begin{split}
\int_0^t\mathcal{A}_{\mu}^{>0}(&U_{\iota_1}(s),W_{\iota_2}(s),P_{\leq K}W'_+(s))\,ds  =- \int_0^t \mathcal{T}_{\mu}^{>0}((\partial_s+i\Lambda_{\iota_1})U_{\iota_1}(s),W_{\iota_2}(s),P_{\leq K}W'_+(s))\,ds\\
&- \int_0^t \mathcal{T}_{\mu}^{>0}(U_{\iota_1}(s),(\partial_s+i\Lambda_{\iota_2})W_{\iota_2}(s),P_{\leq K}W'_+(s))\,ds\\
&- \int_0^t \mathcal{T}_{\mu}^{>0}(U_{\iota_1}(s),W_{\iota_2}(s),(\partial_s+i\Lambda)P_{\leq K}W'_+(s))\,ds\\
&+\mathcal{T}_{\mu}^{>0}(U_{\iota_1}(t),W_{\iota_2}(t),P_{\leq K}W'_+(t))-\mathcal{T}_{\mu}^{>0}(U_{\iota_1}(0),W_{\iota_2}(0),P_{\leq K}W'_+(0)).
\end{split}
\end{align*}

We examine the terms in the right-hand side of this identity. As in \eqref{LM1.5}--\eqref{LM1.7},
\begin{align*}
\begin{split}
\Big|&\int_0^t \mathcal{T}_{\mu}^{>0}((\partial_s+i\Lambda_{\iota_1})U_{\iota_1}(s),W_{\iota_2}(s),P_{\leq K}W'_+(s))\,ds\Big|\\
&+|\mathcal{T}_{\mu}^{>0}(U_{\iota_1}(t),W_{\iota_2}(t),P_{\leq K}W'_+(t))|+|\mathcal{T}_{\mu}^{>0}(U_{\iota_1}(0),W_{\iota_2}(0),P_{\leq K}W'_+(0))|\lesssim T_\e\oe^4.
\end{split}
\end{align*}
Moreover, using \eqref{gar11}--\eqref{gar12.5} and $2^K\oe\lesssim 1$, 
we have $\|\langle\nabla\rangle^{-1}(\partial_s+i\Lambda)P_{\leq K}W'_+(s))\|_{L^2}\lesssim\oe^2$.
Thus, using again \eqref{ProBulk3sym} and \eqref{ProBulk5sym}, we get
\begin{align*}
\Big|\int_0^t \mathcal{T}_{\mu}^{>0}(U_{\iota_1}(s),W_{\iota_2}(s),(\partial_s+i\Lambda)P_{\leq K}W'_+(s))\,ds
\lesssim T_\e\oe^4.
\end{align*}
Similarly, since 
$\|\langle\nabla\rangle^{-1}(\partial_s+i\Lambda_{\iota_2})P_{\leq K+4}W_{\iota_2}(s))\|_{L^2}
\lesssim\oe^2$, we also have
\begin{align*}
\Big|\int_0^t \mathcal{T}_{\mu}^{>0}(U_{\iota_1}(s),(\partial_s+i\Lambda_{\iota_2})
  P_{\leq K+4}W_{\iota_2}(s),P_{\leq K}W'_+(s))\,ds\Big|\lesssim T_\e\oe^4.
\end{align*}
Finally, notice that $\mathcal{T}_{\mu}^{>0}(P_{\leq K}U_{\iota_1}(s),
  (\partial_s+i\Lambda_{\iota_2})P_{>K+4}W_{\iota_2}(s),P_{\leq K}W'_+(s))=0$. 
Thus 
\begin{equation*}
\begin{split}
\Big|&\int_0^t \mathcal{T}_{\mu}^{>0}(U_{\iota_1}(s),(\partial_s+i\Lambda_{\iota_2})P_{>K+4}W_{\iota_2}(s),
  P_{\leq K}W'_+(s))\,ds\Big|
\\
& \lesssim T_\e\sup_{s\in[0,T]}\big[\|\langle\nabla\rangle^{10}P_{>K}U_{\iota_1}(s)\|_{L^2}
  \|\langle\nabla\rangle^{-2}
(\partial_s+i\Lambda_{\iota_2})P_{>K+4}W_{\iota_2}(s)\|_{L^2}\|P_{\leq K}W'_+(s))\|_{L^2}\big]
\\
& \lesssim T_\e\oe^4.
\end{split}
\end{equation*}

It follows from the identities and inequalities above that 
\begin{equation}\label{gar40}
\Big|\int_0^t\mathcal{A}_{\mu}^{>0}(U_{\iota_1}(s),W_{\iota_2}(s),P_{\leq K}W'_+(s))\,ds\Big|\lesssim T_\e\oe^4.
\end{equation}
The desired bounds \eqref{gar9} follow using also \eqref{gar39.5}.
\end{proof}

\section{Energy estimates II: special structures and small modulations}\label{SecEn2} 

In this section we consider the remaining contributions corresponding to small modulations. 
This case is more difficult because normal forms lead to small denominators, 
which require more subtle arguments. 
In particular, we will need to use the precise form of the desired estimates \eqref{gar0}--\eqref{gar2}, 
the assumptions on the multipliers $\mu_0,\mu_1^{\pm},\mu_2$ in Lemma \ref{ProBulk}, 
and the generic lower bounds in Propositions \ref{prop1} and \ref{prop2}.

\subsection{Small modulations and high frequencies}\label{SecEn13} 
We start with the contribution of high frequencies and prove the following:

\begin{lemma}\label{gar45} Assume that $\iota_1,\iota_2\in\{+,-\}$, $W'_+\in\mathcal{W}_+$, $W_{\iota_2}\in\mathcal{W}_{\iota_2}$, and $t\in[0,T]$.

(i) Assume that $\mu\in\{\mu_0,\mu_1^{\iota_2}\}$ and let $K$ be such that
\begin{equation}\label{gar40.4}
2^K := \oe T_\e \KK_g=\oe^{-2/3}\KK_g^{-4/3}[\log(\KK_g/\oe)]^{-2}. 
\end{equation}
Then
\begin{equation}\label{gar46.5}
\Big|\int_0^t\mathcal{A}_{\mu}^{\leq 0}(U_{\iota_1}(s),W_{\iota_2}(s),P_{>K} W'_+(s))\,ds\Big|\lesssim \oe^2\KK_g^{-1},
\end{equation}
\begin{equation}\label{gar46.51}
\Big|\int_0^t\mathcal{A}_{\mu}^{\leq 0}(U_{\iota_1}(s),P_{>K+4}W_{\iota_2}(s),P_{\leq K} W'_+(s))\,ds\Big|\lesssim \oe^2\KK_g^{-1}.
\end{equation}

(ii) Assume that 
\begin{equation}
\label{gar20}
2^{3J/2} := \oe T_\e \KK_g=\oe^{-2/3}\KK_g^{-4/3}[\log(\KK_g/\oe)]^{-2}.
\end{equation}
Then
\begin{equation}
\label{gar21}
\Big|\int_0^t \mathcal{A}_{\mu_2}^{\leq 0}(U_{\iota_1}(s),W_{\iota_2}(s),P_{> J}W'_+(s)) \Big|\lesssim \oe^2\KK_g^{-1},
\end{equation}
\begin{equation}\label{gar21.1}
\Big|\int_0^t \mathcal{A}_{\mu_2}^{\leq 0}(U_{\iota_1}(s),P_{>J+4}W_{\iota_2}(s),P_{\leq J}W'_+(s)) \Big|\lesssim \oe^2\KK_g^{-1}.
\end{equation}
\end{lemma}

\begin{proof} 
We start with the proof of \eqref{gar46.5} when $\mu=\mu_0$. If $\iota_2=+$ then the symbol $\mu_0$ also gains a derivative in the support of the operator,
\begin{equation}\label{gar46}
|\mu_0(\xi,\eta)|\lesssim\langle\xi-\eta\rangle^{4}(1+|\xi|+|\eta|)^{-1}\quad\text{ if }\quad|\Phi_{\iota_1+}(\xi,\eta)|\lesssim 1. 
\end{equation}
This follows from \eqref{Factor}. Thus, using \eqref{gar10}--\eqref{gar11}, for any $s\in[0,t]$,
\begin{equation*}
|\mathcal{A}_{\mu_0}^{\leq 0}(U_{\iota_1}(s),W_+(s),P_{>K}W'_+(s))|\lesssim 2^{-K}\|U_{\iota_1}(s)\|_{H^{12}}\|W_+(s)\|_{L^2}\|P_{>K}W'_+(s)\|_{L^2}
  \lesssim 2^{-K}\oe^3.
\end{equation*}
Since $2^K=\oe T_\e\KK_g$, this suffices to prove \eqref{gar46.5} when $\iota_2=+$.

On the other hand, if $\iota_2=-$ then the operators $\mathcal{A}_{\mu_0}^{\leq 0}$ are nontrivial only when the frequency $|\xi-\eta|$ is large,
i.e. $\mathcal{A}_{\mu_0}^{\leq 0}(P_{\leq K-4}U_{\iota_1}(s),W_-(s),P_{>K}W'_+(s))=0$. Thus, for any $s\in[0,t]$,
\begin{equation*}
\begin{split}
|\mathcal{A}_{\mu_0}^{\leq 0}(U_{\iota_1}(s),W_-(s),P_{>K}W'_+(s))|\lesssim \|P_{>K-4}U_{\iota_1}(s)\|_{H^{12}}\|W_+(s)\|_{L^2}\|P_{>K}W'_+(s)\|_{L^2}
  \lesssim 2^{-K}\oe^3,
\end{split}
\end{equation*} 
using again \eqref{gar10}--\eqref{gar11}. The desired bounds \eqref{gar46.5} follow when $\iota_2=-$ as well.

The bounds \eqref{gar46.5} follow in a similar way if $\mu=\mu_1^{\iota_2}$. In fact, this case is easier because symbol bounds similar to \eqref{gar46} hold (due to the assumptions \eqref{ProBulk3sym}), and one does not need to consider two different cases. The bounds \eqref{gar21} hold as well, by the same argument and using the assumptions \eqref{ProBulk5sym}.

To prove \eqref{gar46.51} we notice that for $s\in[0,t]$ and $\mu\in\{\mu_0,\mu_1^{\iota_2}\}$
\begin{equation*}
\begin{split}
|\mathcal{A}_{\mu}^{\leq 0}&(U_{\iota_1}(s),P_{>K+4}W_{\iota_2}(s),P_{\leq K}W'_+(s))|=|\mathcal{A}_{\mu}^{\leq 0}(P_{>K}U_{\iota_1}(s),P_{>K+4}W_{\iota_2}(s),P_{\leq K}W'_+(s))|\\
&\lesssim \|P_{>K}U_{\iota_1}(s)\|_{H^{12}}\|W_+(s)\|_{L^2}\|P_{>K}W'_+(s)\|_{L^2}
  \lesssim 2^{-2K}\oe^3,
\end{split}
\end{equation*} 
which as claimed. The bounds \eqref{gar21.1} are similar, which completes the proof of the lemma.
\end{proof}

\subsection{Small modulations and low frequencies}\label{SecEn14} 

In this subsection we complete the proof of Proposition \ref{strategy}. We have to estimate the contribution of small modulations and small frequencies, i.e. expressions of the form
\begin{equation*}
\int_0^t\mathcal{A}_{\mu}^{\leq 0}(U_{\iota_1}(s),P_{\leq K+4}W_{\iota_2}(s),P_{\leq K} W'_+(s))\,ds,
\end{equation*}
for $\mu\in\{\mu_0,\mu_1^{\iota_2}\}$ (and a similar expression when $\mu=\mu_2$, with $K$ replaced by $J$). 

\subsubsection{The operators $\mathcal{A}_{\mu_2}$} We start with the simpler case when $\mu=\mu_2$ and prove the following:

\begin{lemma}\label{lasty1}
Assume that $\iota_1,\iota_2\in\{+,-\}$, $W'_+\in\mathcal{W}_+$, $W_{\iota_2}\in\mathcal{W}_{\iota_2}$, $2^{3J/2} = \oe T_\e \KK_g$ as in \eqref{gar20}, and $t\in[0,T]$. Then
\begin{equation}\label{gar22}
\Big|\int_0^t \mathcal{A}_{\mu_2}^{\leq 0}(U_{\iota_1}(s),P_{\leq J+4}W_{\iota_2}(s),P_{\leq J}W'_+(s)) \Big|\lesssim \oe^2\KK_g^{-1}.
\end{equation}
\end{lemma}

\begin{proof}
Notice that $\widehat{U}(0,s) = \widehat{W}(0,s)$ (see \eqref{defscalarunk}, \eqref{defW_n}, and \eqref{Taf}),
and therefore we can insert discrete Littlewood-Paley operators. 
For \eqref{gar22} it suffices to prove that
\begin{equation}
\label{gar23}
\Big|\int_0^t \mathcal{A}_{\mu_2}^{\leq 0}(P_{k_1}U_{\iota} (s), P_{k_2}W_{\iota_2}(s),P_{k}W'_+(s))\,ds\Big| 
\lesssim \oe^2\KK_g^{-1} [\log(1/\oe)]^{-1}2^{-k_1},
\end{equation}
for any $W'\in \mathcal{W}_+$, $W_{\iota_2} \in \mathcal{W}_{\iota_2}$, 
and $k,k_1,k_2\in[0,\infty)\cap\Z$ satisfying
\begin{equation}
\label{gar23.5}
\begin{split}
& |\max(k,k_1,k_2)-\mathrm{med}(k,k_1,k_2)|\leq 4,
\\
& \max\{2^k,2^{k_2}\}\leq 2^{J+4}\lesssim \oe^{-4/9} \log(1/\oe)^{-4/3} \KK_g^{-8/9}.
\end{split}
\end{equation}

{\bf{Step 1.}} To prove \eqref{gar23} we begin by using the lower bound 
\eqref{SDbound0} to integrate by parts in $s$.
More precisely, it follows from Proposition \ref{prop1} and \eqref{ml12} that, for $g \not\in \mathcal{N}$,
\begin{equation}
\label{gar24}
\frac{1}{|\Phi_{\iota_1\iota_2}(\xi,\eta)|}\lesssim 2^{3\overline{k}/2}(1+|\overline{k}|)^{3/2}2^{4k_1}, \qquad \overline{k}:=\max(k,k_1,k_2),
\end{equation}
for any $(\xi,\eta)$ in the support of the operator in the left-hand side of \eqref{gar23}. In analogy with \eqref{LMnot1} we define
\begin{align}
\label{SMnot1}
&\mathcal{T}_{\mu_2}^{\leq 0}(F,G,H):=\sum_{\xi,\eta\in\Z^2}\mu_2(\xi,\eta)\widehat{F}(\xi-\eta) \widehat{G}(\eta)\widehat{\overline{H}}(-\xi)
  \cdot \frac{\varphi_{\leq 0}(\Phi(\xi,\eta))}{i\Phi(\xi,\eta)}.
\end{align}
We integrate by parts in $s$, as in \eqref{LM1}, to obtain
\begin{align}
\label{gar25}
\int_0^t\mathcal{A}_{\mu_2}^{\leq 0}(P_{k_1}U_{\iota_1}(s),P_{k_2}W_{\iota_2}(s),P_kW'_+(s))\,ds = I + II_1 + II_2 + III,
\end{align}
where
\begin{align}
\label{gar25.5}
\begin{split}
I &:=- \int_0^t \mathcal{T}_{\mu_2}^{\leq 0}((\partial_s+i\Lambda_{\iota_1})P_{k_1}U_{\iota_1}(s),P_{k_2}W_{\iota_2}(s),P_{k}W'_+(s))\,ds,
\\
II_1 &:=- \int_0^t \mathcal{T}_{\mu_2}^{\leq 0}(P_{k_1}U_{\iota_1}(s),(\partial_s+i\Lambda_{\iota_2})P_{k_2}W_{\iota_2}(s),P_{k}W'_+(s))\,ds,
\\
II_2 &:=- \int_0^t \mathcal{T}_{\mu_2}^{\leq 0}(P_{k_1}U_{\iota_1}(s),P_{k_2}W_{\iota_2}(s),(\partial_s+i\Lambda)P_{k}W'_+(s))\,ds,
\\
III &:=\mathcal{T}_{\mu_2}^{\leq 0}(P_{k_1}U_{\iota_1}(t),P_{k_2}W_{\iota_2}(t),P_{k}W'_+(t)) 
  - \mathcal{T}_{\mu_2}^{\leq 0}(P_{k_1}U_{\iota_1}(0),P_{k_2}W_{\iota_2}(0),P_{k}W'_+(0)).
\end{split}
\end{align}
Notice that the denominators $\Phi=\Phi_{\iota_1\iota_2}$ in the definition above do not vanish, due to \eqref{gar24}.

Using the basic trilinear estimate \eqref{Fact2}, the small divisors bound \eqref{gar24}, 
the bound on $\mu_2$ in \eqref{ProBulk5sym}, and \eqref{gar10}--\eqref{gar11} we bound, for any $s\in[0,T]$
\begin{equation*}
\begin{split}
|\mathcal{T}_{\mu_2}^{\leq 0}&(P_{k_1}U_{\iota_1}(s),P_{k_2}W_{\iota_2}(s),P_{k}W'_+(s))|
\\
& \lesssim 2^{3\overline{k}/2}(1+|\overline{k}|)^{3/2}\cdot 2^{-3\overline{k}/2} 2^{12k_1} 
  \|P_{k_1}U_{\iota_1}(s)\|_{H^2} \|P_{k_2}W_{\iota_2}(s)\|_{L^2} \|P_{k}W'_+(s)\|_{L^2}\\
  &\lesssim \oe^3 2^{-4k_1} (1+|\overline{k}|)^{3/2}.
\end{split}
\end{equation*}
This is clearly bounded by the right-hand side of \eqref{gar23}, also in view of \eqref{gar23.5}, as desired.

Next we consider the space-time contributions $I,II_1,II_2$. These are easy to estimate when $k_1$ is large. 
Indeed, using \eqref{Fact2}, \eqref{gar24}, and \eqref{gar10}--\eqref{gar12.5} we bound
\begin{align}
 \label{gar25.9}
\begin{split}
|I|+|II_1|+|II_2|\lesssim T_\e \cdot 2^{3\overline{k}/2}(1+|\overline{k}|)^{3/2} \cdot 2^{-3\overline{k}/2}2^{12k_1} \cdot \oe^42^{3\overline{k}/2}2^{-20k_1}\lesssim T_\e \oe^4 2^{1.6\overline{k}} 2^{-6k_1}.
\end{split}
\end{align}
This suffices to prove the desired bounds if $k_1\geq\overline{k}-30$. 

On the other hand, if $k_1\leq \overline{k}-30$ then we may assume that $|k-k_2|\leq 4$, $|\overline{k}-k|\leq 4$, see \eqref{gar23.5}.
We may also assume that $\iota_2=+$, otherwise the modulation cannot be small and the operators are trivial.

We can now estimate $|I|$ easily, using \eqref{Fact2}, \eqref{gar24}, and \eqref{gar10}--\eqref{gar11} as before,
\begin{equation*}
|I|\lesssim 2^{3k/2}(1+|k|)^{3/2}2^{12k_1}2^{-3k/2}\cdot T_\e \oe^4 2^{-20k_1} \lesssim 2^{-2k_1} (T_\e\oe \KK_g)^{1.1} \oe^3,
\end{equation*}
which suffices. Then we notice that the two integrals $II_1$ and $II_2$ are similar,
after changes of variables. In
the case of small modulations there is no meaningful cancellation between these terms, 
because the fractions $1/\Phi(\xi,\eta)$ vary wildly as the variables $\xi$ and $\eta$ move to nearby lattice points.
Therefore, to prove \eqref{gar23} it remains to show that
\begin{equation}
\label{gar26}
\Big| \int_0^t \mathcal{T}_{\mu_2}^{\leq 0} (P_{k_1} U_{\iota_1}(s), P_{k_2}(\partial_s+i\Lambda)W(s), P_{k}W'(s))\,ds \Big|
  \lesssim (\oe^2\KK_g^{-1}) [\log(1/\oe)]^{-1}2^{-k_1},
\end{equation}
for any $t\in[0,T]$, $\iota_1\in\{+,-\}$, $W,W' \in\mathcal{W}_+$, 
and $k,k_1,k_2\in[-4,\infty)\cap\Z$ satisfying
\begin{equation} 
\label{gar26.5}
|k-k_2|\leq 4, \qquad k_1\leq k-20, \qquad 2^k\lesssim (T_\varep\oe\KK_g)^{2/3}.
\end{equation}

The formula \eqref{gar12} and the decomposition \eqref{ExpV} show that
\begin{equation}
\label{gar27}
(\partial_t + i\Lambda)W = - iT_{V_1\cdot\zeta}W-iT_{(V-V_1)\cdot\zeta}W- iT_{\Sigma- \Lambda}W + \mathcal{E}_W,
\end{equation}
for all $W\in\mathcal{W}_+$. As a consequence of \eqref{gar12.5} and the restriction $\oe 2^{k_2}\lesssim 1$,
\begin{equation*}
\|P_{k_2}T_{(V-V_1)\cdot\zeta}W\|_{L^2}+\|P_{k_2}T_{\Sigma-\Lambda}W\|_{L^2} + \|P_{k_2}\mathcal{E}_W\|_{L^2} \lesssim \oe^2 2^{k_2/2}.
\end{equation*}
Therefore, using \eqref{Fact2}, the bounds \eqref{gar24} and \eqref{ProBulk5sym}, and \eqref{gar10}--\eqref{gar11}, we estimate
\begin{equation}\label{gar27.1}
\begin{split}
\Big|\int_0^t \mathcal{T}_{\mu_2}^{\leq 0}(P_{k_1} U_{\iota_1}(s),P_{k_2}(-iT_{(V-V_1)\cdot\zeta}W-iT_{\Sigma - \Lambda}W + \mathcal{E}_W)(s),P_{k}W'_+(s))\,ds\Big|
\\
\lesssim T_\e \cdot 2^{3k/2}|k|^{3/2}  \cdot 2^{-3k/2} 2^{12k_1} \cdot \oe 2^{-20k_1} \cdot \oe^3 2^{k/2} 
\\
\lesssim 2^{-2k_1} \cdot T_\e \oe^4 (\KK_g T_\e \oe)^{1/3} [\log(1/\oe)]^{3/2}.
\end{split}
\end{equation}
having used the restriction on the frequencies \eqref{gar23.5}.
This is consistent with the desired bounds \eqref{gar26}, once we recall that $\KK_gT_\e \oe\approx \KK_g^{-4/3}\oe^{-2/3}[\log(1/\oe)]^{-2}$, see \eqref{al0}.

After these reductions, for \eqref{gar26} it remains to prove that  
\begin{equation}
\label{gar26.2}
\Big| \int_0^t \mathcal{T}_{\mu_2}^{\leq 0} (P_{k_1} U_{\iota_1}(s), P_{k_2}T_{V_1\cdot\zeta}W(s), P_{k}W'(s))\,ds \Big|
  \lesssim (\oe^2\KK_g^{-1}) [\log(1/\oe)]^{-1}2^{-k_1},
\end{equation}
for any $k,k_1,k_2\in[-4,\infty)\cap\Z$ satisfying \eqref{gar26.5}.
We need to be more careful here, because estimating this term in the same way would not allow us to reach the desired time $T_\e$ in \eqref{al0}.

\medskip
{\bf{Step 2.}} Using the definitions we can write the integral in the left-hand side of \eqref{gar26.2} as
\begin{equation}\label{gar99}
\begin{split}
C\int_0^t\sum_{\xi,\eta,\rho\in\Z^2} \mu_2(\xi,\eta) &\frac{\varphi_{\leq 0}(\Phi(\xi,\eta))}{\Phi(\xi,\eta)}
  \widehat{P_{k_1}U_{\iota_1}}(\xi-\eta,s) \varphi_k(\xi)\overline{\widehat{W'}}(\xi,s)
\\
& \qquad \times\varphi_{k_2}(\eta)\chi\Big(\frac{|\eta-\rho|}{|\eta+\rho|}\Big)\frac{(\eta-\rho)\cdot(\eta+\rho)}{|\eta-\rho|^{1/2}}
  \widehat{\Im U}(\eta-\rho,s)\widehat{W}(\rho,s)\,ds.
\end{split}
\end{equation} 
We observe that the factor $(\eta-\rho)\cdot(\eta+\rho)$ is a depletion factor similar to the factor $\mathfrak{d}$ in \eqref{ProBulk2sym}.
According to \eqref{Factor}, this factor is expected to gain $1/2$ derivative in the region where the $(\eta,\rho)$ modulation is small. We will the exploit this fact,
and alternatively integrate by parts in time in the remaining region where the $(\eta,\rho)$ modulation is large.

To implement this strategy we define the {\it{four-wave modulation functions}}
\begin{equation}\label{gar100}
\Psi_{\iota_1\iota_3}(\xi,\eta,\rho):=\Lambda(\xi)-\Lambda(\rho)-\iota_1\Lambda(\xi-\eta)-\iota_3\Lambda(\eta-\rho),
\end{equation}
where $\iota_1,\iota_3\in\{+,-\}$.
We use these functions to decompose the integrals into low and high four-wave modulations. 
More precisely, we define the multipliers
\begin{equation}\label{gar101}
\begin{split}
\beta_\ast(\xi,\eta,\rho) := & \varphi_\ast(\Psi_{\iota_1\iota_3}(\xi,\eta,\rho)) \cdot 
  \mu_2(\xi,\eta)\varphi_{k_1}(\xi-\eta)\varphi_k(\xi)\varphi_{k_2}(\eta)
\\
& \qquad \times\frac{\varphi_{\leq 0}(\Phi_{\iota_1+}(\xi,\eta))}{\Phi_{\iota_1+}(\xi,\eta)}
  \chi\Big(\frac{|\eta-\rho|}{|\eta+\rho|}\Big)\frac{(\eta-\rho)\cdot(\eta+\rho)}{|\eta-\rho|^{1/2}},
\end{split}
\end{equation}
where $\ast\in\{\leq 0,>0\}$. Then we define the integrals
\begin{equation}\label{gar102}
\mathcal{L}_{k,k_1,k_2}^\ast 
  := \int_0^t\sum_{\xi,\eta,\rho\in\Z^2}\beta_\ast(\xi,\eta,\rho)\widehat{U_{\iota_1}}(\xi-\eta,s)\widehat{U_{\iota_3}}(\eta-\rho,s)
  \widehat{W}(\rho,s)\overline{\widehat{W'}}(\xi,s),
\end{equation}
and observe that, in view of \eqref{gar99}, for \eqref{gar26.2} it suffices to prove that
\begin{equation}\label{gar103}
\big|\mathcal{L}_{k,k_1,k_2}^{\leq 0}\big| + \big|\mathcal{L}_{k,k_1,k_2}^{>0}\big|\lesssim (\oe^2\KK_g^{-1})\log(1/\oe)^{-1}2^{-k_1},
\end{equation}
for any $\iota_1,\iota_3\in\{+,-\}$ and $k,k_1,k_2\in[-4,\infty)\cap\Z$ satisfying \eqref{gar26.5}.

We prove now the inequalities \eqref{gar103}. 
We will use the following quartic analog of the bounds \eqref{Fact2}: for any functions $F,G,H,H'$ and any symbol $m$ we have
\begin{equation}
\label{Fact3}
\sum_{\xi,\eta,\rho\in\Z^2}\big|m(\xi,\eta,\rho)\widehat{F}(\xi-\eta) \widehat{G}(\eta-\rho)\widehat{H}(\rho)\overline{\widehat{H'}}(\xi)\big|
  \lesssim\|m\|_{L^\infty}\|F\|_{H^2}\|G\|_{H^2}\|H\|_{L^2}\|H'\|_{L^2}.
\end{equation}
This follows easily using the Cauchy--Schwarz inequality.

We observe that
\begin{equation*}
\Phi_{\iota_1+}(\xi,\eta)+\Phi_{\iota_3+}(\eta,\rho)=\Psi_{\iota_1\iota_3}(\xi,\eta,\rho).
\end{equation*}
In particular $|\Phi_{\iota_3+}(\eta,\rho)|\lesssim 1$ in the support of the sum defining $\mathcal{L}_{k,k_1,k_2}^{\leq 0}$. Using the estimates on symbols \eqref{ProBulk5sym}, \eqref{Factor}, and \eqref{gar24}, we can bound
\begin{equation*}
|\beta_{\leq 0}(\xi,\eta,\rho)| \lesssim \varphi_k(\xi)\varphi_{k_2}(\eta)\varphi_{k_1}(\xi-\eta) \varphi_{\leq k-10}(\eta-\rho)
  \cdot 2^{k/2} |k|^{3/2} \cdot 2^{12k_1} \langle\eta-\rho\rangle^2.
\end{equation*}
Thus, using \eqref{Fact3}, \eqref{gar10}--\eqref{gar11}, and \eqref{gar26.5}, we obtain
\begin{equation*}
\big|\mathcal{L}_{k,k_1,k_2}^{\leq 0}\big| \lesssim T_\e \oe^4 \cdot 2^{k/2} |k|^{3/2} 2^{-2k_1} 
  \lesssim 2^{-2k_1} T_\varep\oe^4(\KK_g T_\e \oe)^{1/3} [\log(1/\oe)]^{3/2},
\end{equation*}
as desired (compare with \eqref{gar27.1}).

{\bf{Step 3.}} To estimate $\big|\mathcal{L}_{k,k_1,k_2}^{>0}\big|$ we have to integrate by parts in time again. We define
\begin{equation*}
\mathcal{J}_{k,k_1,k_2}^{\geq 0}(F,G,H,H') := \sum_{\xi,\eta,\rho\in\Z^2} \frac{\beta_{\geq 0}(\xi,\eta,\rho)}{i\Psi_{\iota_1\iota_3}(\xi,\eta,\rho)}
  \widehat{F}(\xi-\eta)\widehat{G}(\eta-\rho)\widehat{H}(\rho)\overline{\widehat{H'}}(\xi),
\end{equation*}
where  $\beta_{\geq 0}$ as in \eqref{gar101}. Similarly to \eqref{gar25}--\eqref{gar25.5} we can write
\begin{equation}
\label{gar105}
\mathcal{L}^{>0}_{k,k_1,k_2} = L_1+L_2+L_3+L_4+L_5,
\end{equation}
where
\begin{equation*}
\begin{split}
& L_1:=-\int_0^t\mathcal{J}_{k,k_1,k_2}^{\geq 0}
((\partial_s+i\Lambda_{\iota_1})U_{\iota_1}(s),U_{\iota_3}(s),W(s),W'(s))\,ds,
\\
& L_2:=-\int_0^t\mathcal{J}_{k,k_1,k_2}^{\geq 0}
(U_{\iota_1}(s),(\partial_s+i\Lambda_{\iota_3})U_{\iota_3}(s),W(s),W'(s))\,ds,
\\
& L_3:=-\int_0^t\mathcal{J}_{k,k_1,k_2}^{\geq 0}
(U_{\iota_1}(s),U_{\iota_3}(s),(\partial_s+i\Lambda)W(s),W'(s))\,ds,
\\
& L_4:=-\int_0^t\mathcal{J}_{k,k_1,k_2}^{\geq 0}
(U_{\iota_1}(s),U_{\iota_3}(s),W(s),(\partial_s+i\Lambda)W'(s))\,ds,
\\
& L_5:=\mathcal{J}_{k,k_1,k_2}^{\geq 0}(U_{\iota_1}(t),U_{\iota_3}(t),W(t),W'(t))
-\mathcal{J}_{k,k_1,k_2}^{\geq 0}(U_{\iota_1}(0),U_{\iota_3}(0),W(0),W'(0)).
\end{split}
\end{equation*}

The definition \eqref{gar101} and the bounds \eqref{ProBulk5sym} and \eqref{gar24} show that
\begin{equation}
\label{gar106}
\Big| \frac{\beta_{\geq 0}(\xi,\eta,\rho)}{i\Psi_{\iota_1\iota_3}(\xi,\eta,\rho)}\Big|
  \lesssim \varphi_k(\xi)\varphi_{k_2}(\eta)\varphi_{k_1}(\xi-\eta)\varphi_{\leq k-10}(\eta-\rho)\cdot 2^{k}|k|^{3/2}2^{12k_1}\langle\eta-\rho\rangle^2.
\end{equation}
Therefore, using \eqref{Fact3}, \eqref{gar23.5}, and the bounds \eqref{gar10}--\eqref{gar12.5}, we can estimate
\begin{equation*}
|L_1|+|L_2|+|L_3|+|L_4|+|L_5| \lesssim T_\e \cdot \oe^5 2^{2k}|k|^{3/2} 2^{-2k_1} \lesssim  2^{-2k_1}T_\varep\oe^5(\KK_gT_\e \oe)^{4/3}[\log(1/\oe)]^{3/2},
\end{equation*} 
in view of the constraints \eqref{gar26.5}. The desired bounds for $|\mathcal{L}^{>0}_{k,k_1,k_2}|$ in \eqref{gar103} follow using also \eqref{al0}. This completes the proof of the lemma.
\end{proof}

\subsubsection{The operators $\mathcal{A}_{\mu_0}$ and $\mathcal{A}_{\mu_1^{\pm}}$} In this section we complete the proof of the bounds \eqref{gar0}--\eqref{gar1b} in Proposition \ref{strategy}, and therefore of the main bootstrap Proposition \ref{MainBootstrap}. More precisely:

\begin{lemma}\label{gar50}
For any $t\in[0,T]$ and $\iota_1,\iota_2\in\{+,-\}$ we have
\begin{equation}
\label{gar90}
\Big| \Re \int_0^t\mathcal{A}_{\mu_0}^{\leq 0}(iU_{\iota_1}(s),P_{\leq K+4}W_{\iota_2}^0(s),P_{\leq K}W^0(s))\,ds\Big| \lesssim \oe^2 \KK_g^{-1},
\end{equation}
\begin{equation}\label{gar91a}
\Big| \Re \int_0^t \mathcal{A}_{\mu_1^+}^{\leq 0}(iU_{\iota_1}(s),P_{\leq K+4}W^0(s),P_{\leq K}W^0(s))\,ds\Big| \lesssim \oe^2 \KK_g^{-1},
\end{equation}
\begin{equation}\label{gar91b}
\Big|\int_0^t \mathcal{A}_{\mu_1^-}^{\leq 0}(U_{\iota_1}(s),P_{\leq K+4}W^0_-(s),P_{\leq K}W^0(s))\,ds\Big| \lesssim \oe^2 \KK_g^{-1}.
\end{equation}
\end{lemma}

\begin{proof} Recall that $W^0=T_{\Sigma}^{2N_0/3}U\in\mathcal{W}_+$ and $2^K=\oe T_\e \KK_g$, see \eqref{gar40.4}. The case $\iota_2=-$ can be analyzed as in the proof of Lemma \ref{lasty1}, by noticing that the modulation can only be small if the frequency of the undifferentiated variable $U_{\iota_1}$ is large, so there is no loss of derivatives after applying normal forms and the corresponding integral are bounded by $T_\varep\oe^4$. So we will assume that $\iota_2=+$ and insert discrete Littlewood-Paley projections. After removing also the contribution of large frequencies $k_1$, it remains to prove that
\begin{equation}
\label{gar52.0}
\Big| \Re \int_0^t \mathcal{A}_{\mu}^{\leq 0}(P_{k_1} iU_{\iota_1} (s),P_{k_2}W^0(s),P_{k}W^0(s))\,ds\Big| \lesssim \oe^2\KK_g^{-1} [\log(1/\oe)]^{-1}2^{-k_1},
\end{equation}
for $\mu\in\{\mu_0,\mu_1^+\}$, and any $k,k_1,k_2\in[-4,\infty)\cap\Z$ satisfying
\begin{equation}
\label{gar51.5}
|k-k_2| \leq 4, \qquad k_1 \leq k-30, \qquad 2^k \leq 2^K = \oe T_\e \KK_g\lesssim \oe^{-2/3}\KK_g^{-4/3}[\log(1/\oe)]^{-2}.
\end{equation}

Notice that to prove \eqref{gar52.0} we only need to estimate the real part of the integral, which is important in some of the cases. As in \eqref{SMnot1} we define
\begin{align}
\label{defT}
\mathcal{T}^{\leq 0}_{\mu}(F,G,H) := \sum_{\xi,\eta\in\Z^2} \mu(\xi,\eta)\widehat{F}(\xi-\eta) \widehat{G}(\eta)\widehat{\overline{H}}(-\xi)\cdot 
  \frac{\varphi_{\leq 0}(\Phi(\xi,\eta))}{i\Phi(\xi,\eta)},
\end{align}
where the phase function $\Phi=\Phi_{\iota_1+}$ does not vanish due to \eqref{gar24}. Integrating by parts as before (see \eqref{gar25}--\eqref{gar25.5}) we can write,
for $\mu \in \{\mu_0, \mu_1^+\}$,
\begin{align}
\label{gar52.1}
& \int_0^t \mathcal{A}^{\leq 0}_{\mu}(P_{k_1}iU_{\iota_1} (s), P_{k_2}W^0(s),P_{k}W^0(s)) = I'+ II'_1 + II'_2 + III'
\end{align} 
where
\begin{align}
\label{gar52.2}
\begin{split}
I' & := - i\int_0^t \mathcal{T}^{\leq 0}_{\mu} ((\partial_s+i\Lambda_{\iota_1}) P_{k_1} U_{\iota_1}(s), P_{k_2}W^0(s),P_{k}W^0(s)) )\,ds,
\\
II'_1 & := - i\int_0^t \mathcal{T}^{\leq 0}_{\mu} (P_{k_1} U_{\iota_1}(s), P_{k_2}(\partial_s+i\Lambda)W^0(s), P_k W^0(s)) \,ds,
\\
II'_2 & :=  - i\int_0^t\mathcal{T}^{\leq 0}_{\mu} (P_{k_1} U_{\iota_1}(s), P_{k_2}W^0(s), P_k (\partial_s+i\Lambda)W^0(s)) \,ds,
\\
III' & := i\mathcal{T}^{\leq 0}_{\mu} (P_{k_1} U_{\iota_1}(t), P_{k_2} W^0(t), P_k W^0(t)) - i\mathcal{T}^{\leq 0}_{\mu} (P_{k_1} U_{\iota_1}(0),P_{k_2}W^0(0), P_kW^0(0)).
\end{split}
\end{align} 
We estimate the contributions of these integrals separately, in several steps.

{\bf{Step 1.}} We start with the easier terms $III'$ and $I'$. Notice that 
\begin{equation}\label{boundMu}
|\mu_0(\xi,\eta)|+|\mu_1^+(\xi,\eta)|\lesssim 2^{-k}2^{6k_1}
\end{equation}
in the support of the sum, as a consequence of \eqref{ProBulk2sym}, \eqref{Factor}, and \eqref{ProBulk3sym}. Using the basic trilinear estimate \eqref{Fact2}, the small divisors bound \eqref{gar24}, and the bounds \eqref{gar51.5} we have
\begin{equation*}
\begin{split}
|III'|&\lesssim \sup_{s\in[0,T]}2^{3k/2}|k|^{3/2}\cdot 2^{-k} 2^{12k_1} 
  \|P_{k_1}U_{\iota_1}(s)\|_{H^2} \|P_{k_2}W^0(s)\|_{L^2} \|P_{k}W^0(s)\|_{L^2}\\
  &\lesssim \oe^3 2^{-4k_1} (\oe T_\e \KK_g)^{1/2}[\log(1/\oe)]^{3/2}.
\end{split}
\end{equation*}
This is better than the desired estimates \eqref{gar52.0}. Moreover, using also \eqref{gar10}--\eqref{gar11},
\begin{equation}
\label{tightT}
|I'|\lesssim T_\varep 2^{3k/2}|k|^{3/2}\cdot 2^{-k} 2^{-4k_1} \oe^4\lesssim 2^{-4k_1}T_\varep\oe^4 2^{k/2}[\log(1/\oe)]^{3/2}.
\end{equation}
Using the definition \eqref{al0}, we have 
\begin{equation}
\label{tyu3}
T_\varep\oe^4 2^{k/2}\lesssim T_\varep\oe^4  (\oe T_\e \KK_g)^{1/2}\lesssim \oe^2\KK_g^{-2.5}[\log(1/\oe)]^{-3},
\end{equation}
thus the estimates \eqref{tightT} precisely match the desired estimates \eqref{gar52.0}, up to logarithmic factors.

{\bf{Step 2.}} It remains to consider the harder terms $II'_1$ and $II'_2$. 
The two terms are similar, after changes of variables. To bound them, we examine the formula
\begin{equation*}
(\partial_t + i\Lambda)W^0 = - iT_{V\cdot\zeta}W^0 -iT_{\Sigma_1}W^0 - iT_{\Sigma -\Sigma_1 - \Lambda}W^0 + \mathcal{E}_{W^0},
\end{equation*}
see \eqref{gar12}. Notice that, as a consequence of \eqref{gar12.5},
\begin{equation*}
\|P_{k_2}T_{\Sigma-\Lambda-\Sigma_1}W^0(s)\|_{L^2} + \|P_{k_2}\mathcal{E}_{W^0}(s)\|_{L^2} \lesssim \oe^3 2^{3k/2} + \oe^2\lesssim \oe^2
\end{equation*}
Therefore, using \eqref{Fact2}, \eqref{boundMu}, \eqref{gar24}, and \eqref{gar10}--\eqref{gar11} as before, we estimate
\begin{equation}\label{gar60}
\begin{split}
\Big|\int_0^t \mathcal{T}^{\leq 0}_{\mu}(P_{k_1} U_{\iota_1}(s),P_{k_2}(-iT_{\Sigma - \Lambda - \Sigma_1}W^0 + \mathcal{E}_{W^0})(s),P_{k}W^0(s))\,ds\Big|
\\
\lesssim T_\e \cdot 2^{3k/2}|k|^{3/2} 2^{12k_1}2^{-k} \cdot \oe^4 2^{-20k_1}\\
\lesssim 2^{-4k_1} \cdot T_\e  \oe^4 (T_\e\oe\KK_g)^{1/2}[\log(1/\oe)]^{3/2}.
\end{split}
\end{equation}
This suffices, due to \eqref{tyu3}. It remains to bound the contributions of the terms $T_{V\cdot\ze}W^0$ and $T_{\Sigma_1}W^0$. These estimates are harder, and we prove them separately in Lemma \ref{lasty2} below.\end{proof}

\begin{lemma}\label{lasty2}
For any $t\in[0,T]$, $\iota_1\in\{+,-\}$, $\mu \in \{\mu_0,\mu_1^+\}$, 
and $k,k_1,k_2\in[-4,\infty)\cap\Z$ satisfying \eqref{gar51.5}, we have
\begin{equation}
\label{gar61a}
\Big|\Re \int_0^t \mathcal{T}^{\leq 0}_{\mu} (P_{k_1} U_{\iota_1}(s), P_{k_2}T_{V\cdot\ze}W^0(s),P_{k}W^0(s))\,ds\Big|\lesssim (\oe^2\KK_g^{-1})[\log(1/\oe)]^{-1}2^{-k_1},
\end{equation}
and
\begin{equation}
\label{gar61b}
\Big|\Re \int_0^t \mathcal{T}^{\leq 0}_{\mu} (P_{k_1} U_{\iota_1}(s), P_{k_2}T_{\Sigma_1}W^0(s),P_{k}W^0(s))\,ds\Big|
  \lesssim (\oe^2\KK_g^{-1})[\log(1/\oe)]^{-1}2^{-k_1},
\end{equation} 
\end{lemma}

\begin{proof} 
The proofs of \eqref{gar61a} and \eqref{gar61b} are more involved.
The use of the same argument as in Lemma \ref{lasty1} would lead to an additional loss of a factor of $2^{k/2}$,  which ultimately leads to a smaller value of $T_\e$. Our main additional ingredient is the lower bounds in Proposition \ref{prop2}.

\medskip
{\bf{Proof of \eqref{gar61a}.}} 
{\bf{Step 1.}} We start with some reductions. Notice first that we may replace $V$ by $V_1$, see \eqref{ExpV}, at the expense of acceptable errors. Also, we decompose $V_1=\sum_{k_3\geq -4}P_{k_3}V_1$ and estimate, using \eqref{Fact3} and \eqref{boundMu}, 
\begin{equation*}
\begin{split}
\Big|\int_0^t \mathcal{T}^{\leq 0}_{\mu} (P_{k_1} U_{\iota_1}(s), &P_{k_2}T_{P_{k_3}V_1\cdot\ze}W^0(s),P_{k}W^0(s))\,ds\Big|\lesssim T_\varep\oe^42^{3k/2}|k|^{3/2}2^{-20k_1}2^{-20k_3}\\
&\lesssim T_\varep\oe^4  (\oe T_\e \KK_g)^{1/2}[\log(1/\oe)]^{3/2}\cdot (2^k2^{-20k_1}2^{-20k_3}).
\end{split}
\end{equation*}
This suffices to bound the contribution of the pairs $(k_1,k_3)$ for which $2^{18k_1}2^{18k_3}\geq\KK_g^{-1} 2^k$, using also \eqref{tyu3}.

To bound the contributions of the pairs $(k_1,k_3)$ for which $2^{18k_1}2^{18k_3}\leq\KK_g^{-1} 2^k$ we subdivide the operators $\mathcal{T}_\mu^{\leq 0}$ in \eqref{defT} into modulations which are much smaller than $2^{-k/2}$ and the remaining ones. More precisely, define $B=B(k,k_1,k_3)$ such that
\begin{align}
\label{defcut}
2^B := \KK_g^{-1}2^{-k/2} (2^{k_1}+2^{k_3})^{-2},
\end{align}
and decompose
\begin{align}
\label{defAmu}
\begin{split}
& \mathcal{T}_{\mu}^{\leq 0}(F,G,H) = \mathcal{T}^{\leq B}_{\mu}(F,G,H) + \mathcal{T}^{(B,0]}_{\mu}(F,G,H)
,\\
& \mathcal{T}^{\ast}_{\mu}(F,G,H) := \sum_{\xi,\eta\in\Z^2} \mu(\xi,\eta)\widehat{F}(\xi-\eta) \widehat{G}(\eta)\widehat{\overline{H}}(-\xi)\cdot 
  \frac{\varphi_\ast(\Phi(\xi,\eta))}{i\Phi(\xi,\eta)},
\end{split}
\end{align}
where $\ast\in\{\leq B,(0,B]\}$, $\mu\in\{\mu_0,\mu_1^+\}$, and $\Phi(\xi,\eta)=\Phi_{\iota_1+}(\xi,\eta)=\Lambda(\xi)-\Lambda(\eta)-\iota_1\Lambda(\xi-\eta)$. Notice that
\begin{align}
\label{gar52'}
\Big|\int_0^t \mathcal{T}^{(B,0]}_{\mu} (P_{k_1} U_{\iota_1}(s), 
  P_{k_2}T_{P_{k_3}V_1\cdot\ze}W^0(s), P_k W^0(s)) \,ds\Big| 
  \lesssim \oe^2\KK_g^{-1} [\log(1/\oe)]^{-1}2^{-k_1}2^{-k_3},
\end{align}
for $\iota_1\in\{+,-\}$ and $k,k_1,k_2$ verifying \eqref{gar51.5}. 
Indeed, using the lower bounds $|\Phi(\xi,\eta)| \gtrsim 2^B$, the basic multilinear estimate \eqref{Fact2}, the bounds on the symbols \eqref{boundMu}, and the estimates \eqref{gar10}--\eqref{gar12.5}, we can bound the left-hand side of \eqref{gar52'} by
\begin{align*}
CT_\e \cdot 2^{k/2}\KK_g 2^{-k}2^{10k_1}2^{4k_3}\|P_{k_1}U_{\iota_1}(s)\|_{H^2} \|P_{k_2}T_{P_{k_3}V_1\cdot\ze}W^0(s)\|_{L^2}\|P_{k}W^0(s)\|_{L^2}\\
\lesssim 2^{-4k_1}2^{-4k_3}T_\varep\oe^4  \KK_g(\oe T_\e \KK_g)^{1/2}.
\end{align*}
Notice that this suffice due to \eqref{tyu3}. After these reductions, for \eqref{gar61a} it suffices to prove that 
\begin{align}
\label{gar58}
\Big| \Re \int_0^t \mathcal{T}^{\leq B}_{\mu} (P_{k_1} U_{\iota_1}(s), P_{k_2}T_{P_{k_3}V_1\cdot\ze}W^0(s), P_k W^0(s)) \,ds\Big| \lesssim \oe^2\KK_g^{-1} [\log(1/\oe)]^{-1}2^{-k_1}2^{-k_3},
\end{align}
provided that $\iota_1\in\{+,-\}$ and $k,k_1,k_2,k_3\in[-4,\infty)$ satisfy
\begin{equation}
\label{tyu8}
|k-k_2| \leq 4,\qquad 2^k \leq 2^K = \oe T_\e \KK_g\lesssim \oe^{-2/3}\KK_g^{-4/3}[\log(1/\oe)]^{-2},\qquad 2^{18k_1}2^{18k_3}\leq\KK_g^{-1} 2^k.
\end{equation}

{\bf{Step 2.}} To prove \eqref{gar58} we expand the expression in the left-hand side as
\begin{align}
\label{gartr1}
\begin{split}
&  \frac{-1}{4\pi^2}\Re\int_0^t \sum_{\xi,\eta,\rho\in\Z^2} \mu(\xi,\eta)\frac{\varphi_{\leq B}(\Phi_{\iota_1+}(\xi,\eta))}{i\Phi_{\iota_1+}(\xi,\eta)}
  \varphi_{k_1}(\xi-\eta)\widehat{U_{\iota_1}}(\xi-\eta,s) \varphi_k(\xi)\overline{\widehat{W^0}}(\xi,s)
\\
& \qquad \times\varphi_{k_2}(\eta)\chi\Big(\frac{|\eta-\rho|}{|\eta+\rho|}\Big)\frac{(\eta-\rho)\cdot(\eta+\rho)}{4|\eta-\rho|^{1/2}}
    \varphi_{k_3}(\eta-\rho)(\widehat{U}-\widehat{\overline{U}})(\eta-\rho,s)\widehat{W^0}(\rho,s)\,ds.
\end{split}
\end{align}
Therefore, we define the multipliers $\gamma=\gamma_{\mu;k,k_1,k_2,k_3;\iota_1}$, 
\begin{equation}\label{tyu1}
\begin{split}
\gamma(\xi,\eta,\rho):=&\varphi_{k_1}(\xi-\eta)\varphi_{k_3}(\eta-\rho)\varphi_k(\xi)\varphi_{k_2}(\eta)\\
 &\times\mu(\xi,\eta)\frac{\varphi_{\leq B}(\Phi_{\iota_1+}(\xi,\eta))}{\Phi_{\iota_1+}(\xi,\eta)}\chi\Big(\frac{|\eta-\rho|}{|\eta+\rho|}\Big)\frac{(\eta-\rho)\cdot(\eta+\rho)}{|\eta-\rho|^{1/2}},
 \end{split}
\end{equation}
and it suffices to prove that, for any $\iota_3\in\{+,-\}$ and any $k,k_1,k_2,k_3$ satisfying \eqref{tyu8},
\begin{equation}\label{tyu2}
\begin{split}
\Big|\Re\int_0^t \sum_{\xi,\eta,\rho\in\Z^2} i\gamma(\xi,\eta,\rho) \widehat{U_{\iota_1}}(\xi-\eta,s)
	\widehat{U_{\iota_3}}(\eta-\rho,s)\widehat{W^0}(\rho,s)\overline{\widehat{W^0}}(\xi,s)\,ds\Big|
\\
	\lesssim  (\oe^2\KK_g^{-1})[\log(1/\oe)]^{-1}2^{-k_1}2^{-k_3}.
\end{split}
\end{equation}

To prove this bound we would like to integrate by parts in time again. For this we want to use Proposition \ref{prop2} to prove lower bounds on the associated four-wave modulation functions. 
To apply this proposition we first need to remove the contribution of the {\it ``trivial'' resonances}, which correspond to 
$\iota_1=-\iota_3$ and $\xi=\rho$ (the condition on the frequencies is already verified since $2^{18k_1}2^{18k_3}\leq\KK_g^{-1} 2^k$, see \eqref{tyu8}). Indeed, the sum in \eqref{tyu2} when $\rho=\xi$ and $\iota_3=-\iota_1$ is
\begin{align*}
\sum_{\xi,\eta \in\Z^2} i\gamma(\xi,\eta,\xi)|\widehat{U_{\iota_1}}(\xi-\eta,s)|^2|\widehat{W^0}(\xi,s)|^2.
\end{align*}
This is a purely imaginary expression, since the symbols $\mu_0$ and $\mu_1^+$ are real-valued. 
Thus, the contribution to the term in \eqref{tyu2} vanishes when the summation is taken over $\xi=\rho$.
  
We are now ready to use Proposition \ref{prop2}. Recall the four-wave modulation functions $\Psi_{\iota_1\iota_3}(\xi,\eta,\rho) = \Lambda(\xi)-\Lambda(\rho)-\iota_1\Lambda(\xi-\eta)-\iota_3\Lambda(\eta-\rho)$ defined in \eqref{gar100}. 
Recall that $|\Lambda(\xi)-\iota_1\Lambda(\xi-\eta)-\Lambda(\eta)| \leq 2^{B+2} \leq 4\KK_g^{-1}2^{-k/2}(2^{k_1}+2^{k_3})^{-2}$, see \eqref{defcut}. 
Applying the inequality \eqref{ml42} and recalling that $g\notin\mathcal{R}$ and $\KK_g^{-1}$ is small, we see that 
\begin{equation*}
|\Lambda(\rho)-\Lambda(\eta)+\iota_3\Lambda(\eta-\rho)|\gtrsim 2^{-k/2}(2^{k_1}+2^{k_3})^{-2}.
\end{equation*}
In particular
\begin{equation}
\label{4phaselb}
|\Phi_{\iota_3+}(\eta,\rho)|\approx | \Psi_{\iota_1\iota_3}(\xi,\eta,\rho) | \gtrsim\KK_g^{-1} 2^{-k/2} (2^{k_1}+2^{k_3})^{-2},
\end{equation} 
for all $(\xi,\eta,\rho)$ in the support of the sum in \eqref{tyu2}. 

Therefore, if we define, for $\ell\in\mathbb{Z}$
\begin{equation}\label{tyu10}
\gamma_\ell(\xi,\eta,\rho):=\gamma(\xi,\eta,\rho)\cdot\varphi_{\ell}(\Psi_{\iota_1\iota_3}(\xi,\eta,\rho)),
\end{equation}
then it suffices to prove that for any $\iota_1,\iota_3\in\{+,-\}$ and any $k,k_1,k_2,k_3$ satisfying \eqref{tyu8},
we have \footnote{Notice
 that we do not need to take the real part of the integral anymore. The main point in taking the real part was to eliminate the contribution of the trivial resonances.}
\begin{equation}\label{tyu12}
\begin{split}
\sum_{\ell\geq B+4}\Big|\int_0^t \sum_{\xi,\eta,\rho\in\Z^2} \gamma_\ell(\xi,\eta,\rho)\widehat{U_{\iota_1}}(\xi-\eta,s)\widehat{U_{\iota_3}}(\eta-\rho,s)\widehat{W^0}(\rho,s)\overline{\widehat{W^0}}(\xi,s)\,ds\Big|\\
\lesssim  (\oe^2\KK_g^{-1})[\log(1/\oe)]^{-1}2^{-k_1}2^{-k_3}.
\end{split}
\end{equation}

{\bf{Step 3.}} We can now integrate by parts in time again. We define 
\begin{equation*}
\mathcal{S}^{\ell}(F,G,H,H') := \sum_{\xi,\eta,\rho\in\Z^2} \frac{\gamma_{\ell}(\xi,\eta,\rho)}{i\Psi_{\iota_1\iota_3}(\xi,\eta,\rho)}
  \widehat{F}(\xi-\eta)\widehat{G}(\eta-\rho)\widehat{H}(\rho)\overline{\widehat{H'}}(\xi).
\end{equation*}
As in  \eqref{gar105} we can write
\begin{equation}
\label{tyu14}
\int_0^t \sum_{\xi,\eta,\rho\in\Z^2} \gamma_\ell(\xi,\eta,\rho)\widehat{U_{\iota_1}}(\xi-\eta,s)\widehat{U_{\iota_3}}(\eta-\rho,s)\widehat{W^0}(\rho,s)\overline{\widehat{W^0}}(\xi,s)\,ds=S_1^\ell+S_2^\ell+S_3^\ell+S_4^\ell+S_5^\ell,
\end{equation}
where
\begin{equation}\label{tyu14.5}
\begin{split}
& S^\ell_1:=-\int_0^t\mathcal{S}^{\ell}((\partial_s+i\Lambda_{\iota_1})U_{\iota_1}(s),U_{\iota_3}(s),W^0(s),W^0(s))\,ds,
\\
& S^\ell_2:=-\int_0^t\mathcal{S}^{\ell}(U_{\iota_1}(s),(\partial_s+i\Lambda_{\iota_3})U_{\iota_3}(s),W^0(s),W^0(s))\,ds,
\\
& S^\ell_3:=-\int_0^t\mathcal{S}^{\ell}(U_{\iota_1}(s),U_{\iota_3}(s),(\partial_s+i\Lambda)W^0(s),W^0(s))\,ds,
\\
& S^\ell_4:=-\int_0^t\mathcal{S}^{\ell}(U_{\iota_1}(s),U_{\iota_3}(s),W^0(s),(\partial_s+i\Lambda)W^0(s))\,ds,
\\
& S^\ell_5:=\mathcal{S}^{\ell}(U_{\iota_1}(t),U_{\iota_3}(t),W^0(t),W^0(t))-\mathcal{S}^{\ell}(U_{\iota_1}(0),U_{\iota_3}(0),W^0(0),W^0(0)).
\end{split}
\end{equation}

Notice that, as a consequence of \eqref{boundMu}, \eqref{gar24}, and \eqref{Factor},
\begin{equation}
\label{tyu20}
\begin{split}
\Big|\frac{\gamma_{\ell}(\xi,\eta,\rho)}{\Psi_{\iota_1\iota_3}(\xi,\eta,\rho)}\Big|&\lesssim \varphi_{k_1}(\xi-\eta)\varphi_{k_3}(\eta-\rho)\varphi_k(\xi)\varphi_{k_2}(\eta)\varphi_\ell(\Psi_{\iota_1\iota_3}(\xi,\eta,\rho))\\
&\times 2^{-\ell}2^k|k|^{3/2}2^{12k_1}\cdot (2^\ell+2^{k_3})2^{4k_3}.
\end{split}
\end{equation}
Therefore, using \eqref{gar10}--\eqref{gar11} and the general bounds \eqref{Fact3},
\begin{equation*}
|S^\ell_1|+|S^\ell_2|+|S^\ell_5|\lesssim T_\varep\oe^5\cdot 2^k|k|^{3/2}(1+2^{-\ell})2^{-6k_3}2^{-6k_1}.
\end{equation*}
This suffices, in view of \eqref{tyu3} and the estimate $2^{-B}\lesssim \KK_g2^{k/2}(2^{2k_1}+2^{2k_3})$. 
Moreover, the contributions of the terms $S^\ell_3$ and $S^\ell_4$ are similar 
(in fact, equivalent, by changes of variables). 
Using \eqref{tyu20} we see that
\begin{equation}
 \label{tyu23rev1}
\begin{split}
\sum_{\ell\geq 0} \Big|\frac{\gamma_{\ell}(\xi,\eta,\rho)}{\Psi_{\iota_1\iota_3}(\xi,\eta,\rho)}\Big|
  & \lesssim \varphi_{k_1}(\xi-\eta)\varphi_{k_3}(\eta-\rho)\varphi_k(\xi)\varphi_{k_2}(\eta)
    \cdot 2^k|k|^{3/2}2^{12k_1}2^{6k_3}.
\end{split}
\end{equation}
Therefore, just \eqref{Fact3} and the $L^2$ estimates \eqref{gar10}--\eqref{gar12.5}
\begin{equation}
\label{tyu23rev2}
\sum_{\ell\geq 0}(|S^\ell_3|+|S^\ell_4|)\lesssim 2^k|k|^{3/2}2^{-4k_1}2^{-4k_3}\cdot T_\e\oe^52^{k}.
\end{equation}
Again, this suffices due to \eqref{tyu3} and the bounds $2^{3k/2}\oe\lesssim 1$. After these reductions, for \eqref{tyu12} it suffices to prove that
\begin{equation}\label{tyu23}
\sum_{\ell\in[B,0]}\Big|\int_0^t\mathcal{S}^{\ell}(U_{\iota_1}(s),U_{\iota_3}(s),
  (\partial_s+i\Lambda)W^0(s),W^0(s))\,ds\Big|\lesssim  (\oe^2\KK_g^{-1})[\log(1/\oe)]^{-1}2^{-k_1}2^{-k_3}.
\end{equation}

{\bf{Step 4.}} 
To prove \eqref{tyu23} we start by expanding $(\partial_s+i\Lambda)W^0(s)$ as in \eqref{gar12}. 
We now claim that the contribution of all the terms except for $-iT_{V_1\cdot\ze}W^0$ can be bounded 
as desired, using \eqref{tyu20}
and $\|(\partial_s+i\Lambda)W^0(s)+iT_{V_1\cdot\ze}W^0(s)\|_{L^2} \lesssim \oe^2 2^{k/2}$. 
To see this, observe that the bound on the symbol in \eqref{tyu20}
gives weaker estimates than \eqref{tyu23rev1} and \eqref{tyu23rev2}
when the sums are taken over $\ell \in (B,0]$;
in particular there is an additional loss of a factor of $2^{-B} \approx 2^{k/2}$ (see \eqref{defcut}). 
Estimating using \eqref{tyu20} we get
\begin{equation*}
\begin{split}
\sum_{\ell\in[B,0]} &\Big|\int_0^t\mathcal{S}^{\ell}(U_{\iota_1}(s),U_{\iota_3}(s),
  (\partial_s+i\Lambda)W^0(s) + iT_{V_1\cdot\ze}W^0(s), W^0(s))\,ds\Big|\\
&\lesssim 2^{-B} \cdot 2^k|k|^{3/2}2^{-4k_1}2^{-4k_3}\cdot T_\e \oe^5 \cdot 2^{k/2}\lesssim 2^{2k} |k|^{3/2}2^{-2k_1}2^{-2k_3}\cdot T_\e \oe^5
\end{split}
\end{equation*}
which, recall \eqref{gar51.5},  suffices for \eqref{tyu23}.

Notice that the bound above is actually tight in terms of powers of $2^k$, and therefore 
the same argument above would not suffice 
when considering the contribution of $T_{V_1\cdot\ze}W^0$ if we just use the basic
bound $\|P_k T_{V_1\cdot\ze}W^0\|_{L^2} \lesssim \oe^2 2^k$. 
We then need to look more in details at the oscillations and the quintic phase \eqref{tyu31}.
More precisely, we write
\begin{equation}\label{tyu30}
\begin{split}
\int_0^t&\mathcal{S}^{\ell}(U_{\iota_1}(s),U_{\iota_3}(s),T_{V_1\cdot\zeta}W^0(s),W^0(s))\,ds
  =\sum_{\iota_4\in\{+,-\}}C\iota_4\int_0^t\sum_{\xi,\eta,\rho,\theta\in\Z^2}
  \frac{\gamma_{\ell}(\xi,\eta,\rho)}{\Psi_{\iota_1\iota_3}(\xi,\eta,\rho)}
  \\
  &\times\chi\Big(\frac{|\rho-\theta|}{|\rho+\theta|}\Big)
   \frac{(\rho-\theta)\cdot(\rho+\theta)}{|\rho-\theta|^{1/2}}
  \what{U_{\iota_1}}(\xi-\eta,s)\what{U_{\iota_3}}(\eta-\rho,s)\what{U_{\iota_4}}(\rho-\theta,s)
  \what{W^0}(\theta,s)\overline{\widehat{W^0}}(\xi,s)\,ds,
\end{split}
\end{equation}
and define the associated five-wave modulation functions
\begin{equation}\label{tyu31}
\begin{split}
\Psi'_{\iota_1\iota_3\iota_4}(\xi,\eta,\rho,\theta)&:=\Lambda(\xi)-\Lambda(\theta)-\iota_1\Lambda(\xi-\eta)-\iota_3\Lambda(\eta-\rho)-\iota_4\Lambda(\rho-\theta)\\
&=\Psi_{\iota_1\iota_3}(\xi,\eta,\rho)+\Phi_{\iota_4+}(\rho,\theta).
\end{split}
\end{equation}
Using these functions we further decompose the expression in the right-hand of \eqref{tyu30} 
into low and high modulations. For $\ell\leq 0$ and $\ast\in\{\leq 6,>7\}$ we define
\begin{equation*}
\begin{split}
& \nu_{\ell,\ast}(\xi,\eta,\rho,\theta):=\frac{\gamma_{\ell}(\xi,\eta,\rho)}{\Psi_{\iota_1\iota_3}(\xi,\eta,\rho)}
  \chi\Big(\frac{|\rho-\theta|}{|\rho+\theta|}\Big)
  \frac{(\rho-\theta)\cdot(\rho+\theta)}{|\rho-\theta|^{1/2}}\varphi_{\ast}
  (\Psi'_{\iota_1\iota_3\iota_4}(\xi,\eta,\rho,\theta)),
\\
& \mathcal{Q}^{\ell,\ast}(F_1,F_2,F_3,H,H')
  :=\sum_{\xi,\eta,\rho,\theta\in\Z^2} \nu_{\ell,\ast}(\xi,\eta,\rho,\theta) 
  \widehat{F_1}(\xi-\eta)\widehat{F_2}(\eta-\rho)
  \widehat{F_3}(\rho-\theta)\widehat{H}(\theta)\overline{\widehat{H'}}(\xi).
\end{split}
\end{equation*}
Notice that, if $\ell\leq 0$,
\begin{equation}\label{tyu31.5}
\big| \nu_{\ell,\leq 6}(\xi,\eta,\rho,\theta) \big| \lesssim \varphi_{k_1}(\xi-\eta)\varphi_{k_3}(\eta-\rho)
	\langle\rho-\theta\rangle^4\varphi_k(\xi)\varphi_{k_2}(\eta)
	\cdot 2^{-\ell}2^{3k/2}|k|^{3/2}2^{12k_1}2^{5k_3},
\end{equation}
see \eqref{tyu20} and \eqref{Factor} (in view of \eqref{tyu31}, 
$|\Phi_{\iota_4+}(\rho,\theta)|\lesssim 1$ in the support of the multiplier). Thus
\begin{equation*}
\sum_{\ell\in[B,0]}\Big|\int_0^t\mathcal{Q}^{\ell,\leq 6}(U_{\iota_1}(s),U_{\iota_3}(s),U_{\iota_4}(s),W^0(s),W^0(s))\,ds\Big|\lesssim T_\varep\oe^52^{-B}2^{3k/2}|k|^{3/2}2^{-6k_1}2^{-6k_3},
\end{equation*}
using a multilinear estimate similar to \eqref{Fact2} and \eqref{Fact3}. This is consistent with the desired bounds \eqref{tyu23} since $2^{-B}2^{3k/2}\lesssim 2^{k/2}\oe^{-1}(2^{2k_1}+2^{2k_3})$ (see \eqref{defcut} and \eqref{tyu3}). 

In view of \eqref{tyu30}, for \eqref{tyu23} it remains to prove that
\begin{equation}\label{tyu35}
\sum_{\ell\in[B,0]}\Big|\int_0^t\mathcal{Q}^{\ell,>7}(U_{\iota_1}(s),U_{\iota_3}(s),U_{\iota_4}(s),W^0(s),W^0(s))\,ds\Big|\lesssim (\oe^2\KK_g^{-1})[\log(1/\oe)]^{-1}2^{-k_1}2^{-k_3}.
\end{equation}
for any $\iota_1,\iota_3,\iota_4\in\{+,-\}$ and any $k,k_1,k_2,k_3$ satisfying \eqref{tyu8}. We integrate by parts for one last time. This leads to an identity similar to \eqref{tyu14}--\eqref{tyu14.5}. Since the quintic modulation is $\gtrsim 1$, integration by parts gains a factor $\approx \oe 2^k$. Also, the multipliers $\nu_{\ell,>7}$ satisfy bounds similar to \eqref{tyu31.5}, with an additional loss of a factor of $2^{k/2}$. Therefore, the left-hand side of \eqref{tyu35} is bounded by $T_\varep\oe^62^{-B}2^{3k}|k|^{3/2}2^{-6k_1}2^{-6k_3}$, which suffices to prove \eqref{tyu35}. This completes the proof of the main bounds \eqref{gar61a}.

\medskip
{\bf{Proof of \eqref{gar61b}.}} 
This is similar to the proof of \eqref{gar61a} (in fact slightly easier because we do not need to use \eqref{Factor} in the first two steps), so we will only provide an outline of the proof. 
First, we may replace $h$ with $(g+|\nabla|^2)^{-1/2}\Re U$, at the expense of acceptable errors (see \eqref{ExpV0}), 
and decompose dyadically the resulting symbol. More precisely, we define 
\begin{equation}\label{tyu40}
\begin{split}
\Sigma'_{1,l} :&= \frac{1}{4}\frac{\Lambda(\zeta)}{|\zeta|} \Big[ \Delta (g+|\nabla|^2)^{-1/2}
	P_l\Re U - \frac{\zeta_i\zeta_j}{|\zeta|^2}\partial_{i}\partial_j(g+|\nabla|^2)^{-1/2}P_l\Re U\Big]
	\\
& - \frac{1}{2} \frac{|\zeta|}{\Lambda(\zeta)} \Lambda^2 (g+|\nabla|^2)^{-1/2}P_l\Re U,
\end{split}
\end{equation} 
and it suffices to prove that, for any $\iota_1\in\{+,-\}$ and $k,k_1,k_2,k_3\in[-4,\infty)$ satisfying \eqref{tyu8},
\begin{align}
\label{tyu48}
\Big| \Re \int_0^t \mathcal{T}^{\leq B}_{\mu} (P_{k_1} U_{\iota_1}(s), P_{k_2}T_{\Sigma'_{1,k_3}}W^0(s), 
	P_k W^0(s)) \,ds\Big| \lesssim \oe^2\KK_g^{-1} [\log(1/\oe)]^{-1}2^{-k_1}2^{-k_3},
\end{align}
where $B$ is defined as in \eqref{defcut}. Compare with \eqref{gar58}. 

Next, we expand the expression in the left-hand side of \eqref{tyu48} as
\begin{align}
\label{tyu49}
\begin{split}
&  \frac{1}{8\pi^2}\sum_{\iota_3\in\{+,-\}}\Re\int_0^t \sum_{\xi,\eta,\rho\in\Z^2} \mu(\xi,\eta)
	\frac{\varphi_{\leq B}(\Phi_{\iota_1+}(\xi,\eta))}{i\Phi_{\iota_1+}(\xi,\eta)}
  	\varphi_{k_1}(\xi-\eta)\widehat{U_{\iota_1}}(\xi-\eta,s) \varphi_k(\xi)\overline{\widehat{W^0}}(\xi,s)
\\
& \qquad \times\varphi_{k_2}(\eta)\chi\Big(\frac{|\eta-\rho|}{|\eta+\rho|}\Big)s(\eta-\rho,(\eta+\rho)/2)
    \varphi_{k_3}(\eta-\rho)\widehat{U_{\iota_3}}(\eta-\rho,s)\widehat{W^0}(\rho,s)\,ds,
\end{split}
\end{align}
where
\begin{equation}\label{tyu50}
s(\theta,\zeta) := -\frac{1}{4}\frac{\Lambda(\zeta)}{|\zeta|}\frac{|\theta|^2}{\sqrt{g+|\theta|^2}} 
	+ \frac{1}{4}\frac{\Lambda(\zeta)}{|\zeta|}\frac{(\zeta\cdot\theta)^2}{|\zeta|^2\sqrt{g+|\theta|^2}}
	- \frac{1}{2} \frac{|\zeta|}{\Lambda(\zeta)} \frac{\Lambda(\theta)^2}{\sqrt{g+|\theta|^2}}.
\end{equation}
Since $\mu$ and $s$ are real-valued, we make the key observation that the contribution of the trivial resonance $\xi=\rho$, 
$\iota_1=-\iota_3$ vanishes. 
Therefore we can use Proposition \ref{prop2}, insert a dyadic decomposition based on the size of the quartic modulation, 
and reduce matters to proving that 
\begin{equation}\label{tyu51}
\begin{split}
\sum_{\ell\geq B+4}\Big|\int_0^t \sum_{\xi,\eta,\rho\in\Z^2} \widetilde{\gamma}_\ell(\xi,\eta,\rho) 
	\widehat{U_{\iota_1}}(\xi-\eta,s)\widehat{U_{\iota_3}}(\eta-\rho,s)
	\widehat{W^0}(\rho,s)\overline{\widehat{W^0}}(\xi,s)\,ds\Big|
\\
\lesssim  (\oe^2\KK_g^{-1})[\log(1/\oe)]^{-1}2^{-k_1}2^{-k_3},
\end{split}
\end{equation}
for any $\iota_1,\iota_3\in\{+,-\}$ and any $k,k_1,k_2,k_3$ satisfying \eqref{tyu8} (compare with \eqref{tyu12}). Here 
\begin{equation}\label{tyu52}
\begin{split}
\widetilde{\gamma}_\ell(\xi,\eta,\rho) & := \varphi_{k_1}(\xi-\eta)\varphi_{k_3}(\eta-\rho)\varphi_k(\xi)\varphi_{k_2}(\eta)
	\mu(\xi,\eta)\frac{\varphi_{\leq B}(\Phi_{\iota_1+}(\xi,\eta))}{\Phi_{\iota_1+}(\xi,\eta)}
	\\
& \times\chi\Big(\frac{|\eta-\rho|}{|\eta+\rho|}\Big)s(\eta-\rho,(\eta+\rho)/2)
	\cdot\varphi_{\ell}(\Psi_{\iota_1\iota_3}(\xi,\eta,\rho)).
\end{split}
\end{equation}
Finally, we examine the definition \eqref{tyu50} and notice that the multipliers $\widetilde{\gamma}_\ell$ satisfy the bounds
\begin{equation*}
\begin{split}
\Big|\frac{\widetilde{\gamma}_{\ell}(\xi,\eta,\rho)}{\Psi_{\iota_1\iota_3}(\xi,\eta,\rho)}\Big|&\lesssim \varphi_{k_1}(\xi-\eta)\varphi_{k_3}(\eta-\rho)\varphi_k(\xi)\varphi_{k_2}(\eta)\varphi_\ell(\Psi_{\iota_1\iota_3}(\xi,\eta,\rho))\cdot 2^{-\ell}2^k|k|^{3/2}2^{12k_1}2^{4k_3}.
\end{split}
\end{equation*}
These are slightly stronger than the bounds \eqref{tyu20}. Thus the argument in Steps 3 and 4 in the proof of \eqref{gar61a} can be applied, essentially with no changes, to prove the desired bounds \eqref{tyu51}. This completes the proof of the lemma.
\end{proof}

\end{document}